\documentclass[12pt,a4paper,final]{article}

\usepackage[height=23.5cm,left=2.5cm,right=2.5cm]{geometry}

\usepackage[utf8]{inputenc}  
\usepackage{epic,eepic}
\usepackage{xcolor,scrtime,graphicx}
\usepackage{amsmath,amsfonts,amsmath}
\usepackage{eucal}
\usepackage{mathrsfs}   
\usepackage[font=small]{caption}
\usepackage{hyperref}
\usepackage{pdfsync}
\usepackage{bef_am_ml}
\def\FG{\bm}
\def\bigset#1#2{\big\{\: #1 \: \big| \: #2 \: \big\} }
\def\Bigset#1#2{\Big\{\; #1 \; \Big| \; #2 \; \Big\} }
\numberwithin{equation}{section}

\usepackage{psfrag,bm}

\def\scrET{\mathscr{E\!T}\!}

\def\mafo{\mathrm}
\def\AC{\mathrm{AC}}
\def\MMM{\calM}
\def\div{\mathop{\mafo{div}}}

\def\FBoltz{F_\mathrm{B}}
\def\trcos{{\cos_\pi}}
\renewenvironment{cases}{\left\lbrace \ba{cl}}{ \ea \right. }

\def\REVER{\mathop{\mbox{\Large\raisebox{0.1em}{$\rightharpoonup$}%
        \hspace*{-1em}\raisebox{-0.1em}{$\leftharpoondown$}}}}

\def\mfA{\mathfrak A}
\def\mfC{\mathfrak C}

\def\mfL{\mathfrak L}
\def\mfN{\mathfrak N}
\def\mfP{\mathfrak P}
\def\mfR{\mathfrak R}
\newcommand{\mfS}{\mathfrak S}

\newcommand{\Prob}{\calP}

\newcommand{\Length}{\mathop{\mafo{Length}}\nolimits}

\newcommand{\bigslant}[2]{{\raisebox{.2em}{$#1$}\left/\raisebox{-.2em}{$#2$}\right.}} 

\def\mfC{\mathfrak C}
\def\CC{\mfC_\Omega}    
\def\TT{{\mathfrak o}}  

\newcommand{\msHK}{\mathsf{H\hspace{-0.22em}K}} 
\newcommand{\msET}{\mathsf{E\hspace{-0.22em}T}} 
\newcommand{\res}{\mathop{\hbox{\vrule height 7pt width .5pt depth 0pt
    \vrule height .5pt width 6pt depth 0pt}}\nolimits}

\begin{document}

\title{Optimal transport in competition with reaction:\\ 
the Hellinger-Kantorovich distance and geodesic curves}

\author{Matthias Liero\footnote{Weierstra\ss-Institut f\"ur Angewandte Analysis und Stochastik, Berlin.}, 
Alexander Mielke$^{*,}$\footnote{Humboldt-Universit\"at zu Berlin.}, 
and Giuseppe Savar\'e\footnote{Universit\`a di Pavia.}} 

\date{September 24, 2015}

\maketitle

\begin{abstract}\noindent
  We discuss a new notion of distance on the space of finite and
  nonnegative measures  on $\Omega \subset \R^d$, which we call
  Hellinger--Kantorovich distance. It  can be seen as  an
  inf-convolution  of the well-known Kantorovich--Wasserstein
  distance  and the Hellinger-Kakutani distance.   The new
  distance is based on a dynamical formulation given by an Onsager
  operator that is the sum of a Wasserstein diffusion part and an
  additional reaction part describing the generation and absorption of
  mass.

  We present a full characterization of the distance and  some of
   its properties.  In particular,  the distance can be
  equivalently described by an optimal transport problem on the cone
  space over the underlying  space $\Omega$.  We give a
  construction of geodesic curves and discuss examples and their
  general  properties.
\end{abstract}

\section{Introduction}
\label{s:Intro}
Starting from the pioneering works \cite{JoKiOt97FEFP} and
\cite{JoKiOt98VFFP}, the reinterpretation of certain scalar diffusion
equations as so-called Wasserstein gradient flows led to new analytic
tools and concepts and gave deeper insight into diffusion problems,
see e.g.\ \cite{Vill09OTON} and \cite{ADPZ11LDPW}. In particular, in
connection with suitable convexity properties of the driving
functional the abstract theory of gradient flows in metric space
developed in \cite{AmGiSa05GFMS} provides a sound and comprehensive
geometric framework for these evolution equations.

The recent reformulation of classes of reaction-diffusion systems
as gradient systems, see \cite{Miel11GSRD,Miel13TMER,LieMie13GSGC},
raises the question whether the abstract metric theory can be  also developed 
for this wider class of problems. 

Following \cite{Miel11GSRD} we understand a gradient system as a
triple $(\bfX,\calF,\bbK)$ consisting of a state space $\bfX$, a
driving functional $\calF$, and an Onsager operator $\bbK$. The latter
means that $\bbK$ is a state-dependent, symmetric, and positive
semidefinite linear operator.  In many cases the Onsager operator
$\bbK$ induces a dissipation distance $\msD_\bbK$ on the state space
$\bfX$ by minimizing an action functional over all curves connecting
two states. Now, the development of a metric theory rests upon the
ability to characterize this distance and its properties.

This paper together with the companion paper
\cite{LiMiSa14?Theory} provides rigorous characterization of such a 
dissipation distance. It is based on the simple Onsager operator
\[
\bbK_{\alpha,\beta}(u)\xi= - \alpha \div(u\nabla \xi) + \beta u \xi,
\]
where $\alpha,\beta\geq 0$ are fixed parameters.  Obviously, the
Onsager operator $\bbK_{\alpha,\beta}=\alpha \bbK_\mafo{Wass} + \beta
\bbK_\text{c-a}$ is a sum of a Wasserstein part for diffusion and a
creation-annihilation part, which is the simplest case of a
reaction term. For the latter part it is not difficult to develop
a corresponding analog to the Wasserstein distance $\msW$. For this,
we simply note that $\bbK_\text{c-a}$ is the inverse of a metric
tensor, such that the formal associated Riemannian structure is given
by $v \mapsto \int_\Omega v^2/(\beta u)\dd x $. Thus, it is easy to
see that the Riemannian distance induced by $\bbK_{0,\beta}$ is a
multiple of the Hellinger--Kakutani distance $\msH$, see 
\cite{Hell09NBTQ,Kaku48EIPM} and \cite{Schi96SPLD}, where also
the correct function spaces are discussed. In particular, for measures
of the form $\mu_j=f_j\dd x$ we obtain
\[
\msD_{0,\beta}(\mu_0,\mu_1) = \frac{2}{\sqrt{\beta}}
\msH(\mu_0,\mu_1)
\text{ with } \msH(\mu_0,\mu_1)^2 = \int_\Omega \Big( 
\sqrt{f_0} - \sqrt{f_1}\Big)^2  \dd x .
\]
Of course, this distance generalizes to the space of finite,
nonnegative Borel measures, denoted by $\MMM(\Omega)$, see Section
\ref{ss:ScaRDE}.  In the following we shall always assume that the
domain $\Omega\subset\R^d$ is convex and compact.

We will show that $\bbK_{\alpha,\beta}$ generates a proper distance
$\msD_{\alpha,\beta}$ on $\MMM(\Omega)$, which is formally given
in a generalized Benamou--Brenier formulation (see \cite{BenBre00CFMS}), 
i.e., by minimizing over 
all sufficiently smooth curves $s\mapsto \mu(s)$ connecting measures $\mu_0$
and $\mu_1$, viz.
\[
\begin{aligned}
\msD_{\alpha,\beta}(\mu_0,\mu_1)^2=\inf\Big\{\int_0^1
\int_\Omega\big[\alpha|\nabla\xi|^2+\beta\xi^2\big]\dd \mu(s)\,\dd s\,
\Big|\qquad\qquad\\
\tfrac{\dd}{\dd s}\mu +\alpha\div(\mu\nabla\xi)=\beta\mu\xi,~\mu_0
 \overset{\mu}{\leadsto} \mu_1\Big\}.
\end{aligned}
\]
This characterization also works for reaction-diffusion systems,
see \cite[Sect.\,2(e)]{LieMie13GSGC}. 
 
We call $\msD_{\alpha,\beta}$ the \emph{Hellinger--Kantorovich distance},
since it can be understood as
an inf-convolution (weighted by $\alpha$ and $\beta$) of the
Kantorovich--Wasserstein distance $\msW$ and the Hellinger distance $\msH$. 
In particular, geodesic curves for the distance $\msD_{\alpha,\beta}$ 
will optimize the usage of transport against the
usage of creation or annihilation. As an outcome of our theory we will
find that transport never occurs over distances longer than
$\pi\sqrt{\alpha/\beta}$.  

For the rest of this introduction we will use the special choice
$\alpha=1$ and $\beta=4$, which simplifies the notation considerably. 

To give a full characterization of $\msD_{1,4}$, we go a detour 
which will highlight the underlying geometry of the
distance much better. Motivated by an explicit formula for the 
distance between two Dirac measures, we introduce the cone space $\CC$ over
$\Omega$ and define the Hellinger--Kantorovich distance
$\msHK(\mu_0,\mu_1)$ by lifting the measures $\mu_j$ to measures
$\lambda_j$ on the cone and then minimizing the Wasserstein distance
$\msW_\mfC$ induced by a suitable cone distance on $\CC$. It is then easy to
show that $\msHK$ is indeed a geodesic distance, since $\msW_\mfC$ is a
geodesic distance. It is  the purpose of Section \ref{se:equiDyn}  
to show that $\msD_{1,4}$ indeed equals $\msHK$. 

For this proof, we will rely on a third characterization of the
Hellinger--Kantorovich distance,  which is given in terms of the
entropy-transport functional for calibration measures $\eta\in\MMM(\Omega{\times}\Omega)$
given via
\[
\scrET_{1,4}(\eta;\mu_0,\mu_1):=  \int_\Omega \FBoltz\Big( 
  \frac{\rmd \eta_0}{\rmd\mu_0}\Big) \dd\mu_0 + \int_\Omega \FBoltz\Big( 
  \frac{\rmd \eta_1}{\rmd\mu_1}\Big) \dd\mu_1 +
\int_{\Omega\ti\Omega} \msc_{1,4}(|x_0{-}x_1|) \dd \eta,
\]
where $\FBoltz(z)=z\log z - z +1$, $\eta_i = \Pi^i_\#\eta$ denote the usual marginals, 
and the cost function $\msc_{1,4}$ is given by  
\[
\msc_{1,4}(L):=\left\{ \ba{cl}
- 2\log\big(\cos  L \big) & \text{for } L < \pi/2, \\ \infty & \text{for } 
 L \geq \pi/2. \ea \right.
\]   
Since $\scrET_{1,4}(\cdot;\mu_0,\mu_1)$ is  convex, it is easy to find 
minimizers, see  \cite{LiMiSa14?Theory} for more details and the proof
that $\msHK(\mu_0,\mu_1)^2=\min\set{\scrET_{1,4}(\eta;\mu_0,\mu_1)}{
  \eta \in \MMM(\Omega\ti \Omega) }$.

To be more specific, we return to the question of computing the distance between
to Dirac masses $\mu_j=a_j\delta_{y_j}$ with $y_j\in \Omega$ and
$a_j\geq 0$. Looking at connecting curves of the form
$\mu(s)=a(s)\delta_{x(s)} $ we can indeed minimize the length of these
one-mass point curves (1mp) and find the result 
\begin{equation}
  \label{eq:Int.1mp}
  \msD^\text{1mp}(a_0\delta_{y_0}, a_1\delta_{y_1})^2 = \left\{\ba{cl}
a_0+a_1 - 2\sqrt{a_0a_1}\cos(|y_1{-}y_0|) 
  &\text{for } |y_1{-}y_0| \leq \pi,\\[0.5em] a_0+a_1+ 2\sqrt{a_0a_1}& 
\text{for } |y_1{-}y_0| \geq \pi. \ea \right.
\end{equation}
In fact, a minimizer exists only for $|y_1{-}y_0|<\pi$ where for
$|y_1{-}y_0|\geq \pi$ the value $\msD^\text{1mp}(a_0\delta_{y_0},
a_1\delta_{y_1})$ is an infimum only. However, it will turn out that 
these curves are only optimal for $|y_1{-}y_0|
\leq \pi/2$, while for $|y_1{-}y_0| > \pi/2$, the two-mass point curve
$\mu(s)=(1{-}s)^2a_0\delta_{y_0} + s^2a_1\delta_{y_1}$ is shorter,
since its squared length is $a_0+a_1$. Thus, creation and annihilation
is better than transport in this case. 

Moreover, the formula in \eqref{eq:Int.1mp} suggest to introduce a
cone distance $\msd_\mfC$ on the cone $\CC$ over $\Omega$ given by 
the elements $[x,r]$ for $r>0$ and the tip $\TT$ which is an
identification of $\set{[x,0]}{x\in \Omega}$. The cone distance is 
defined as
\[
\msd_\mfC ([x_0,r_0],[x_1,r_1])^2:= r_0^2+r_1^2 - 2r_0r_1
  \trcos(|x_{1}{-}x_0|) 
 \text{ with }\cos_b a=\cos\big(\min\{|a|,b\}\big),
\]  
see \cite[Sect.\,3.6.2]{BuBuIv01CMG}. This distance is again a
geodesic distance and we can define the associated Wasserstein
distance $\msW_\mfC$, see Section \ref{sss:WassDistReservoir}.  

Based on this observation we can now lift measures $\mu$ on $\Omega$
to measures $\lambda$ on $\CC$ such that $\mu= \mfP\lambda$, where the 
projection $\mfP : \MMM_2(\CC) \to \MMM(\Omega)$ is defined via 
\[
\int_\Omega \phi(x)\dd (\mfP\lambda)(x) =
\int_{\CC} r^2 \phi(x)\dd\lambda([x,r]) \quad \text{for all }
 \phi\in \rmC^0(\Omega).
\]
Now, the first definition of the Hellinger--Kantorovich distance is  
\begin{equation}
  \label{eq:Int.DHKbyLift}
  \msHK(\mu_0,\mu_1)= \min\Bigset{ \msW_\mfC (\lambda_0,\lambda_1)
}{ \mfP\lambda_0 = \mu_0,\ \mfP\lambda_1 = \mu_1 }.
\end{equation}
To further analyze this construction,  one needs to study the optimality
conditions for the lifts, which can be done by exploiting the characterization via
$\scrET_{1,4}$, see Theorem~\ref{th:properties} and
Section~\ref{sss:SpecialLifts}, where the crucial duality theory
is taken from \cite{LiMiSa14?Theory}.  

In Section \ref{se:equiDyn} we finally show the identity $\msD_{1,4}=
\msHK$ by a full characterization of all absolutely continuous curves
with respect to the distance $\msHK$, see Theorem
\ref{thm:timeLower}. This is done by lifting curves in $\MMM(\Omega)$
to curves in $\calM_2(\CC)$ and using a characterization of
absolutely continuous curves with respect to $\msW_\mfC$ which can be
found in \cite{Lisi07CACC}. In Corollary \ref{co:ConeGeodesic} we
obtain the important result that \emph{all} geodesics curves in
$(\calM(\Omega),\msHK)$ are obtained as projections 
\[
\mu(s) = \mfP \lambda(s), \text{ where }\lambda:[0,1]\to \calM_2(\CC) 
\]
is a geodesic curve in $(\calM_2(\CC),\msW_\mfC)$ connecting optimal
lifts $\lambda_0$ and $\lambda_1$ in
\eqref{eq:Int.DHKbyLift}. Throughout this work, the notion ``geodesic
curve'', or shortly ``geodesic'', means constant-speed minimal
geodesic, viz.
\[
\msHK(\mu(s),\mu(t))= |s{-}t|\,\msHK(\mu(0),\mu(1)) \text{ for all
}s,t\in [0,1].
\]  

Section \ref{se:Geodesics} is devoted to various examples for geodesic
curves, which are obtained by doing optimal lifts to the cone space
$\CC$ and then constructing geodesic curves for the Wasserstein distance
$\msW_\mfC$ and projecting them down. Since the geodesic curves
on the cone $\CC$ are explicit, this provides an explicit formula for
geodesic curves $\mu:[0,1]\to \MMM(\Omega)$, as soon as the lifts are
specified. 

In particular, using this explicit construction we show that the total mass
$m(s)=\mu(s)(\Omega)$ along geodesic curves is
2-convex and 2-concave since  we have the identity 
\begin{equation}\label{eq:Int.MassIdent}
m(s) = (1{-}s) m(0) + sm(1) -s(1{-}s)\msHK(\mu_0,\mu_1)^2.
\end{equation}
We discuss geodesic $\Lambda$-convexity of some functionals, in
particular, we show that the linear functional
$\calF(\mu)=\int_\Omega \Phi(x)\dd \mu(x)$  is geodesically
$\Lambda$-convex if and only if
the function $[x,r]\mapsto r^2\Phi(x)$ is geodesically
$\Lambda$-convex in $(\CC,\msd_\mfC)$.

It is also worth to note that the unique geodesic 
connecting $\mu_1$ to the null measure
$\mu_0\equiv 0$, which has the lifts $\alpha \delta_{\TT}$ for
$\alpha \geq 0$, is done by the unique Hellinger geodesic 
\[
\mu^\rmH(s) = s^2 \mu_1.
\]
This simple observation immediately shows that the logarithmic entropy 
given by $\calE(\mu)= \int_\Omega \FBoltz(u(x))\dd x $ for $\mu = u \dd x$ is
not geodesically $\Lambda$-convex, since 
\[
\calE(s^2\mu)= s^2 \calE(\mu) + s^2\log(s^2) \mu(\Omega)
+1{-}s^2.
\]

In Section \ref{ss:AllDiracM} we reconsider our standard example of
the geodesic connections of Dirac masses $\mu_j=a_j\delta_{y_j}$. It
turns out that in the critical case $|y_0{-}y_1|=\pi/2$ there is an
infinite dimensional convex set of geodesic curves that can be 
constructed by showing that there are many optimal lifts to the
cone space $\CC$.

Section \ref{ss:Dilations} provides a generalization of the classical
dilation of measures in the Wasserstein case. For the
Hellinger--Kantorovich distance there is a similar dilation where the
mass inside the ball $\set{x}{|x{-}y_0|< \pi/2}$ is radially transported
and partly annihilated into the point $y_0$ while the mass at larger
distance is simply annihilated according to the Hellinger distance.

In Section \ref{ss:CharFunct} we show how the transport of two
characteristic functions occurs in the Hellinger--Kantorovich
case. While the too distant parts are simply annihilated or created
according to the Hellinger metric the parts that are close enough lead
to a continuous transition, see Figure \ref{fig:CharFunct}. 

In Section \ref{ss:CharGeodConn} we show that the
Hellinger-Kantorovich geodesic between two measures $\mu_0$ and
$\mu_1$ is unique if one of the two measures is absolutely continuous
with respect to the Lebesgue measure. Finally, Section
\ref{ss:NotSemiconcave} shows that $\msHK$ is not semiconcave in
$\calM(\Omega)$ if $\Omega\subset \R^d$ has dimension two or higher,
which is in sharp contrast to the Wasserstein distance, see
\cite[Def.\,12.3.1]{AmGiSa05GFMS}.  \medskip

This work, together with its companion paper \cite{LiMiSa14?Theory}, will form the basis
of subsequent work where we will explore the metric properties of the
the space $(\MMM(\Omega),\msHK)$ and study gradient systems on this
space. In particular, in the spirit of \cite{LieMie13GSGC}, 
we aim to establish a metric theory for scalar reaction-diffusion
equations of the form
\[
\dot u = - \bbK_{\alpha,\beta}(u) \delta\calF(u)= \div\big( \alpha u \nabla(\delta
\calF(u)\big) - \beta u \delta \calF(u),
\]
where $\delta\calF$ denotes a variational derivative. 

\paragraph*{Note during final preparation.} The earliest parts of the 
work presented here were first presented at the ERC Workshop on Optimal 
Transportation and Applications in Pisa in 2012. Since then the authors 
developed the theory continuously further and presented results at different 
workshops and seminars. We refer to \cite[Sect.\,A]{LiMiSa14?Theory}
for some remarks concerning the chronological development. In June 2015
they became aware of the parallel work
\cite{KoMoVo15?NOTD}. Moreover, in mid August 2015 we became aware of  
\cite{CPSV15?IDOT,CPSV15?UOTG}. So far, these independent works are
not reflected in the present version of this manuscript.

\section{Gradient structures for reaction-diffusion equations}
\label{s:GradStruct}

\subsection{General philosophy for gradient systems}
\label{ss:AbstrGS}

We call a triple $(\bfX,\calF,\Psi)$ a gradient system in the
differentiable sense, if $\bfX$ is a Banach space containing  the
states $u$, if the functional 
$\calF:\bfX\to \R_\infty:=\R\cup\{\infty\}$ has a Fr\'echet
subdifferential $\rmD \calF(u)\in \bfX^*$ on a suitable subset of
$\bfX$, and if $\Psi$ is a dissipation potential. The latter  means that 
$\Psi(u,\cdot):\bfX\to [0,\infty]$ is a lower semicontinuous and
convex functional with $\Psi(u,0)=0$. Denoting by
$\Psi^*(u,\cdot):\bfX^* \to [0,\infty]$ the
Legendre-Fenchel transform $\Psi^*(u,\xi)=\sup\set{ \langle
  \xi,v\rangle - \Psi(u,v)}{ v\in \bfX}$, the gradient evolution is
given via 
\begin{equation}
  \label{eq:genGS}
  \dot u \in  \rmD_\xi\Psi^*(u,-\rmD\calF(u)) \quad \text{or equivalently}
\quad  0=\rmD_{\dot u} \Psi(u,\dot u) +\rmD\calF(u).
\end{equation}
For simplicity, we assume that the Fr\'echet subdifferential $\rmD\calF$
and the convex subdifferentials $\rmD_{\dot u} \Psi$ and
$\rmD_\xi\Psi^*$ are single-valued, but the set-valued case can be
treated similarly by the standard generalizations. 

If the map $v\mapsto\Psi(u,v)$ is quadratic, we call the above system 
a classical gradient system while otherwise we speak of generalized gradient systems.
In the classical case we can write 
\[
\Psi(u,v)=\frac12\langle \bbG(u)v,v\rangle \quad \text{and} \quad 
\Psi^*(u,\xi)=\frac12\langle \xi,\bbK(u)\xi\rangle,
\]
where $\bbG(u):\bfX\to\bfX^*$ and $\bbK(u):\bfX^*\to\bfX$ are symmetric and positive (semi)definite
operators. Since $\Psi$ and $\Psi^*$ form a dual pair we have $\bbG(u)^{-1}=
\bbK(u)$ and $\bbK(u)^{-1}=\bbG(u)$ if we interpret these
identities in the sense of quadratic forms. We call $\bbG$ the
Riemannian operator, as it generalizes the Riemannian tensor on
finite-dimensional manifolds, while we call $\bbK$ the Onsager
operator because of Onsager's fundamental contributions in
justifying gradient systems via his reciprocal relations
$\bbK(u)=\bbK(u)^*$, cf.\ \cite[Eqn.\,(1.11)]{Onsa31RRIP} or
\cite[Eqs.\,(2-1)--(2-4)]{OnsMac53FIP}. Thus, for classical
gradient systems the general form \eqref{eq:genGS} specializes to 
\begin{equation}
  \label{eq:classGS}
  \dot u = - \bbK(u) \rmD\calF(u) \quad \text{or equivalently} \quad 
 \bbG(u) \dot u = -\rmD\calF(u). 
\end{equation}
We emphasize that $\bbK(u)$ maps (a subspace of) $\bfX^*$ to $\bfX$,
so generalized thermodynamic driving forces are mapped to
rates. Similarly $\bbG(u)$ maps rates to viscous dissipative forces,
which have to balance the potential restoring force $-\rmD\calF(u)$. 

Our work follows the same philosophy as in
\cite{JoKiOt98VFFP,Otto01GDEE}: Even though the above gradient
structure is only formal, it may generate a new dissipation distance,
which can be made rigorous such that finally the gradient structure
can be considered as a mathematically sound metric gradient flow as
discussed in \cite{AmGiSa05GFMS}. For this one introduces the
dissipation distance associated with the dissipation potential $\Psi$,
which is defined via 
\[
\msD_\bbK(u_0,u_1)^2:= \inf\Bigset{ \int_0^1 \langle \bbG(u)\dot
  u,\dot u\rangle \dd s}{ u\in \rmH^1([0,1];\bfX), u(j)=u_j}.
\]

\subsection{Dissipation distances for reaction-diffusion systems}
\label{ss:DissDistRDS}
It was shown in \cite{Miel11GSRD} that certain reaction-diffusion systems
admit a formal gradient structure, which is given by an Onsager operator
$\bbK$ and a driving functional $\calF$ of the form
\[
\bbK(\bfc)\bfxi = -\div(\bbM(\bfc)\nabla\bfxi) + \bbH(\bfc)\bfxi,\qquad
\calF(\bfc) = \int_\Omega \sum_{i=1}^IF(c_i)\dd x,
\]
where $\bfc=(c_i)_{i=1,\ldots,I}$ is the vector of non-negative
concentrations of the species $X_i$, $i=1,\ldots,I$, and 
$\bbM(\bfc)$ and $\bbH(\bfc)$ are a symmetric and positive definite mobility tensor
and a reaction matrix, respectively. With the diffusion tensor $\bbD(\bfc) = \bbM(\bfc)\rmD^2\calF(\bfc)$
and the reaction term $\bfR(\bfc)=\bbH(\bfc)\rmD\calF(\bfc)$ the generated gradient-flow equation
reads
\[
\dot\bfc= -\bbK(\bfc) \rmD \calF(\bfc) = \div\big(\bbD(\bfc)\nabla\bfc\big) -\bfR(\bfc).
\]

As in the theory of the Kantorovich--Wasserstein distance (cf.\
\cite{Otto98DLPF,JoKiOt98VFFP,Otto01GDEE,Vill09OTON}) the operator
$\bbK(\bfc)$ can be seen as the inverse of a metric tensor $\bbG(\bfc)$ that
gives rise to a geodesic distance between two densities $\bfc_0,\,\bfc_1\in
\rmL^1( \Omega;{[0,\infty[}^I)$ defined abstractly via 
\begin{equation}
  \label{eq:GeodDist0}
  \msD_\bbK(\bfc_0,\bfc_1)^2 := \inf \Bigset{ \int_0^1 \langle
    \bbK(\bfc(t))^{-1} \dot \bfc(t), \dot\bfc(t)\rangle\dd t }{ \bfc_0
    \overset{\bfc}{\leadsto} \bfc_1}. 
\end{equation}
Here ``$\bfc_0\overset{\bfc}{\leadsto} \bfc_1$'' means that $t\mapsto\bfc(t)$ is
a sufficiently smooth curve with $\bfc(0)=\bfc_0$ and
$\bfc(1)=\bfc_1$. 
 
Since in general the inversion of $\bbK$ is difficult or even not
well-defined, it is better to use the following formulation in terms
of the dual variable $\bfxi(s)=\bbK(\bfc(s))\dot\bfc(s)$, namely 
\begin{equation}
  \label{eq:GeodDist1}
  \msD_\bbK(\bfc_0,\bfc_1)^2 := \inf \Bigset{ \int_0^1 \langle
   \bfxi(t), \bbK(\bfc(t)) \bfxi(t)\rangle\dd t }{ \dot\bfc=\bbK(\bfc)\bfxi,\ \bfc_0
    \overset{\bfc}{\leadsto} \bfc_1}. 
\end{equation}
In our case of reaction-diffusion operators we can make this even more
explicit, namely 
\begin{equation}
  \label{eq:GeodDist2}
\begin{aligned}  \msD_\bbK(\bfc_0,\bfc_1)^2 := \inf \Big\{\; \int_0^1 \int_\Omega
   \nabla\bfxi\mdot \wt\bbM(\bfc)\nabla \bfxi +
   \bfxi\cdot\bbH(\bfc)\bfxi \dd x \dd t \;\Big| \qquad \\
 \dot\bfc
   =-\div(\wt\bbM(\bfc)\nabla\bfxi\big)+\bbH(\bfc) \bfxi,\ \bfc_0
    \overset{\bfc}{\leadsto}\bfc_1 \;\Big\} ...
\end{aligned}
\end{equation}
Finally, we can use the Benamou-Brenier argument \cite{BenBre00CFMS}
to find the following characterization (cf.\
\cite[Sect.\,2.5]{LieMie13GSGC}): 

\begin{proposition}\label{p:equiDist}
We have the equivalence
\begin{equation}
  \label{eq:GeodDist3}
\begin{aligned}
  \msD_\bbK(\bfc_0,\bfc_1)^2&= \inf \Big\{\; \int_0^1 \int_\Omega
   \bfXi\mdot \wt\bbM(\bfc) \bfXi +
   \bfxi\cdot\bbH(\bfc)\bfxi \dd x \,\dd t \;\Big| \qquad\qquad \\
 &\hspace{12em}\dot\bfc =-\div\big(\wt\bbM(\bfc)\bfXi\big)+\bbH(\bfc) \bfxi,\ \bfc_0
    \overset{\bfc}{\leadsto} \bfc_1 \;\Big\} \\
 &= \inf \Big\{\; \int_0^1 \int_\Omega
   \bfP\mdot \wt\bbM(\bfc)^{-1} \bfP +
   \bfs\cdot\bbH(\bfc)^{-1}\bfs\, \dd x \,\dd t \;\Big| \qquad\qquad \\
 &\hspace{12em}\dot\bfc
   =-\div(\bfP)+\bfs,\ \bfc_0
    \overset{\bfc}{\leadsto} \bfc_1 \;\Big\} 
\end{aligned}
\end{equation}
where $\bfxi(t,x)=\bbH(\bfc(t,x))^{-1}\bfs(t,x)\in \R^I$ and 
$\bfXi(t,x)=\wt\bbM(\bfc(t,x))^{-1}\bfP(t,x)\in \R^{I\ti d}$. 
\end{proposition}
\begin{proof} Clearly, the right-hand side in \eqref{eq:GeodDist3} gives a
  value that is smaller or equal than that in \eqref{eq:GeodDist2},
  because we have dropped the constraint $\bfXi=\nabla \bfxi$. 

  To show that the two definitions give the same value, we have to
  show that for minimizers (do they exist), the constraint
  $\bfXi=\nabla \bfxi$ is automatically satisfied. For this we use
  that $\bfxi$ and $\bfXi$ are related by the continuity equation
  $\dot\bfc =-\div(\wt\bbM(\bfc)\bfXi\big)+\bbH(\bfc) \bfxi$. 

  Keeping $\bfc$ fixed (and sufficiently smooth) we can minimize 
  the integral in \eqref{eq:GeodDist3} with
  respect to $\bfxi$ and $\bfXi$, which is a quadratic functional with
  an affine constraint. Hence, we can apply the Lagrange multiplier
  rule to 
\[
\calL(\bfXi,\bfxi,\bflambda) = \int_0^1\int_\Omega \bfXi\mdot \wt\bbM(\bfc) \bfXi +
   \bfxi\cdot\bbH(\bfc)\bfxi+ \bflambda\vdot\big(\dot\bfc
   +\div(\wt\bbM(\bfc)\bfXi\big)-\bbH(\bfc) \bfxi\big)
\dd x \dd t
\]
to obtain the Euler-Lagrange equations 
\begin{align*}
0=2\wt \bbM\bfXi -\wt\bbM \nabla \bflambda ,\quad 
0= 2\bbH\bfxi - \bbH \bflambda, \quad 0=\dot\bfc
   +\div(\wt\bbM(\bfc)\bfXi\big)-\bbH(\bfc).  
\end{align*}
>From the first two equations we conclude $\bfXi=\frac12\nabla\bflambda
= \nabla \bfxi$, which is the desired result. 
\end{proof}

\subsection{Scalar reaction-diffusion equations}
\label{ss:ScaRDE}

On $\Omega \subset \R^d$, which is a bounded and
convex domain, we consider scalar equations of the form 
\[
\dot u = \div( a(u) \nabla u) - f(u) \quad \text{in }\Omega, \qquad
\nabla u \cdot \nu=0 \text{ on }\pl\Omega, 
\]
where we assume that $f$ changes sign, such that $f(u)(u{-}1)$ is
positive for $u\in {]0,1[}\cup {]1,\infty[}$. We want to write the
above equation as a gradient system $(\bfX,\calF,\bbK)$ with
$\bfX=\rmL^1(\Omega)$, 
\[
\calF(u)=\int_\Omega F(u(x)) \dd x, \quad \text{ and } \quad 
\bbK(u) \xi = - \div\big( \mu(u)\nabla \xi\big) + k(u) \xi \text{ with
  } \nabla \xi\cdot \nu=0,
\]
where $F:\R\to [0,\infty]$ is a strictly convex function with
$F(u)=\infty$ for $u<0$ and $F(1)=0$. Moreover, we assume $\mu(u),
k(u)\geq 0$, such that the dual dissipation potential is the
nonnegative quadratic form 
\[
\Psi^*(u,\xi)=\int_\Omega \frac{\mu(u(x))}2 |\nabla \xi(x)|^2 +
\frac{k(u(x))}2 \xi(x)^2 \dd x. 
\]

Using $\rmD \calF(u)= F'(u(x))$ we obtain 
\[
\bbK(u)\rmD\calF(u)= -\div\big(\mu(u)F''(u)\nabla u\big) + k(u)F'(u).
\]
Hence we see that we obtain the above reaction-diffusion equation, if
we choose $F,\:\mu$, and $k$ such that the relations 
\[
a(u)=\mu(u)F''(u) \quad \text{ and } \quad k(u)F'(u)=f(u). 
\]

There are several canonical choices. Quite often one is interested in
the case $a\equiv 1$, which gives rise to the simple semilinear
equation $\dot u = \Delta u -f(u)$. To realize this one chooses
$\mu(u)=1/F''(u)$. This is particularly interesting in the case of the
logarithmic entropy where $F(u)=\FBoltz(u):= u \log u -u
+1$. Then, $\mu(u)=1/F''(u)=u$ and we obtain the Wasserstein operator
$\xi \mapsto - \div(u \nabla \xi)$ 
for the diffusion part. 

For the reaction part one simply chooses $k(u)=f(u)/F'(u)$, which is
positive, since $f(u)$ and $F'(u)$ change the sign at $u=1$. For
$F=\FBoltz$ and the equation 
\[
\dot u = \Delta u -\kappa( u^\beta- u^\alpha) \quad \text{with }0\leq
\alpha < \beta
\]
we obtain $k(u)= \kappa (u^\beta-u^\alpha)/\log u=
\kappa(\beta{-}\alpha) \Lambda(u^\alpha,u^\beta)$, where
the logarithmic mean is given via 
$\Lambda(a,b)=(a{-}b)/(\log a {-}\log b)$.
This equation models the evolution of
a single diffusing species undergoing the
creation-annihilation reaction
\[
\alpha X \REVER_{\kappa}^{\kappa} \beta X.
\]

The simplest example of a reaction-diffusion distance is the
Hellinger--Kantorovich distance $\msD_{\alpha,\beta}$ studied 
in Section \ref{s:HellWass} in great detail. It is defined via 
the scalar Onsager operator 
\begin{equation}\label{eq:HK-Operator}
\bbK_{\alpha,\beta}(c) \xi := - \alpha \div( c \,\nabla \xi) 
+ \beta\, c\, \xi,
\end{equation}
where $\alpha,\beta$ are nonnegative parameters. The special
property of this operator is that it is linear in the variable
$c$. This will allow us to do explicit calculations for the
corresponding dissipation distance $\msD_{\alpha,\beta}:=\msD_{\bbK_{\alpha,\beta}}$. In particular, the associated 
distance is defined for all pairs of (nonnegative and finite) measures 
$\mu_0,\mu_1 \in\MMM(\Omega)$, not just for probability measures 
$\Prob(\Omega)$. In fact, we will see that for $\beta>0$ the geodesic 
curves connecting to different probability measures will have 
mass less than one for all arclength parameters $s\in {]0,1[}$. 
   
For $\beta =0$ we obtain the scaled Kantorovich--Wasserstein 
distance, namely 
\[
\msD_{\alpha,0}(\mu_0,\mu_1)= \left\{\ba{cl} \ds\frac1{\sqrt{\alpha|\mu_0|}}
    \msW\Big( \frac{\mu_0}{|\mu_0|}, \frac{\mu_1}{|\mu_1|}\Big)&
    \text{ if } |\mu_0|=|\mu_1|,\\[1.2em] 
\infty& \text{ else.} \ea \right. 
\]
Here, $|\mu_j|= \mu_j(\Omega)$ is the total mass of the 
measure. The geodesic curves are given in terms of the classical
optimal transport, see \cite[Ch.\ 7]{AmGiSa05GFMS}.  

For $\alpha=0$ we obtain a scaled version of the Hellinger distance
(sometimes also called Hellinger--Kakutani distance), namely
\[
\msD_{0,\beta}(\mu_0,\mu_1) = \frac{2}{\sqrt{\beta}}\msH(\mu_0,\mu_1)
= \frac2{\sqrt\beta} \left(\int_\Omega \left[
\Big(\frac{\rmd \mu_0}{\rmd\mu_*} \Big)^{1/2} -
\Big(\frac{\rmd \mu_1}{\rmd\mu_*} \Big)^{1/2} \right]^2 \dd\mu_*
\right)^{1/2}
\]
for a reference measure $\mu_*$ with $\mu_i\ll\mu_*$ (e.g.\
$\mu_*=\mu_0+\mu_1$), see \cite[Theorem 4]{Schi96SPLD}. The geodesic
curves are given by linear interpolation of the square roots of the
densities, i.e.\
\begin{equation}
  \label{eq:HellGeod}
  \big(\mu^\rmH(s)\big)(A)  = \int_A \left( (1{-}s)
    \Big(\frac{\rmd \mu_0}{\rmd(\mu_0{+}\mu_1)} \Big)^{1/2} +s
    \Big(\frac{\rmd \mu_1}{\rmd(\mu_0{+}\mu_1)} \Big)^{1/2} \right)^2 
  \dd(\mu_0{+}\mu_1).
\end{equation}
By using the estimate $\mu^\rmH(s)\geq (1{-}s)^2\mu_0+ s^2\mu_1$ and
choosing $s\in[0,1]$ optimally, we obtain the lower estimate
$|\mu^\rmH(s)|\geq |\mu_0||\mu_1|/(|\mu_0|{+}|\mu_1|)$, i.e.\ the
total mass of the geodesic $\mu^\rmH(s)$ is bounded from below by half
of the harmonic mean of the total masses of $\mu_0$ and
$\mu_1$. Moreover, an elementary calculation gives the identity
\begin{equation}
|\mu^\rmH(s)| = (1{-}s)|\mu_0| + s|\mu_1| -s(1{-}s)\msH(\mu_0,\mu_1)^2.
\end{equation}

\section{The Hellinger--Kantorovich distance}
\label{s:HellWass}
In this section we discuss the dissipation distance
$\msD_{\alpha,\beta}(\mu_0,\mu_1)$ that is induced by the Onsager
operator $\bbK_{\alpha,\beta}(c)\xi = -\div (\alpha c \nabla \xi) +
\beta c \xi$, given for $\mu_0,\mu_1\in\MMM(\Omega)$ as in
\eqref{eq:GeodDist0}.  Using Proposition~\ref{p:equiDist} we can
rewrite this formulation in an equivalent form as
\begin{equation}\label{eq:dissHK1}
\begin{aligned}
\msD_{\alpha,\beta}(\mu_0,\mu_1)^2=\inf\Big\{\int_0^1
\int_\Omega\big[\alpha|\Xi|^2+\beta\xi^2\big]\dd \mu(s)\,\dd s\,
\Big|\qquad\qquad\\
\tfrac{\dd}{\dd s}\mu +\alpha\div(\mu\Xi)=\beta\mu\xi,~\mu_0
 \overset{\mu}{\leadsto} \mu_1\Big\}
\end{aligned}
\end{equation}
with $\Xi:[0,1]\times\Omega\to\R^d$ denoting the vector field.

In most of this section we will restrict ourselves 
without loss of generality to the case $\alpha=1$ and $\beta=4$ for
simplicity.  Occasionally, we will give some of the formulas for
general $\alpha$ and $\beta$ to highlight the dependence on these
parameters. Note that we can always use the simple scaling
$\bbK_{\alpha,\beta}= \beta \bbK_{\alpha/\beta,1}$ giving the general
relation $\msD_{\alpha,\beta}(\mu_0,\mu_1)=
\msD_{\alpha/\beta,1}(\mu_0,\mu_1)/\sqrt{\beta}$. Moreover, the factor
$\sqrt{\alpha/\beta}$ can be transformed away by rescaling $\Omega$,
i.e.\ $x \mapsto \sqrt{\alpha/\beta} x$.

Note that for sufficiently regular $\mu$, $\xi$, and $\Xi$ in \eqref{eq:dissHK1}
we obtain by Proposition \ref{p:equiDist} $\Xi=\nabla\xi$ and formal calculation
leads to the following system of equations for geodesic curves
\begin{align}
\dot\mu = -\alpha\div(\mu\nabla\xi) + \beta\mu\xi,\quad
\dot \xi +\frac{\alpha}{2}|\nabla\xi|^2 +\frac{\beta}{2}\xi^2 .
\end{align}
For the case $\beta = 0$ and $\alpha=1$ this corresponds to
\cite[Eqn.\,(37)]{BenBre00CFMS}. A full justification of this coupled
system is given in \cite[Sect.\,8.6]{LiMiSa14?Theory}.

\subsection[The optimal curves for ${\msD_{\alpha,\beta}}$ with one or two
  mass-points]{The optimal curves for $\bm{\msD_{\alpha,\beta}}$ with one or two
  mass-points}
\label{ss:MassDHW}

The striking feature of optimal transport is that for affine
mobilities point masses (Dirac measures) are transported as point
masses, i.e.\ the geodesic curve connecting $\mu_0=\delta_{x_0}$ and
$\mu_1=\delta_{x_1}$ is given by $\mu_s=\delta_{x(s)}$, where
$[0,1]\ni s \mapsto x(s)$ is a geodesic curve in the underlying domain
$\Omega$.

Since the Onsager operator $\bbK_{\alpha,\beta}(c)$ in \eqref{eq:HK-Operator}
depends only linearly on the state $c$, we expect a similar behavior.
In particular, note that the definition of the distance in
\eqref{eq:dissHK1} is well-defined for general curves of 
measures $\mu(s)\in \MMM(\Omega)$ if we understand the linear
constraint $\tfrac{\dd}{\dd s}\mu +\alpha\div(\mu\Xi)=\beta\mu\xi$
in the distributional sense.

As a first step, it is instructive to study the 
$\bbK_{\alpha,\beta}$-length of curves given by a moving 
point mass in the form
\[
\gamma_{x,a}:[0,1]\ni s \mapsto \mu(s)= a(s)\delta_{x(s)} \quad \text{with }
x(s)\in \Omega \text{ and } a(s)\geq0.
\]
Minimizing the action functional in \eqref{eq:dissHK1} only 
over curves of this form for given end points $a_i\delta_{x_i}$, 
$i=0,1$, always gives an upper bound for the distance 
$\msD_{\alpha,\beta}(a_0\delta_{x_0},a_1\delta_{x_1})$.
Indeed, we show that up to a certain threshold for 
the Euclidean distance $|x_0{-}x_1|$ it will even 
be the exact distance and the minimizing $\gamma_{x,a}$ 
is a geodesic curve.

The main point is that we are able to calculate the $s$-derivative
of $\mu(s) = \gamma_{x,a}(s)$ and compare it to the continuity equation. 
Multiplying the continuity equation with test functions we obtain
after integration by parts
\[
\frac\rmd{\rmd s} \mu_i(s) = 
-\div\big( \dot x(s)a(s)\delta_{x(s)}\big) + \dot a(s)\delta_{x(s)}= 
-\div\big(\dot x(s) \mu(s)\big) + \frac{\dot a(s)}{a(s)}\mu(s).
\] 
Thus, comparing with the continuity equation in the definition of
$\msD_{\alpha,\beta}$ we find the relations
\begin{equation}\label{eq:charODEa}
\Xi(s,x(s)) = \frac{1}{\alpha}\dot x(s) \quad \text{and} \quad
\xi(s,x(s)) = \frac{\dot a(s)}{\beta a(s)}.
\end{equation}
We may realize the constraint $\Xi=\nabla \xi$ via
$\xi(s,y)= \xi(s,x(s))+ \frac1\alpha\dot x(s)\cdot (y{-}x(s))$. 

Having identified the vector and scalar field $\Xi$ and $\xi$,
respectively, we obtain the $\bbK_{\alpha,\beta}$-length of the curve 
$s\mapsto a(s)\delta_{x(s)}$ via
\begin{equation}
  \label{eq:Length}
  \Length_{\alpha,\beta}(\gamma_{x,a})^2= \int_0^1 
\Big[\frac{1}{\alpha}|\dot x(s)|^2 + \frac1\beta\Big(\frac{\dot
  a(s)}{a(s)}\Big)^2\Big]a(s)\dd s,
\end{equation} 
for $\alpha=0$ and $\beta=8$ this corresponds to the representation 
in \cite[Thm.\,4]{Schi96SPLD} for the Hellinger--Kakutani distance.
Minimizing this expression for given endpoints of $\gamma_{x,a}$ 
we find that $x(s)$ travels along a straight line, which reflects the
fact that our choice of metric in $\Omega$ is the Euclidean one. However,  
the speed will not be constant. Hence, we introduce functions
\begin{equation}
  \label{eq:mfR}
\begin{split}
  \rho \in \mfR(0,1)&:= \set{ \rho\in \rmH^1(0,1)}{ \rho(0)=0,\ \rho(1)=1,\
  \dot\rho\geq 0}\\
a\in\mfA(a_0,a_1)&:=\set{ a\in \rmH^1(0,1)}{ a>0,\ a(0)=a_0,\
  a(1)=a_1}
\end{split}
\end{equation}
such that we can write $x(s) = (1{-}\rho(s))x_0+\rho(s) x_1$ and
have $|\dot x(s)|=\dot\rho(s)L$ with $L=|x_1{-}x_0|$. For the 
one-mass-point problem we define the function 
$\calJ^\mathrm{1mp}_{\alpha,\beta}:\left[0,\infty\right[\times
\left[0,\infty\right[^2\to\left[0,\infty\right[$ via  
\begin{equation}\label{eq:infJ1mp}
\calJ^\mathrm{1mp}_{\alpha,\beta}(L^2,a_0,a_1):=
\inf\Big\{ \int_0^1 \frac{L^2}\alpha \dot\rho(s)^2a(s) 
+ \frac{\dot a(s)^2}{\beta a(s)} \dd s\,\Big|\,\rho\in \mfR(0,1), \ a\in
  \mfA(a_0,a_1)\Big\}.
\end{equation}
The functional $\calJ_{\alpha,\beta}^\mathrm{1mp}$ satisfies 
a scaling identity with respect to the parameter $\alpha,\beta>0$: For
$\theta>0$ we have
\begin{equation}\label{eq:scalJ1mp}
\calJ^\mathrm{1mp}_{\alpha,\beta}(L^2,a_0,a_1)= \frac1\theta
\calJ^\mathrm{1mp}_{1,\beta/\theta}\Big(\frac{\theta L^2 }{\alpha}, a_0,a_1\Big). 
\end{equation}
Hence, we can restrict ourselves to one particular choice of
$\beta/\theta>0$ such that the general case can be recovered 
from a rescaling of the Euclidean distance in $\Omega$. In 
particular, it will prove convenient to choose $\theta = \beta/4$
such that we will consider $\calJ_{1,4}^\mathrm{1mp}$ first, which 
is also the scaling used in \cite{LiMiSa14?Theory}.

\begin{theorem}
\label{th:MassPointCost}
We have 
\begin{equation}\label{eq:MassPoiCost}
\begin{aligned}
\calJ^\mathrm{1mp}_{1,4}(L^2,a_0,a_1)&= a_0+a_1
-2\sqrt{a_0a_1}\,\trcos(L)\ 
\text{ with }\\ 
\trcos(L) &:=
  \begin{cases}
   \cos(L) &\text{for } L< \pi,\\ -1&\text{for }L\geq \pi.
  \end{cases}
\end{aligned}
\end{equation}
The infimum is a minimum for $L<\pi$, and it is attained for 
\begin{equation}
\label{eq:1mpCurve}
\begin{split}
a(s)& = (1{-}s)^2 a_0 + s^2 a_1 + 2 s (1{-}s) \sqrt{a_0a_1}\cos (L),\\
\rho(s)&=
\left\{
\begin{array}{ll}
\frac1{L}\arctan\big(\frac{s\sin(L)\sqrt{a_1}}{ (1{-}s)\sqrt{a_0}+s\cos(L)\sqrt{a_1}} \big)&
\text{if $(1{-}s)\sqrt{a_0}+s\cos(L)\sqrt{a_1}>0$},\\[0.5em]
\frac{\pi}{2L}& 
\text{if $(1{-}s)\sqrt{a_0}+s\cos(L)\sqrt{a_1}=0$},\\[0.5em]
\frac1{L}\arctan\big(\frac{s\sin(L)\sqrt{a_1}}{ (1{-}s)\sqrt{a_0}+s\cos(L)\sqrt{a_1}} \big)+ \frac{\pi}{L} &
\text{otherwise}.
\end{array}\right.
\end{split}
\end{equation}
For $L\geq \pi$ the minimizing sequences converge to $a(s)= c
(s{-}\theta)^2$ and $\dot\rho(s)=\delta_{\theta}(s)$ for certain 
$c\geq0$ and $\theta\in[0,1]$. 
\end{theorem}
\begin{proof} 
  To study the infimum of $\calJ^\text{1mp}_{1,4}$ we transform
  the system by using $b(s)=\sqrt{a(s)}$. Keeping $L > 0$ fixed we
  obtain the functional
\[
\calK_L(b,\rho)= \int_0^1  L^2\, b(s)^2\, \dot \rho(s)^2 +
\dot b(s)^2 \dd s .
\]
Clearly, the infimum of $\calK_L$ gives the infimum in the definition
of $\calJ^\text{1mp}_{1,4}$. We now consider a minimizing
sequence $(b_n,\rho_n)$ and observe that $(b_n)$ must remain bounded
in $\rmH^1(0,1)$. Hence, after choosing a suitable subsequence (not
relabeled), we may assume $b_n \rightharpoonup b$ in $\rmH^1(0,1)$. 
We distinguish between the following three cases.
\smallskip

\textbf{Case 1.} $\underline b:= \min\set{ b(s) }{s\in [0,1]} >0$:
In this case we may further conclude that $\rho_n$ is also bounded in
$\rmH^1(0,1)$. Hence, we can also assume $\rho_n \rightharpoonup \rho$
in $\rmH^1(0,1)$. Because of the lower semicontinuity of $\calK_L$ on
$\rmH^1(0,1)^2$ we conclude that $(b,\rho)$ is the global minimizer of
$\calK_L$. This implies that $(a,\rho)=(b^2,\rho)$ is the global
minimizer of $\calJ^\text{1mp}_{1,4}$, which certainly
satisfies the Euler--Lagrange equations 
\[
\tfrac{\dd}{\dd s}\big( \dot\rho a\big)=0,\quad  4\big(L \dot\rho\big)^2 
-\big( \dot a/a\big)^2 -2
\tfrac{\dd}{\dd s}\big( \dot a/a\big) =0.
\]
>From the first equation and $\int_0^1 \dot\rho\dd s =1$ we obtain 
\[
\dot\rho(s) = H[a]/a(s),\quad \text{where } H[a]=\Big( \ts \int_0^1
1/a(s)\dd s\Big)^{-1} 
\]
denotes the harmonic mean, which satisfies $H[a]\geq \underline b^2
>0$. By inserting this into the second equation we see that all solutions
are given in the form
\[
a(s) = c_0 + c_1(s{-}\theta)^2 \quad \text{with } c_0c_1= L^2 H[a]^2.
\]
Hence, together with the boundary conditions $a_0=a(0)=c_0+c_1\theta^2$
and $a_1=a(1)=c_0+c_1(1{-}\theta)^2$ we have three nonlinear equations
for the unknowns $c_0,c_1$, and $\theta$, which can be solved easily for 
$L < \pi$ giving a unique solution. For $L\geq \pi$ no solution with
positive $\underline b$ exists. In particular, since the primitive of the inverse
of a strictly positive quadratic function is given in terms of the $\arctan$ function
we obtain the formulas in \eqref{eq:1mpCurve}, using also the addition theorem
for $\arctan$.
 
Thus, in the case $\underline b>0$ the global minimizer is the unique, positive
critical point $(a,\rho)$.\smallskip

\textbf{Case 2.} $\underline b=0$: Since $b$ is continuous, there
exists $s_*\in [0,1]$ with $b(s_*)=0$. Neglecting the term
$L^2 b^2\dot \rho^2$ in the integrand in $\calK_L$ we can
minimize the remaining quadratic term subject to the boundary
conditions $b(0)=\sqrt{a_0}$, $b(1)=\sqrt{a_1}$, and $b(s_*)=0$. This
leads to a minimizer that is piecewise affine and gives the lower bound 
\begin{equation}
  \label{eq:calK-opt}
\calK_L(b,\rho)\geq  \frac{a_0}{s_*} +
\frac{a_1}{1{-}s_*}  \geq \big( \sqrt{a_0}+ \sqrt{a_1}\big)^2,
\end{equation}
where the last estimate follows from minimization in $s_*$. 

It is now easy to see that the value $\big( \sqrt{a_0}+
\sqrt{a_1}\big)^2$ is indeed the infimum, since it can be obtained as
a limit of a minimizing sequence. For this take piecewise affine
functions $(b_n,\rho_n)$ satisfying $(b_n(s),\dot\rho_n(s))=(0,n)$ 
for $s \in [s_n,s_n{+}1/n]$ with $s_n
\to s_0$, where $s_0$ is the optimal $s_*$ in
\eqref{eq:calK-opt}. On $[0,s_n]$ we take $(\dot
b_n(s),\rho(s))=({-}\sqrt{a_0}/s_n, 0)$ and similarly on $[s_n{+}1/n,1]$.
\smallskip

\textbf{General case:} Since the infimum obtained in case 2 is
strictly larger than that in case 1 (because of $L<\pi$), we see that
the two cases exclude each other. If $L<\pi$ then case 1 occurs while
for $L\geq \pi$ case 2 sets in. Hence, the theorem is established. 
\end{proof}

Although the stationary states in Theorem \ref{eq:MassPoiCost} may be the global
minimizers, they are not always the geodesic curves with respect
to $\msD_{1,4}$. To see this we consider the pure reaction
case and define the curve
\[
\wh\gamma_{\ol a_0,\ol a_1}(s)= \ol a_0(s) \delta_{x_0} + \ol a_1(s)\delta_{x_1}
\]
also connecting the measures $\mu_j= a_j\delta_{x_j}$ if $\ol a_j(j)=
a_j$ and $\ol a_j(1{-}j)=0$ for $j=0,1$. If $\ol a_0(s), \,\ol
a_1(s)>0$, this curve consists of two separated mass points that do not
move. As in the previous case of the moving mass point we can compute
the solutions of the continuity equation to obtain
\[
\Xi\equiv 0 \quad\text{ and }\quad \xi(s,x_j) = 
\frac{\dot{\ol a}_j(s)}{4\ol a_j(s)}.
\]
The squared length of these curves is given by $\frac14\sum_j\int_0^1
\dot{\ol a}_j^2/\ol a_j\dd s$ and the optimal choice for $\ol a_j$ is
$\ol a_0(s)=a_0(1{-}s)^2$ and $\ol a_1(s)=a_1 s^2$ giving the minimal
squared length 
\[
\mathrm{Length}_{1,4}(\wh\gamma_\text{opt})^2 \;=\; a_0+a_1\;
\geq \;\msD_{1,4}(a_0\delta_{x_0}, a_1 \delta_{x_1}).
\]
We see that this result is less that
$\calJ^\mathrm{1mp}_{1,4}(|x_{1}{-}x_0|^2, a_0,a_0)$ for $\pi/2
<|x_1{-}x_0| < \pi$. In fact, we will show 
later that the last estimate is sharp if and only if
$|x_1{-}x_0| \geq  \pi/2$. 

\begin{figure}[t]
\centerline{\includegraphics[width=0.8\textwidth]{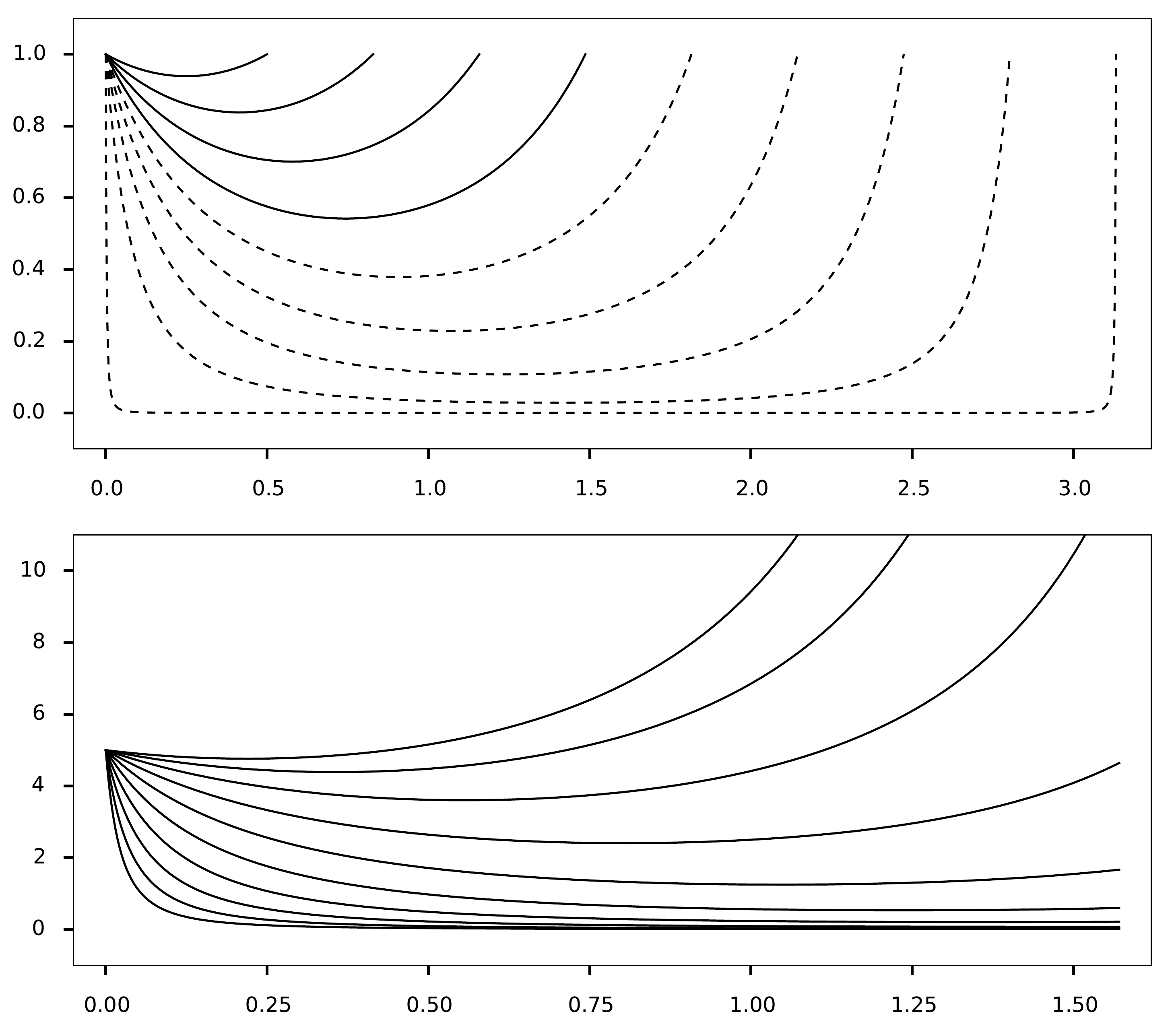}}
\caption{\emph{Top:} The curves $s\mapsto (L\int_0^s\dot\rho(\tau)\dd \tau, a(s))$ 
for different values of $0<L<\pi$. Solid curves are  true geodesics, while 
dashed curves are shortest ``one-mass-point paths'' but not geodesic curves. 
\emph{Bottom:} Curves for $L=\pi/2$ and different mass ratios $a_0/a_1$.}
\label{fig.J-curves} 
\end{figure}

\begin{figure}[t]
\includegraphics[width=\textwidth]{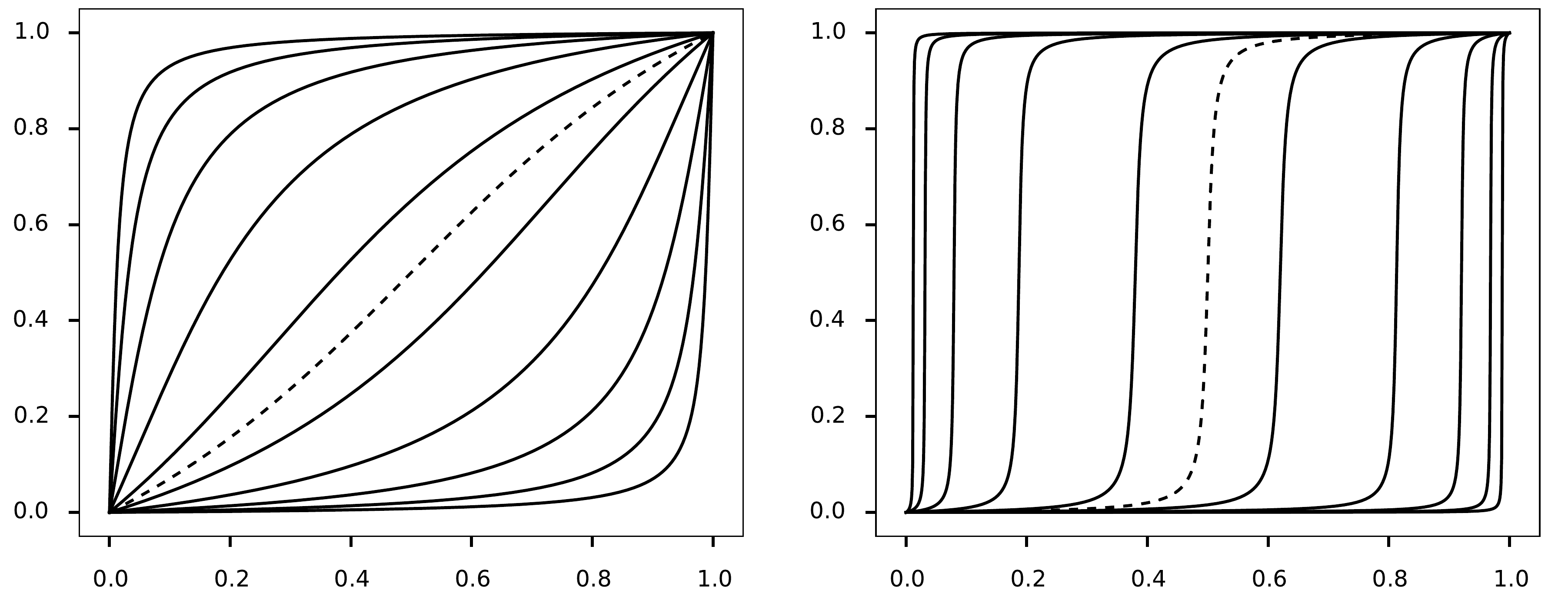}
\caption{The function $s\mapsto\rho(s)$ in Theorem~\ref{th:MassPointCost} for different 
ratios $a_0/a_1$ and $L=\pi/2$ (\emph{left}) and $L=\pi/1.1$ (\emph{right}). The dashed 
curve corresponds to $a_0/a_1=1$, while curves above and below satisfy $a_0/a_1<1$ and 
$a_0/a_1>1$, respectively.}
\label{fig.speed_fcn} 
\end{figure}

To highlight the dependencies on $\alpha$ and $\beta$ in
$\calJ^\mathrm{1mp}_{\alpha, \beta}$ we can use the scaling
\eqref{eq:scalJ1mp} for all $\alpha,\beta>0$. To include the limit
cases of the Hellinger distance (i.e.\ $\alpha=0$) and the
Kantorovich--Wasserstein distance (i.e.\ $\beta =0$) we define the
functions $ \mfS_{\alpha,\beta}$ via
\begin{equation}
  \label{eq:DD.alpha.beta}
  \mfS_{\alpha,\beta}(L^2, b_0,b_1):=
\begin{cases}
  \frac4\beta\Big( b_0^2 + b_1^2 - 2b_0b_1 \trcos\big( 
  {\ts\sqrt{\frac{\beta}{4\alpha}}\, L}\big) \Big) & \text{for }\alpha,\,\beta>0,
\\[0.5em]
  \frac{4}\beta \big( b_0^2 + b_1^2\big) & \text{for } \alpha=L=0,
  \text{ and } \beta>0,
\\[0.5em]
  \frac{L^2}{\alpha } b_0^2& \text{for }\beta=0,\ \alpha>0,\text{ and }b_1=b_0,
\\[0.5em]
  0& \text{for }\alpha=\beta=L=0 \text{ and }b_0=b_1,
\\[0.5em]
   \infty & \text{otherwise}. 
\end{cases} 
\end{equation}
We emphasize that $\mfS_{0,\beta}$ and $\mfS_{\alpha,0}$ can be
obtained as $\Gamma$-limits of $\mfS_{\alpha_n,\beta_n}$ for
$\beta_n\searrow 0$ or $\alpha_n\searrow 0$, respectively.

Using $ \mfS_{\alpha,\beta}$ we can express
$\calJ^\mathrm{1mp}_{\alpha,\beta} $ for all $\alpha,\,\beta\geq 0$,
where the cases $\alpha=0$ or $\beta =0$ mean that $\dot\rho \equiv0$
or $\dot a\equiv 0$, respectively. Moreover
$\calJ^\mathrm{1mp}_{\alpha,\beta}=+\infty$, if the set of competitors
$(a,\rho)$ providing finite values is empty.

\begin{corollary}\label{cor:boundaryCases}
For all $\alpha,\,\beta \geq 0$ we have  
\[
\calJ^\mathrm{1mp}_{\alpha,\beta}(L^2,a_0,a_1)= \mfS_{\alpha,\beta} 
\big( L^2, \sqrt{a_0}, \sqrt{a_1}\big). 
\]
\end{corollary}
\begin{proof}
We only need to consider the boundary cases. 

For $\alpha=\beta =0 $ we have $\dot a \equiv 0 \equiv  \dot \rho$,
which implies that $\calJ^\mathrm{1mp}_{0,0}$ is finite only for $L=0$
and $a_0=a_1$.

For $\alpha=0$ and $\beta>0$ we have $\dot\rho\equiv 0$ and obtain a
finite value only for $L=0$. Clearly, the infimum of $\int_0^1 \dot
a^2/(\beta a) \dd s$ is given by
$4(\sqrt{a(1)}{-}\sqrt{a(0)})^2/\beta$, which is the desired result.

The case $\beta=0$ and $\alpha>0$ provides $\dot a\equiv 0$ and
$\dot\rho\equiv 1$. Hence, the infimum is $L^2 a_0/\alpha$ for
$a_1=a_0$ and $\infty$ otherwise. 
\end{proof}

\begin{example}[Mass splitting]\label{ex:MassSplit1}
At the end of this subsection we give a more complicated example
for an optimal curve consisting of two point masses. We want to connect the measures 
$\mu_0=a_0\delta_{x_0}$ and $\mu_1=a_1\delta_{x_0}+b_1\delta_{x_1}$, where
$L=|x_0{-}x_1|<\pi/2$, i.e.\ $\trcos(L)=\cos(L)>0$. So the
question is how much of the mass at $x_0$ is kept there, how much
of the mass is used for transport, and how much mass is created at
$x_1$. We consider the curve 
\[
\gamma(s) = a(s)\delta_{x_0} +c(s)\delta_{x(s)} + b(s)\delta_{x_1}
\]
with $a(s),b(s),c(s)\geq 0$ and the boundary conditions 
\[
x(0)=x_0,\ x(1)=x_1,\quad a(0)+c(0)=a_0,\quad a(1)=a_1,\quad b(0)=0,
\quad  c(1)+b(1)=b_1.
\]
Choosing $\alpha=1$ and $\beta=4$ and optimizing each of the given three
curves under their own boundary 
conditions gives
\[
\mathrm{Length}_{1,4}(\gamma)^2 = (\sqrt{a(0)}{-}\sqrt{a(1)})^2 +
c(0){+}c(1) - 2\sqrt{c(0)c(1)}\trcos(L) + b(1). 
\]
>From the constraint $c(1)+b(1)=b_1$ and the second last term, we see
that it is optimal to choose $c(1)$ as large as possible, namely
$c(1)=b_1$ and $b(1)=0$. In particular, we have no creation at $x_1$,
i.e.\ $b\equiv 0$. Setting  $c_0=c(0)$
and eliminating $a(0)=a_0-c_0$ we find 
\[
\mathrm{Length}_{1,4}(\gamma)^2 = a_0-2\sqrt{a_1}\sqrt{a_0{-}c_0} -
2\trcos(L)\sqrt{b_1c_0} + b_1 + a_1.
\]
The minimal value is achieved for the choice 
$c_0=a_0b_1 \trcos(L)^2/(a_1{+}b_1\trcos(L)^2)$, which means a mass splitting as
$0<c_0 < a_0$. Hence, we have established the estimate
\[
\msD_{1,4}(\mu_0,\mu_1)^2=\msD_{1,4}(a_0\delta_{x_0}\,, \, a_1\delta_{x_0} +
b_1\delta_{x_1} )^2 \leq  a_0+a_1+b_1 -
2\sqrt{a_0(a_1{+}b_1\trcos(|x_0{-}x_1|)^2)} .
\]
In fact, it will be shown in Example \ref{ex:MassSplit2} that the
curve $\gamma$ is indeed a geodesic curve, i.e.\ ``$\leq$'' can be
replaced by ``$=$''.  
\end{example}

\subsection{Optimal transport on the cone}

The crucial point in the characterization of the distance
$\msD_{\alpha,\beta}$ induced by the Onsager operator $\bbK_{\alpha,\beta}$
in \eqref{eq:dissHK1}, for $\alpha=1$ and $\beta=4$,
is that the functional $\calJ^\mathrm{1mp}_{1,4}$ in
\eqref{eq:infJ1mp}, which gives the cost for optimally transporting
a single mass point, is closely related to the metric construction
of a \emph{cone over the  metric space $(\Omega,|\cdot|)$}. 
We will briefly explain the construction in this section and refer
to \cite[Sect.\,3.6.2]{BuBuIv01CMG} for more details.

 Given the closed and
convex domain $\Omega \subset \R^d$ we construct the cone $\CC$
as the quotient of $\Omega\ti {[0,\infty[}$ over $\Omega\ti\{0\}$,
i.e.,
\[
\CC:= \bigslant{\big({\Omega}\ti {[0,\infty[}\big)}
{\big(\textstyle{\Omega}\ti \{0\}\big)}.
\]
In particular, all points in $\Omega\ti \{0\}$ are 
identified with one point, namely the tip of
the cone denoted by $\TT$. For any $x\in \Omega$ and $r>0$
the equivalence classes are denoted by $z=[x,r]\in\CC$
while for $r=0$ the equivalence class $[x,0]$ is equal
to $\TT$. 

Motivated by the previous section we define the distance 
$\msd_\mfC:\CC\ti \CC \to {[0,\infty[}$ on the cone space $\CC$ 
as follows:
\[
\msd_\mfC ([x_0,r_0],[x_1,r_1])^2:= r_0^2+r_1^2 - 2r_0r_1
  \trcos(|x_{1}{-}x_0|), 
\]
where $\trcos$ is defined as in Theorem \ref{th:MassPointCost}. 

For the special case that $\Omega =[0,\ell]\subset \R^2$ with
$0<\ell<2\pi$, we can visualize $\CC=\mfC_{[0,\ell]}$ by the
two-dimensional sector 
\[
\Sigma_\ell := \set{y=(r\cos x, r \sin x)\in \R }{ r\geq 0,\ x\in
  [0,\ell]},
\]
where $y=(0,0)$ corresponds to the tip $\TT$. The induced distance is
the Euclidian distance restricted to $\Sigma_\ell$,
i.e.\ the geodesic curve between $y_0$ and $y_1$ is a straight
segment if $|x_1{-}x_0|\leq \pi$ while it consists of the two rays
connecting $y=0$ with $y_0$ and $y_1$, respectively, if $\pi \leq
|x_1{-}x_0|\leq \ell < 2\pi$, see Figure \ref{fig.ConeSector}. 
\begin{figure}
 \centerline{\unitlength1.5em
\includegraphics[width=10\unitlength]{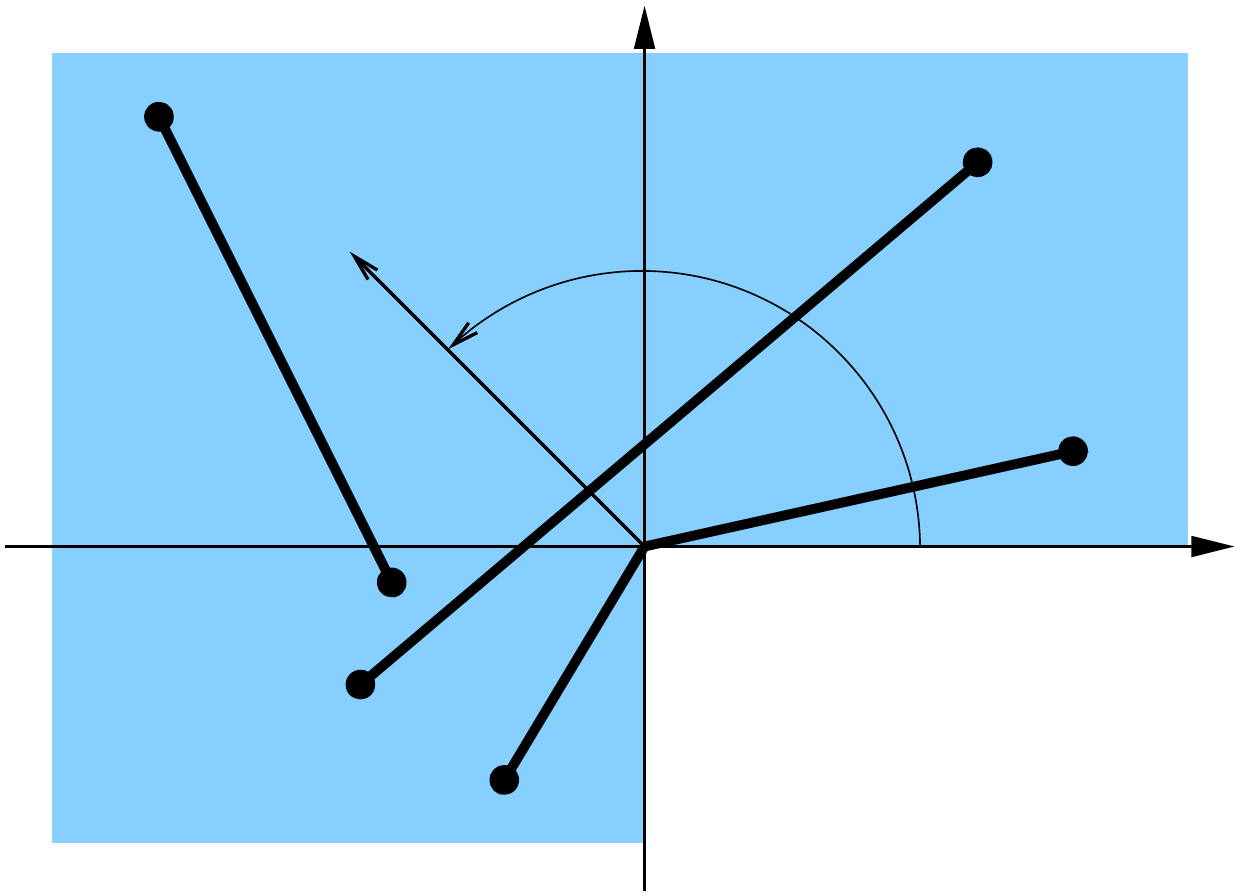}
\begin{picture}(0,0)
\put(-5.8,2.4){
\put(1,-0.2){$0$}
\put(5,-0.2){$y_1$} \put(0.5,5.0){$y_2$}
\put(1.2,2.8){$x$} \put(-0.75,1){$r$}
}
\end{picture}
}
\caption{The cone $\mfC_{[0,3\pi/2]}$ represented as sector in $\R^2$
  via $(y_1,y_2)=(r\cos x,r\sin x)$ and three geodesic curves. 
  The angle $x=\pi$ is critical for smoothness of geodesic
  curves.}
\label{fig.ConeSector} 
\end{figure}
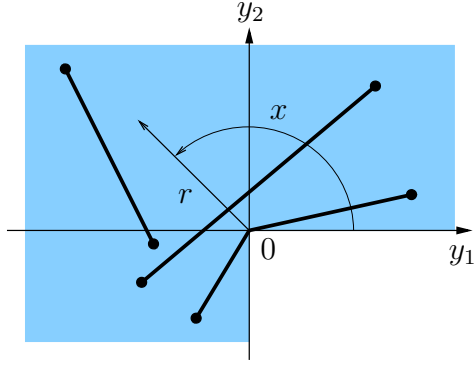 
In the case of the traveling mass point discussed in the previous
section we identify the Dirac measures $a_i\delta_{x_i}$ with pairs
$[x_i,\sqrt{a_i}]\in\CC$. Thus, the result of Theorem \ref{th:MassPointCost} 
can be reformulated as
\[
\calJ_{1,4}^\mathrm{1mp}(|x_0{-}x_1|^2;a_0,a_1)
=\msd_\mfC\big([x_0,\sqrt{a_0}],[x_1,\sqrt{a_1}]\big)^2.
\]

For general coefficients $\alpha,\beta>0$ the distance $\msd_\mfC$
has to be replaced  by
\[
\msd_{\mfC}^{\alpha,\beta}(z_0,z_1)=\msd_{\mfC}^{\alpha,\beta}([x_0,r_0],[x_1,r_1])=
\sqrt{\mfS_{\alpha,\beta}(|x_1{-}x_0|, r_0,r_1)}, 
\]
see \eqref{eq:DD.alpha.beta}. This distance can be seen as a geodesic
distance on the cone $\CC$ induced by a Riemannian metric (outside of
the vertex $\TT$) given by the tensor
\begin{equation}
  \label{eq:bbG-Cone}
  \bbG^\mfC_{\alpha,\beta}([x,r]) := \bma{cc}\frac{r^2}{\alpha}
I_{\R^d}&0\\ 0 & 4/\beta \ema \ \in \R^{(d+1)\ti(d+1)}.
\end{equation} 
This fact is already seen in \eqref{eq:Length} if we set
$a(s)=r(s)^2$.

\subsubsection{Geodesic curves in the cone space}
\label{sss:GeoCurveCone}

As is shown in \cite[Sect.\,3.6.2]{BuBuIv01CMG}, the pair
$(\CC,\msd_\mfC)$ is a complete geodesic space, where each two points can
be connected by a unique arclength-parameterized geodesic curve. These
curves are given by the following geodesic interpolator
$Z(s;\cdot,\cdot):\CC\ti\CC\to\CC$ (recall $z_j=[x_j,r_j]$)
\begin{equation}\label{eq:GeodInterpol}
\begin{split}
Z(s;z_0,z_1) &:= [X(s;z_0,z_1),R(s;z_0,z_1)] \quad \text{where}
\\[0.7em]
R(s;z_0,z_1)^2&:=(1{-}s)^2 r_0^2 + s^2 r_1^2
+2s(1{-}s)r_0r_1 \trcos(|x_0{-}x_1|), 
\\[0.7em]
 X(s;z_0,z_1)&:= \big(1{-}\rho(s;z_0,z_1)\big) x_0 +
 \rho(s;z_0,z_1) x_1,
\\[0.7em]
\rho(s;z_0,z_1)&:= \left\{\ba{cl}\dfrac1{|x_1{-}x_0|} \arccos\Big(
\dfrac{(1{-}s)r_0 + sr_1 \cos|x_1{-}x_0|}{R(s;z_0,z_1)}\Big)
&\text{for } |x_1{-}x_0| < \pi ,  \\ 
 \dfrac12 \Big( 1+ \mathop{\mathrm{sign}}\big((1{-}s)r_0-sr_1\big)
 \Big) & \text{for }|x_1{-}x_0| \geq \pi.  \ea \right. 
\end{split}
\end{equation}
Note that in the definition of $R$ there is a ``$+$'' in front of
the cosine term, while there is a ``$-$'' in the distance $\msd_\mfC$. 
Moreover, by elementary geometric identities it is easy to see
that the formula for $\rho$ is equivalent to the one obtained in 
Theorem~\ref{th:MassPointCost}.

In particular, the curve defined by $\gamma(s)=Z(s;z_0,z_1)$
is a constant speed geodesic curve with respect to the distance
$\msd_\mfC$ connecting $z_0$ and $z_1$, i.e.,
\begin{equation}\label{eq:geodCone}
\forall\,0\leq s<t\leq 1:\quad
\msd_\mfC\big(\gamma(s),\gamma(t)\big) = |t{-}s|\msd_\mfC(z_0,z_1).
\end{equation}
As for the Wasserstein-Kantorovich distance, the geodesic curves in $\CC$ are
key for the construction of the geodesic curves with respect to the 
Hellinger--Kantorovich distance. 
We will discuss this in Section \ref{ss:GeodCurves} in detail.

\subsubsection{The transport distance modulo reservoirs on the cone space} 
\label{sss:WassDistReservoir}

Using the theory of optimal transport (cf.\
\cite{AmGiSa05GFMS,Vill09OTON}) we can define the
Kantorovich--Wasserstein distance $\msW_\mfC$ associated with the
cone distance $\msd_\mfC$ on the set of all nonnegative and finite measures
$\MMM_2(\CC)$ as follows. If $\lambda_0(\CC)\neq \lambda_1(\CC)$,
then we set $\msW_\mfC(\lambda_0,\lambda_1)=\infty$ and otherwise we
set
\[
 \msW_\mfC(\lambda_0,\lambda_1)^2:= \inf\Bigset{ \iint_{\CC\ti\CC}
   \msd_\mfC(z_0,z_1)^2\, \rmd \gamma(z_0,z_1)}{\gamma\in
   \calM(\CC{\ti}\CC),\ \Pi_\#^i\gamma=\lambda_i }.
\]
Moreover, there the geodesic interpolation (which is not unique in
general) can be described using an optimal transport plan $\gamma$
that is a minimizer in the definition of $\msW_\mfC$. We denote the
geodesic interpolator by 
\begin{equation}
\calZ_\gamma(s;\lambda_0,\lambda_1):= Z(s; \cdot,\cdot)_\#
\gamma.
\end{equation}

For the proper handling of creation and annihilation of mass 
we introduce a modified distance. The modification occurs via a reservoir of mass 
in the vertex $\TT$ of the cone such that mass is generated
from the reservoir and absorbed into it. The following result shows that 
if we assume that the reservoir is sufficiently big (and in our 
model for the Hellinger--Kantorovich operator it is in fact 
infinite), we never have any true transport over distances larger 
than $\pi/2$ (respectively, $\sqrt{\alpha/\beta}\pi$ in 
the scaled case), which is only half the critical distance 
of possible transport. 

\begin{proposition}[Optimal transport in the presence of  
    large reservoirs]\label{pr:Reservoir} 
We consider arbitrary measures $\lambda_0,\lambda_1\in\MMM_2(\CC)$ with
equal masses $\lambda_0(\CC)=\lambda_1(\CC)$. 
\begin{itemize}
\item[(a)] The function ${[0,\infty[} \ni \kappa \mapsto
w(\lambda_0,\lambda_1,\kappa):= \msW_\mfC( \kappa\delta_{\TT}
{+}\lambda_0,  \kappa\delta_{\TT}{+}\lambda_1) $ is nonincreasing. 
\item[(b)] Define the real numbers 
\[
\theta_j:=\lambda_j(\CC{\setminus} \TT),\quad
\rho_j:=\lambda_j(\{\TT\}), \quad \text{and }
\kappa_*=\max\{0,\theta_1{-}\rho_0,\theta_0 {-}\rho_1\},
\]
then for all $\kappa\geq \kappa_*$ we have
$w(\lambda_0,\lambda_1,\kappa)= w(\lambda_0,\lambda_1,\kappa_*)$
with similar transport plans differing only by
$(\kappa{-}\kappa_*)\delta_{\TT,\TT}$. 
\item[(c)] For any optimal transport plan $\gamma$ connecting $\kappa_*\delta_{\TT}
{+}\lambda_0$ and $\kappa_*\delta_{\TT}{+}\lambda_1$ we have 
$\gamma(\mfN)=0$, where
\[
\mfN:=\big\{(z_0,z_1)\in\CC{\ti}\CC\,\big|\,r_0,r_1>0,~|x_0{-}x_1|>\pi/2\big\},
\]
i.e., there is no transport in $\Omega$ over distances longer than
$\pi/2$.
\end{itemize} 
\end{proposition}
\begin{proof}
\textit{ad (a) }Let $0\leq\kappa_1<\kappa_2$ be given and let 
$\gamma_{\kappa_1},\,\gamma_{\kappa_2}\in\MMM_2(\CC{\ti}\CC)$ 
be optimal transport plans for the pairs $\kappa_1\delta_{\TT}{+}\lambda_i$
and $\kappa_2\delta_{\TT}{+}\lambda_i$, $i=0,1$, respectively.
Since $\kappa_2>\kappa_1$ we can define the transport plan
$\wh\gamma_{\kappa_2}=(\kappa_2{-}\kappa_1)\delta_{(\TT,\TT)}
+\gamma_{\kappa_1}$, which satisfies $\Pi^i_{\#}\wh\gamma_{\kappa_2} =
\kappa_2\delta_{\TT} +\lambda_i$. Thus, $\wh\gamma_{\kappa_2}$ is 
an admissible plan for the minimization problem in the definition 
of $\msW_{\mfC}$ and we obtain the estimate
\begin{align*}
w(\lambda_0,\lambda_1,\kappa_2)^2 &= \msW_{\mfC}(\kappa_2\delta_{\TT}{+}\lambda_0,
\kappa_2\delta_{\TT}{+}\lambda_1)^2\leq \iint_{\CC{\ti}\CC}\msd_{\mfC}(z_0,z_1)^2
\dd \wh\gamma_{\kappa_2}(z_0,z_1)\\
&=\iint_{\CC{\ti}\CC}\msd_{\mfC}(z_0,z_1)^2\dd \gamma_{\kappa_1}(z_0,z_1)
=w(\lambda_0,\lambda_1,\kappa_1)^2.
\end{align*}
Hence, $\kappa\mapsto w(\lambda_0,\lambda_1,\kappa)$ is not increasing.

\textit{ad (b) }Let $\gamma_{\kappa}$ and $\gamma_{\kappa_*}$
denote the optimal transport plans with respect to $\kappa$
and $\kappa_*$, respectively. Similar to (a) we define the 
measure $\wh\gamma_{\kappa_*} = \gamma_{\kappa} 
- (\kappa{-}\kappa_*)\delta_{(\TT,\TT)}$. Obviously, $\wh\gamma_{\kappa_*}$
satisfies $\Pi_{\#}^i\wh\gamma_{\kappa_*} = \lambda_i + \kappa_*\delta_{\TT}$.
It remains to show that $\wh\gamma_{\kappa_*}$ is nonnegative 
and hence is an admissible transport plan. Indeed, due to the estimate 
$\gamma_{\kappa}(\{\TT\}{\times}(\CC\setminus\{\TT\}))\leq \theta_1$ 
and the definition of $\kappa_*$ we obtain
\begin{align*}
\gamma_{\kappa}(\{\TT\}{\times}\{\TT\}) &=
 \gamma_{\kappa}(\{\TT\}{\times}\CC) -
 \gamma_{\kappa}\big(\{\TT\}{\times}(\CC\setminus\{\TT\})\big)\\
&\geq \kappa +\rho_0 -\theta_1\geq \kappa-\kappa_*.
\end{align*}
Hence, $\wh\gamma_{\kappa_*}\geq0$ and, arguing as for (a) we get
$w(\lambda_0,\lambda_1,\kappa_*)\leq
w(\lambda_0,\lambda_1,\kappa)$. However, due to the first part of
the theorem even equality must hold.

\emph{ad (c)} As before let $\gamma$ denote an optimal plan for
lifts $\lambda_0$ and $\lambda_1$ of $\mu_0$ and $\mu_1$.
Assume that $\gamma(\mfN)>0$ such that
$\msd_{\mfC}(z_0,z_1)^2>r_0^2+r_1^2$ for 
$\gamma$-a.a.\ $(z_0,z_1)\in\mfN$ with $z_i=[x_i,r_i]$. 
We aim to construct a new transport plan $\wh\gamma$ 
based on $\gamma$ giving a strictly lower cost and hence 
showing the non-optimality of $\gamma$.

To this end, we introduce the characteristic function $\chi$ of the subset
$\mfN^c:=(\CC{\ti}\CC)\setminus\mfN$. Moreover,
we denote by $\wt\lambda_i\in\MMM_2(\CC)$ the marginals of
$\wt\gamma_\chi=\chi\gamma$, which are obviously absolutely continuous with
respect to $\lambda_i$. We denote the densities with $\rho_i$ such that
$\wt\lambda_i=\rho_i\lambda_i$. In particular, for
$\lambda_i$-a.e.\ $z\in\CC$ we have that  $0\leq \rho_i\leq 1$.

We define the measure $\wh\gamma$
\[
\wh\gamma(\rmd z_0,\rmd z_1) = \wt\gamma_\chi(\rmd z_0,\rmd z_1) 
+ (1{-}\rho_0)\lambda_0(\rmd z_0)\delta_{\TT}(\rmd z_1)
+\delta_{\TT}(\rmd z_0)(1{-}\rho_1)\lambda_1(\rmd z_1).
\]
We easily check that the marginals of $\wh\gamma$ are given by
$\wh\lambda_i=\lambda_i+\kappa\delta_{\TT}$, $i=0,1$, where $\kappa>0$ 
is given by $\kappa=(\gamma{-}\wt\gamma_\chi)(\CC{\ti}\CC)$. 
In particular, $\wh\lambda_i$ is an admissible lift for $\mu_i$.

It remains to show that $\wh\gamma$ has a strictly lower cost
than $\gamma$. We compute
\begin{align*}
\iint_{\CC\times\CC} \msd_{\mfC}(z_0,z_1)^2\dd\wh\gamma
&=\iint_{\CC\times\CC} \msd_{\mfC}(z_0,z_1)^2\dd\wt\gamma_\chi 
+ \int_{\CC} r_0^2(1{-}\rho_0)\dd\lambda_0+ \int_{\CC} r_1^2(1{-}\rho_1)\dd\lambda_1\\
&=\iint_{\mfN^c}\msd_{\mfC}(z_0,z_1)^2\dd \gamma
+\iint_{\mfN}\big(r_0^2+r_1^2\big)\dd\gamma\\
&<\iint_{\CC\times\CC} \msd_\mfC(z_0,z_1)^2\dd\gamma.
\end{align*}
Thus, $\gamma$ cannot be optimal and any optimal transport plan has to
vanish on $\mfN$.
\end{proof}

Using the above proposition we may define a new distance
$\msW_\text{rsv}$ on $\MMM_2(\CC)$ that assumes that the reservoir is
always big enough. Indeed, defining
\[
\msW_\text{rsv}(\lambda_0,\lambda_1):=
\inf_{\kappa>0}\msW_\mfC(\lambda_0{+}\kappa \delta_{\TT},
\lambda_1{+}\kappa \delta_{\TT}) =  \msW_\mfC(\lambda_0{+}\kappa_* \delta_{\TT},
\lambda_1{+}\kappa_* \delta_{\TT}), 
\]
where $\kappa_*$ is given as in Proposition \ref{pr:Reservoir}(b).

\subsection{The Hellinger--Kantorovich distance}
\label{sss:ProjCone}

We can now easily define a distance for measures on $\Omega$ 
by lifting measures $\mu_j\in \MMM(\Omega)$ to measures on $\MMM_2(\CC)$ 
and projecting back measures from $\MMM_2(\CC)$ into $\MMM(\Omega)$. 
We define the projection $\mfP : \MMM_2(\CC) \to \MMM(\Omega)$ via 
\[
\int_\Omega \phi(x)\dd \mfP\lambda =
\int_{\CC} r^2 \phi(x)\dd\lambda([x,r]) \quad \text{for all }
\forall\, \phi\in \rmC^0(\Omega).
\]
In the last formula we use that for $r>0$ the equivalence class
$[x,r]$ uniquely determines $x$ and $r$ and that the prefactor $r^2$
makes the function $\Phi: [x,r] \mapsto r^2\phi(x)$ continuous if we
set $\Phi(\TT)=0$. In the case $\mu =\mfP \lambda$ we call
$\lambda$ a lift of $\mu$. 

The first and most intuitive result on the distance $\msD_{1,4}$
induced by the Onsager operator $\bbK_{1,4}$ in \eqref{eq:dissHK1} is
the following formula, which we formulate as a definition first and
then show that it equals the distance $\msD_{1,4}$.

\begin{definition}\label{def:HKbyLift} The \emph{Hellinger--Kantorovich distance} on
  $\MMM(\Omega)$ is defined as
\begin{equation}
  \label{eq:DHKbyLift}
  \msHK(\mu_0,\mu_1)= \min\Bigset{ \msW_\mfC (\lambda_0,\lambda_1)
}{ \mfP\lambda_0 = \mu_0,\ \mfP\lambda_1 = \mu_1 }.
\end{equation}
\end{definition}

Before proving the identity $\msHK=\msHK_{1,4}=\msD_{1,4}$ in Section
\ref{se:equiDyn}, we collect some properties of $\msHK$.  First, we
emphasize that the projection $\mfP$ does not see the reservoirs at
$\TT$, hence the above formula already includes arbitrary reservoirs
according to Proposition \ref{pr:Reservoir}.

Next, let us remark that $\msHK$ satisfies an important scaling
invariance: Let $\vartheta:\CC\ti\CC\to \left]0,\infty\right[$ be a
Borel map and define the dilation function
$h_\vartheta:\CC{\ti}\CC\to\CC{\ti}\CC$ via
\begin{equation}
h_\vartheta(z_0,z_1) = \Big(\Big[x_0,\frac{r_0}{\vartheta(z_0,z_1)}\Big],
\Big[x_1,\frac{r_1}{\vartheta(z_0,z_1)}\Big]\Big)\quad
\text{for }z_i=[x_i,r_i].\label{eq:1}
\end{equation}
Then, given any transport plan $\gamma\in\MMM_2(\CC{\ti}\CC)$ we
define the dilated plan $\gamma_\vartheta =
({h_\vartheta})_{\#}(\vartheta^2\,\gamma) $ in $\MMM(\CC{\ti}\CC)$.
Letting $\lambda_i$ and $\lambda_i^\vartheta$ denote the marginals of
$\gamma$ and $\gamma_\vartheta$ we have that
\begin{equation}
  \label{eq:Scaling}
\iint_{\CC\times\CC} \msd_\mfC(z_0,z_1)^2\dd \gamma =
\iint_{\CC\times\CC} \msd_\mfC(z_0,z_1)^2 
\dd \gamma_\vartheta \quad\text{and}\quad
\mfP\lambda_i =\mfP\lambda_i^\vartheta.
\end{equation}
In particular, we can always assume that the transport plans $\gamma$
and the lifts $\lambda_i$ are probability measures, e.g.\ by setting
$\vartheta \equiv (\gamma(\CC{\ti}\CC))^{-1/2} $.

The main result of this section is the following structural theorem. 
For a full proof we  refer to
\cite{LiMiSa14?Theory},  where a more general case is considered.
In particular, there $\Omega$ is replaced by general complete geodesic
spaces. However, because of the strong relevance of 
part (v) for the subsequent applications, we present a full proof of
the identity $\msHK= \msD_{1,4}$ in Section \ref{se:equiDyn}. 

\begin{theorem}[Properties of $\msHK$]\label{th:properties}
The distance $\msHK:\MMM(\Omega)\ti\MMM(\Omega)\to[0,\infty[$
has the following properties:
\begin{itemize}
\item[i)] For each pair $\mu_0,\mu_1$ there exists an optimal pair
 $\lambda_0,\lambda_1$ of lifts;
\item[ii)] For all measures $\mu_0,\mu_1$ the upper bound 
$\msHK(\mu_0,\mu_1)^2\leq\mu_0(\Omega)+\mu_1(\Omega)$ is satisfied;
\item[iii)] $(\MMM(\Omega),\msHK)$ is a complete and separable metric
space;
\item[iv)] The topology induced by $\msHK$ coincides with the weak 
topology on $\MMM(\Omega)$;
\item[v)] The distance $\msHK$ is induced by the Onsager operator $\bbK_{1,4}$.
\end{itemize}
\end{theorem}

\subsubsection{Consistency of above formulas with distance of Dirac masses}
\label{sss:ConsiDiracMass}
We come back to the example of the optimal transport and absorption/desorption
of two point masses in Subsection \ref{ss:MassDHW} and discuss the 
consistency of the above formulas. Let $\mu_0=a_0\delta_{x_0}$ and 
$\mu_1= a_1\delta_{x_1}$ denote two Dirac masses such that $a_0,a_1>0$. 
We consider lifts $\lambda_0,\lambda_1\in\MMM_2(\CC)$
of the particular form
\[
\lambda_0 = \Big(\kappa+\frac{a_1}{r_1^2}\Big)\delta_{\TT}
+\frac{a_0}{r_0^2}\delta_{[x_0,r_0]},\quad\text{and}\quad
\lambda_1 = \Big(\kappa+\frac{a_0}{r_0^2}\Big)\delta_{\TT}
+\frac{a_1}{r_1^2}\delta_{[x_1,r_1]},
\]
where $\kappa\geq 0$ and $r_i>0$ are arbitrary but fixed constants. 
In particular, we have equal mass $\lambda_0(\CC)=\lambda_1(\CC)$ and 
$\mfP\lambda_i = \mu_i$, i.e., $\lambda_i$ is indeed a lift for
$\mu_i$. 

The possible transport plans $\gamma\in\MMM_2(\CC{\ti}\CC)$
are uniquely characterized by the value ${g}:=\gamma(
\{[x_0,r_0],[x_1,r_1]\})\in[0,\min\{a_0/r_0^2,a_1/r_1^2\}]$, 
where the interval boundaries correspond to complete 
absorption/desorption and complete transport.

Denoting $z_i = [x_i,r_i]$ we find
\begin{align*}
\iint_{\CC{\ti}\CC}\msd_{\mfC}(z_0,z_1)^2\dd\gamma(z_0,z_1)&
=\frac{a_0}{r_0^2} \msd_\mfC(z_0,\TT)^2+ \frac{a_1}{r_1^2}\msd_\mfC(z_1,\TT)^2\\
&\quad+{g}\big[\msd_\mfC(z_0,z_1)^2-\msd_\mfC(z_0,\TT)^2-\msd_\mfC(z_1,\TT)^2\big]\\
&=a_0 + a_1 - 2{g}r_0r_1 \trcos(|x_0{-}x_1|).
\end{align*}
To get the optimal cost we have to minimize with respect to
${g}\in[0,\min\{a_0/r_0^2,a_1/r_1^2\}]$: For $L:=|x_0{-}x_1|>\pi/2$
the optimal value is ${g}=0$, which corresponds to the pure Hellinger
reaction case. For $L=\pi/2$ any ${g}$ is possible giving a convex set
of optimal plans.  In fact, it will be shown in
Section~\ref{ss:AllDiracM} that any pair of lifts is optimal in this
case.  Hence, the case $L\geq\pi/2$ yields $\msW_\mfC
(\lambda_0,\lambda_1) =
\sqrt{a_0{+}a_1}=\msHK(a_0\delta_{x_0},a_1\delta_{x_1})$ as desired.
For $L<\pi/2$ we have to choose ${g}=\min\{a_0/r_0^2,a_1/r_1^2\}$,
i.e., the maximal value.  With this we obtain
\[
\msW_\mfC (\lambda_0,\lambda_1)^2 = a_0 + a_1 - 2{g}_*(r_0,r_1)\cos(L),
\]
where ${g}_*(r_0,r_1):=\min\{a_0\frac{r_1}{r_0},a_1\frac{r_0}{r_1}\}$.
In particular, different lifts $\lambda_i=\lambda_i(r_0,r_1)$
give different costs. However, an easy calculation shows that
for $r_1/r_0=\sqrt{a_1/a_0}$ an optimal value is achieved,
such that $\msW_\mfC (\lambda_0,\lambda_1)^2=a_0 + a_1 - 2\sqrt{a_0a_1}\cos(L)
=\msHK(a_0\delta_{x_0},a_1\delta_{x_1})^2$.

For calculating the distance $\msHK$ the particular choice of an
optimal lift is not important, but we will see in Section
\ref{ss:AllDiracM} that in the case $L=\pi/2$ different lifts may give
rise to different geodesic curves. Hence, we highlight here that even
in the trivial case $L<\pi/2$ there are many optimal lifts. E.g.\ for
$a_1=a_0>0$ any $\eta \in \MMM_2({[0,\infty[})$ with $\int_0^\infty
r^2 \dd \eta = a_0$ defines optimal lifts $\lambda_j =
\delta_{x_j}{\otimes}\eta$. 
 
\subsubsection{Logarithmic-entropy transport functional}
\label{sss:LET}

In this subsection we give the formula for the distance via a
minimization problem and discuss a few of its properties, in
particular its consistency with the distance of Dirac masses. 
We do this for the case of general positive $\alpha$ and $\beta$. 

Using the Boltzmann function $\FBoltz(\rho)= \rho \log \rho - \rho +1 \geq 0$
with $\FBoltz'(\rho)=\log \rho$ and  $\FBoltz(\rho)=0$ for $\rho=0$ we
define the Hellinger--Kantorovich functional for any $\mu_0,\mu_1 \in
\MMM(\Omega)$ as follows. For $\eta \in \MMM(\Omega{\ti} \Omega)$ we
define the marginals $\eta_j = \Pi^j_\# \eta$ and assume $\eta_0\ll
\mu_0$ and $\eta_1\ll \mu_1$ and define the Hellinger--Kantorovich
entropy-transport functional via
\[
\scrET_{\alpha,\beta}(\eta;\mu_0,\mu_1):= 
\frac4\beta \int_\Omega \FBoltz\Big( 
  \frac{\rmd \eta_0}{\rmd\mu_0}\Big) \dd\mu_0 +
\frac4\beta \int_\Omega \FBoltz\Big( 
  \frac{\rmd \eta_1}{\rmd\mu_1}\Big) \dd\mu_1 +
\int_{\Omega\ti\Omega} \msc_{\alpha,\beta}(|x_0{-}x_1|) \dd \eta,
\]
where the cost function $\msc_{\alpha,\beta}$ is given by 
\[
\msc_{\alpha,\beta}(L):=\left\{ \ba{cl}
-\frac8\beta \log\Big(\cos\big(\sqrt{\beta/(4\alpha)} 
 L\big) \Big) & \text{for } L < \pi
\sqrt{\alpha/\beta},\\ \infty & \text{for } L \geq \pi
\sqrt{\alpha/\beta}. \ea \right.
\]   
We see that $\scrET_{1,4}(\cdot;\mu_0,\mu_1)$  is a convex functional,
thus it is easy to find minimizers.  
The following characterization is proved in full detail in
\cite{LiMiSa14?Theory}. Here we will only motivate the construction
by giving some examples. 

\begin{theorem}[Characterization of $\msHK_{\alpha,\beta}$ via minimization]
\label{th:CharDHK}
For $\alpha,\beta>0$ the distance induced by the Onsager operator
$\bbK_{\alpha,\beta}$  is given as follows:
\[
\msHK_{\alpha,\beta}(\mu_0,\mu_1)^2= \msET_{\alpha,\beta}(\mu_0,\mu_1):=
\inf\Bigset{\scrET_{\alpha,\beta}(\eta;\mu_0,\mu_1) }{ \eta \in
  \MMM(\Omega{\ti}\Omega),~ \eta_j \ll \mu_j }.  
\]
For every pair $(\mu_0,\mu_1)$ at least one minimizer $\eta$ exists,
which we call a calibration measure for this pair. 

Moreover, an optimal calibration measure $\eta$ satisfies
for $\varrho_i:=\rmd\eta_i/\rmd\mu_i$ the following optimality
conditions
\[
\begin{split}
|x_0{-}x_1| < \pi\sqrt{\alpha/\beta} \quad 
&\text{for }\eta\text{-a.e.}~(x_0,x_1)\in\Omega{\times}\Omega,\\
\varrho_0(x_0)\varrho_1(x_1)\leq \cos\big(\sqrt{\beta/(4\alpha)}|x_0{-}x_1|\big)^2
\quad&\text{for }\mu_0\text{-a.e.}~x_0\in\Omega~\text{and}~\mu_1\text{-a.e.}~x_1\in\Omega,\\
\varrho_0(x_0)\varrho_1(x_1)=\cos\big(\sqrt{\beta/(4\alpha)}|x_0{-}x_1|\big)^2>0 \quad 
&\text{for }\eta\text{-a.e.}~(x_0,x_1)\in\Omega{\times}\Omega.
\end{split}
\]
\end{theorem}

In the framework of this paper, the relevance of this new
characterization of $\msD_{1,4}=\msHK$ is that the minimization of
$\scrET_{1,4}$ is much simpler than the characterization of
$\msHK$ in terms of lifts to the cone space. Finding the optimal
lifts and the calculating the optimal transport on the cone space is
certainly more involved. In \cite{LiMiSa14?Theory}, it is shown that
$\scrET_{1,4}$ has a much stronger intrinsic value and it proves an
essential tool for establishing the results in Theorem
\ref{th:properties}.

\begin{example}[Mass splitting, part 2]\label{ex:MassSplit2}
We return to Example \ref{ex:MassSplit1} where we calculated the
distance between 
\[
\mu_0 = a_0\delta_{x_0} \quad \text{and} \quad \mu_1 = a_1\delta_{x_0}
+ b_1 \delta_{x_1}
\]
with $L=|x_0{-}x_1|<\pi$ and $\alpha=1$, $\beta=4$. We show that the 
formulation in Theorem \ref{th:CharDHK} indeed gives the same cost.
Since the marginals $\eta_j$ have to have a density with respect 
to $\mu_j$ and since $\eta_0$ and $\eta_1$ must have equal mass, 
we consider
\[
\eta_0 = e_0 \delta_{x_0} \quad \text{and} \quad \eta_1 =
(e_0{-}e_1)\delta_{x_0} + e_1\delta_{x_1}
\]
with $e_0,e_1\geq 0$. Using the formula in Theorem \ref{th:CharDHK}
yields for $\msET=\msET_{1,4}$ and $\msc=\msc_{1,4}$
\[
\msET(\mu_0,\mu_1) = \inf\set{ \FBoltz(\tfrac{e_0}{a_0})a_0+
\FBoltz(\tfrac{e_0{-}e_1}{a_1}) a_1 + \FBoltz(\tfrac{e_1}{b_1})b_1  +
e_1\msc(L)}{ e_0\geq e_1\geq 0}.
\]
This infimum can be evaluated explicitly and we obtain 
\[
\msET(\mu_0,\mu_1) = \msHK(\mu_0,\mu_1)^2=  a_0+a_1+b_1 - 2
\sqrt{a_0(a_1{+}b_1\cos(|x_0{-}x_1|)^2) },
\]
which is the  same as in Example \ref{ex:MassSplit1}.
\end{example}

\subsubsection{Reduction to special lifts}
\label{sss:SpecialLifts}
The characterization of the Hellinger--Kantorovich distance in terms
of the logarithmic-entropy transport functional gives rise to another
helpful property: To calculate $\msHK(\mu_0,\mu_1)$ it is
sufficient to consider lifts $\lambda_i$ of a special form
only. Indeed, assume that $\eta\in\MMM(\Omega{\times}\Omega)$ is a
minimizer of $\scrET_{1,4}$ for given $\mu_0$ and $\mu_1$ and consider
for $\eta_i = \Pi_\#^i\eta$ the Lebesgue decomposition $\mu_i =
\sigma_i\eta_i+\mu_i^\perp$. Then, the transport plan
$\gamma_\eta\in\MMM(\CC{\times}\CC)$ defined by
\[
\begin{split}
\gamma_\eta(\rmd z_0,\rmd z_1)&= \delta_{\sqrt{\sigma_0(x_0)}}(\rmd r_0)
\delta_{\sqrt{\sigma_1(x_1)}}(\rmd r_1)\eta(\rmd x_0,\rmd x_1)
\\[0.5em]
&\quad+\delta_1(\rmd r_0)\mu_0^\perp(\rmd x_0)\delta_\TT(\rmd z_1 )
+\delta_{\TT}(\rmd z_0)\delta_1(\rmd r_1)\mu_1^\perp(\rmd x_1)
\end{split}
\]
and the associated lifts $\lambda_i = \Pi_\#^i\gamma_\eta$ are optimal
in the Definition \ref{def:HKbyLift} for $\msHK$, see
\cite[Thm.~7.21]{LiMiSa14?Theory} for the proof.
In particular, we can restrict the analysis to lifts of $\mu_j$
 characterized by a single positive function $\wh r_j>0$ on $\Omega$, namely 
\begin{align*}
&\mfL(\mu,\wh r,\kappa) = \kappa \delta_{\TT}+ \frac1{\wh r(x)^2}
\delta_{\wh r(x)}(\rmd r) \mu(\rmd x), 
\quad\text{such that} \\
&\int_{\CC} \Phi(z)\dd \mfL(\mu,\wh
r,\kappa) = \kappa \Phi(\TT) + 
\int_\Omega \frac{\Phi([x,\wh r(x)])}{\wh r(x)^2} \dd \mu \ \text{ for
  all } \Phi\in\rmC^0(\CC).
\end{align*}
We collect this observation in the following result.

\begin{proposition}[$\msHK$ via special lifts]\label{pr:DHKviaSpLi}
We have the equivalent characterization
\begin{equation}
  \label{eq:HKviaSpLift}
\msHK(\mu_0,\mu_1)=\min\Bigset{\msW_\mfC\big(\mfL(\mu_0,\wh r_0,\kappa_0),
  \mfL(\mu_1,\wh r_1,\kappa_1)\big)} { \kappa_j\geq 0, \ \wh r_j >0 }.
\end{equation}
Moreover, it is sufficient to consider transport plans $\gamma\in\MMM_2(\CC{\ti}\CC)$ 
of the form
\[
\begin{split}
\gamma &= \delta_{\widehat{r}_0(x_0)}(\rmd r_0)\eta_0(\rmd x_0)\delta_\TT(\rmd z_1 )
+\delta_{\TT}(\rmd z_0)\delta_{\widehat{r}_1(x_1)}(\rmd r_1)\eta_1(\rmd x_1)\\[0.5em]
&\quad+\delta_{\widehat{r}_0(x_0)}(\rmd r_0)\delta_{\widehat{r}_1(x_1)}(\rmd r_1)
\eta(\rmd x_0,\rmd x_1)
\end{split}
\]
for positive functions $\wh r_i:\Omega\to \left]0,\infty\right[$
and measures $\eta_i\in\MMM(\Omega)$ and $\eta\in\MMM(\Omega\times\Omega)$.
\end{proposition}

Using the definition of $\msW_\mfC$ in terms of $\msd_\mfC$ and
the form of the lifts, the functional in \eqref{eq:HKviaSpLift} can be
written as 
\[
\scrD(\eta,\wh r_0,\wh r_1; \mu_0,\mu_1):= \mu_0(\Omega) +
\mu_1(\Omega) - \int_{\Omega\ti \Omega} 2\wh r_0(x_0)\wh
r_1(x_1)\cos_{\pi/2}|x_0{-}x_1| \dd\eta(x_0,x_1), 
\]
and the following characterization of $\msHK$ follows:
\begin{equation}
  \label{eq:HKby-wh-r}
  \msHK(\mu_0,\mu_1)=\min\Bigset{ \scrD(\eta,\wh r_0,\wh r_1;
    \mu_0,\mu_1)}{ \eta \in \MMM(\Omega{\ti}\Omega), \Pi^j_\#\eta =
    \wh r_j^2 \mu_j + \mu_j^\bot }. 
\end{equation}

We emphasize that not all optimal transport plans are of the form depicted
in Proposition~\ref{pr:DHKviaSpLi}. In particular, using again the example
of two mass-points we show in Section~\ref{ss:AllDiracM} that in the
case of the critical distance $|x_0{-}x_1|$ lifts are quite arbitrary.

\subsubsection{Recovering the Hellinger and Wasserstein-Kantorovich distances}
The log-entropy formulation of the Hellinger--Kantorovich distance 
is well suited to pass to the limits $\alpha\to0$ or 
$\beta\to0$.

Since apart from the prefactor $1/\beta$ the functional
only depends on $\beta/\alpha$, we can set $\alpha=1$ and consider the
case $\beta\to 0$. For the cost functional we obtain the expansion 
\[
\msc_{1,\beta}(x_0,x_1) = |x_1{-}x_0|^2 + O(\beta)
\]
uniformly on $\Omega\ti \Omega$, which is compact. Hence the linear
transport functional converges to the Kantorovich functional for the
usual Euclidian cost function. Simultaneous the entropic terms blow
up, which means that in the limit $\beta =0$, we obtain the condition
$\eta_j=\mu_j$.  Thus, we expect to obtain the Wasserstein distance in
the limit, i.e.\ $\msHK_{1,0}(\mu_0,\mu_1)= \msW(\mu_0,\mu_1)$. 

Keeping $\beta=4$ fixed and considering $\alpha\to 0$ we obtain 
\[
\msc_{\alpha, 4}(x_0,x_1) \to \left\{ \ba{cl} 0& \text{for }
  x_0=x_1,\\ \infty&\text{for }x_0\neq x_1. \ea \right. 
\] 
Thus, optimal calibration measures for $\alpha=0$ will have support on
the diagonal $\set{(x,x)\in \Omega{\ti}\Omega} {x\in \Omega}$, such
that the transport cost equals 0 and that
$\nu:=\eta_0=\eta_1$. Minimizing the sum of the two entropic terms
with respect to $\nu$ we obtain the unique solution $\nu$ from the
optimality condition $\frac{\rmd \nu}{\rmd\mu_0} \frac{\rmd
  \nu}{\rmd\mu_1} \equiv 1$ and we find
$\msHK_{0,4}(\mu_0,\mu_1)=\msD_\text{Hell}(\mu_0,\mu_1)
=\|\sqrt{\mu_1}{-}\sqrt{\mu_0}\|_{\rmL^2}$.

\subsection{Geodesic curves induced by optimal transport plans} 
\label{ss:GeodCurves}
Let $\mu_0$, $\mu_1\in\MMM(\Omega)$ be two given measures. The
geodesic curves with respect to the Hellinger--Kantorovich distance
$\msHK$ are induced by the geodesic curves in the underlying cone space. 

More precisely, the construction of the geodesic curve $s\mapsto \mu(s)$ 
is based on the geodesic interpolator $Z$ defined in
\eqref{eq:GeodInterpol}: Let $\lambda_0\in\MMM_2(\CC)$ 
and $\lambda_1\in\MMM_2(\CC)$ be optimal lifts for $\mu_0$ and $\mu_1$,
respectively, and let $\gamma\in\MMM_2(\CC{\ti}\CC)$ be the
associated optimal transport plan. Then, a geodesic curve 
$\mu(s) = \calG(s; \mu_0,\mu_1)$ is obtained via
the projection of the geodesic curve for $\lambda_0$ and $\lambda_1$ 
in $\MMM_2(\CC)$ via
\begin{equation}\label{eq:geodesic}
\mu(s)=\calG(s;\mu_0,\mu_1):=\mfP \lambda(s)\quad \text{with}
\quad \lambda(s)=Z(s;\cdot,\cdot)_{\#}\gamma.
\end{equation}
Note that since the optimal transport plan $\gamma$ is not necessarily
unique, the geodesics in $\MMM(\Omega)$ are also not necessarily unique: 

\begin{example}
\begin{itemize}
\item[(i)]On $\Omega=\left]-2,2\right[^2$ we consider 
the measure $\mu_0=\delta_{(-1,0)}+\delta_{(1,0)}$ and 
$\mu_1$ is the line measure concentrated in $\{0\}\times\left]{-}1,1\right[$.
Due to the high symmetry of the problem, it is easy to see that there are 
infinitely many optimal transport plans, which give rise to different geodesic curves.

\item[(ii)] Consider case of two mass points $\mu_i=a_i\delta_{x_i}$ with $|x_0{-}x_1|=\pi/2$.
It is easy to see that in this case $\mu(s)=a(s)\delta_{x(s)}$ with 
$x(s)=(1{-}\rho(s))x_0+\rho(s)x_1$ and $a(s)$ and $\rho(s)$ as in Theorem~\ref{th:MassPointCost}
and $\wt\mu(s)=(1{-}s)^2a_0\delta_{x_0} + s^2a_1\delta_{x_1}$ are both geodesic curves.
However, the situation is even more complicated since even along a geodesic curve...
\end{itemize}
\end{example}

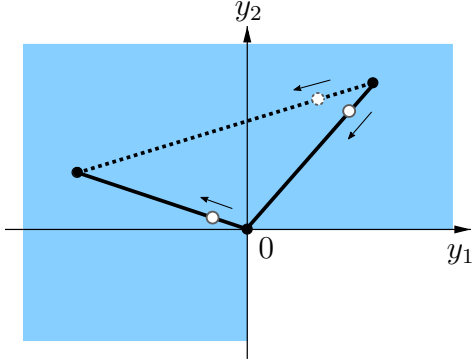
\begin{figure}
 \centerline{\unitlength1.5em
\includegraphics[width=10\unitlength]{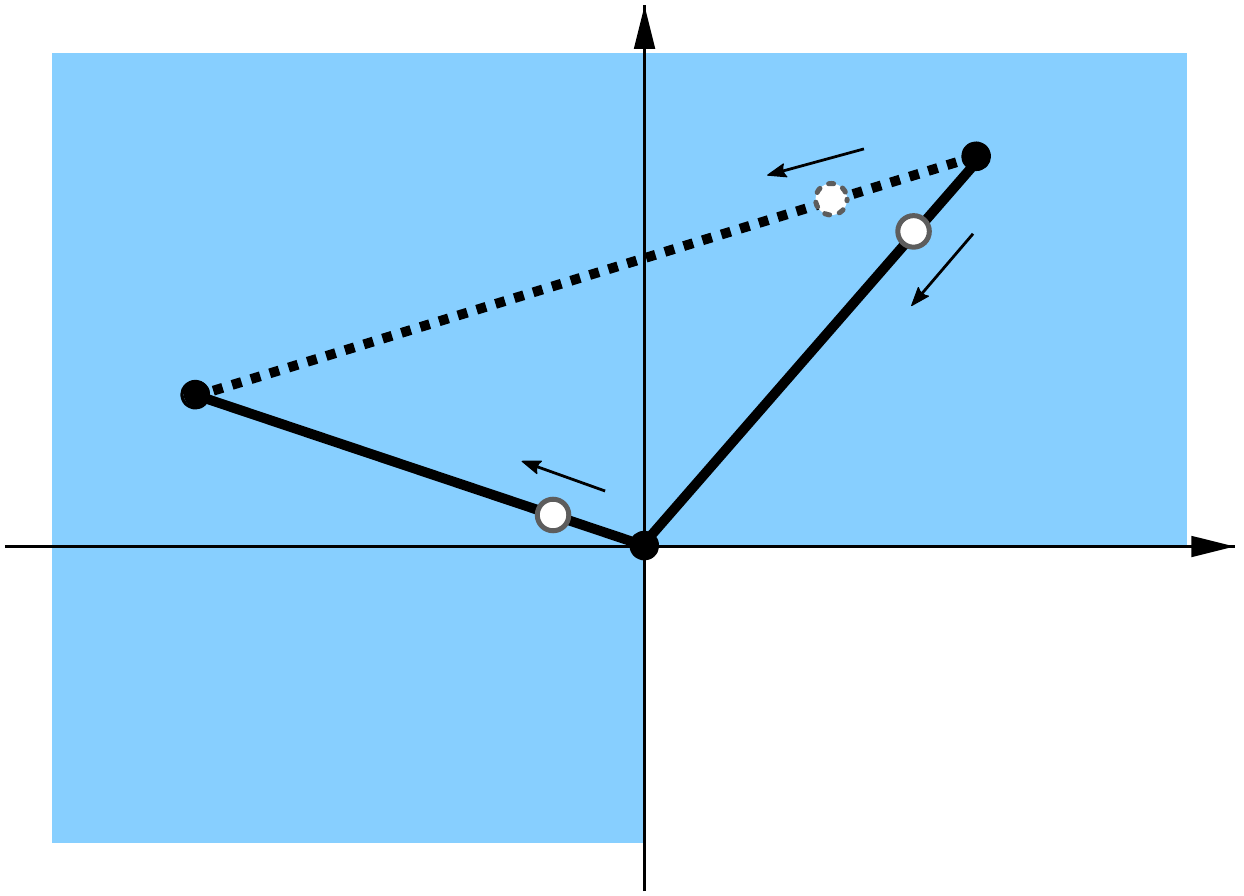}
\begin{picture}(0,0)
\put(-5.8,2.4){
\put(1,-0.2){$0$}
\put(5,-0.2){$y_1$} \put(0.5,5.0){$y_2$}
}
\end{picture}
}
\caption{Cone geodesic (dotted) for $z_0=[x_0,\sqrt{a_0}]$ and $z_1=[x_1,\sqrt{a_1}]$ 
compared to Hellinger--Kantorovich geodesic (solid) for $\mu_0=a_0\delta_{x_0}$
and $\mu_1=a_1\delta_{x_1}$ in the case $|x_0{-}x_1|>\pi/2$. The Hellinger--Kantorovich geodesic
consists of two parts: one part is going to the reservoir (absorption), while the other one is simultaneously
coming from the reservoir (generation).}
\label{fig:ConeVSHKGeodesics} 
\end{figure} 

\begin{theorem}\label{thm:geodesics}
The curve $s\mapsto\mu(s)$ defined in \eqref{eq:geodesic} is a
constant-speed geodesic with respect to the Hellinger--Kantorovich distance $\msHK$, 
i.e.,
\[
\msHK(\mu(s),\mu(t)) =
 |t{-}s|\msHK(\mu_0,\mu_1)\quad\text{for all } 0\leq s< t\leq1.
\]
\end{theorem}
\begin{proof}
Fix $0\leq s<t\leq 1$ and let $\gamma\in\MMM_2(\CC{\ti}\CC)$ denote
the optimal transport plan. We define the map $\Pi_{st}:\CC^2\to\CC^2$
via $\Pi_{st}(z_0,z_1) = (Z(s;z_0,z_1),Z(t;z_0,z_1))$ and introduce
the transport plan $\gamma_{st}=(\Pi_{st})_{\#}\gamma$ whose marginals
are given by $\lambda(s)$ and $\lambda(t)$, respectively. In particular,
we have the upper estimate
\[
\msHK\big(\mu(s),\mu(t)\big)\leq \msW_{\mfC}\big(\lambda(s),\lambda(t)\big)
\leq \Big(\iint_{\CC\times\CC} \msd_{\mfC}(z_0,z_1)^2\dd\gamma_{st}\Big)^{1/2}.
\]
However, using the definition of the $\gamma_{st}$ and that
$Z$ is the geodesic interpolator in $\CC$ we obtain
\begin{equation}\label{eq:ConSpeedIneq}
\msHK\big(\mu(s),\mu(t)\big)\leq |s{-}t| \msHK(\mu_0,\mu_1)
\quad\text{for all }0\leq s <t \leq 1.
\end{equation}
To see that actually equality holds we use the triangle inequality
and \eqref{eq:ConSpeedIneq} to find
\begin{align*}
\msHK(\mu_0,\mu_1)&\leq  \msHK(\mu_0,\mu_s)
 + \msHK(\mu_s,\mu_t) + \msHK(\mu_t,\mu_1)\\
&\leq \big(s + (t{-}s) + (1{-}t)\big) \msHK(\mu_0,\mu_1)
=\msHK(\mu_0,\mu_1).
\end{align*}
Thus, all inequalities are equalities, which proves theorem.
\end{proof}

\section{Equivalence to the dynamical formulation}
\label{se:equiDyn}
In this subsection we provide the proof of part (v) of Theorem
\ref{th:properties} and show the equivalence of the two definitions of
the Hellinger--Kantorovich distance, namely the formulation via lifts
and optimal transport on the cone space and the dynamical formulation
given by the Onsager operator $\bbK:=\bbK_{1,4}$, i.e.\
$\msD_{1,4}=\msHK$ with $\msD_{1,4}$ from \eqref{eq:dissHK1} and
$\msHK$ from \eqref{eq:DHKbyLift}.

The proof is based on the characterization of absolutely continuous
curves and their metric derivative with respect to $\msHK$. In
particular, we show in Theorem~\ref{thm:timeLower} that each
absolutely continuous curve whose metric derivative is square
integrable satisfies the modified continuity equation in the
definition of $\msD_{1,4}$ in the distributional sense for a suitable
vector and scalar field $\Xi$ and $\xi$.  Moreover, the
$\rmL^2(\rmd\mu)$-norms of $\Xi$ and $\xi$ provide a lower bound for
the metric derivative.

Vice versa we prove in Theorem~\ref{thm:timeUpper} that a continuous
solution $t\mapsto \mu(t)$ of the modified continuity equation for
given vector and scalar fields $\Xi$ and $\xi$ is absolutely
continuous with respect to ${\msHK}$ and the $\rmL^2(\rmd\mu)$-norms
give an upper estimate for the metric derivative.

Finally, Theorem \ref{th:properties}(v) is proven at the
end of this subsection. 

We recall that a curve
$[0,1]\ni t\mapsto u(t)$ in a metric space $(\bfY,\msD)$ is called 
absolutely continuous if  there exists a function $m\in\rmL^1(0,1)$
such that
\begin{equation}\label{eq:defAC}
\msD(u(s),u(t))\leq \int_s^t m(r)\dd r\qquad \text{for all }0\leq s<t\leq 1.
\end{equation}
We write $u\in \rmA\rmC^p(0,1;(\bfY,\msD))$ if $m\in\rmL^p(0,1)$ for $p\in[1,\infty]$.
Moreover, among all possible choices for $m$ there exists a minimal one, which
is given by the metric derivative, see e.g.\ \cite[Sect.\ 1.1]{AmGiSa05GFMS}
\begin{equation}\label{eq:defMetricDeriv}
|\dot u|_\msD(t) := \lim_{s\to t}\frac{\msD(u(t),u(s))}{|t-s|}.
\end{equation}
In particular, for any $u\in\rmA\rmC^p(0,1;(\bfY,\msD))$ the metric derivative
exists for a.a.\ $t\in\left]0,1\right[$ and satisfies $|\dot u|_\msD\in \rmL^p(0,1)$
 as well as $|\dot u|_\msD\leq m$ a.e.\ in $[0,1]$ for all $m$ in \eqref{eq:defAC}.

We start with a result for the regular case, i.e.\ the vector and scalar fields
$\Xi$ and $\xi$ are sufficiently smooth. The proof of the following result can be
found in \cite{Mani07PRUR} where representation formulas for solutions
of the inhomogeneous continuity equation 
\begin{equation}\label{eq:inhomCE}
\tfrac{\dd }{\dd t}\mu +\div(\Xi \mu) = 4\xi\mu
\end{equation}
based on dynamic plans are proved. We will
briefly recall these results, however, since the cone structure
did not play a role in \cite{Mani07PRUR}
we will reinterpret the results in our setting. In the following we understand
weak convergence in the space of measures as convergence against bounded and 
continuous functions. Moreover, a curve $s\mapsto \mu(s)\in\MMM(\Omega)$ is called
weakly continuous if and only if $\mu(s)$ weakly converges to $\mu(t)$ in $\MMM(\Omega)$ for $s\to t$.

\begin{proposition}[\cite{Mani07PRUR}, Prop.\ 3.6]\label{p:represCE}
Assume that $\Xi\in \rmL^1(0,T;\rmW^{1,\infty}(\Omega;\R^d))$ and 
$\xi\in \rmC([0,T]\times\Omega)$ is locally Lipschitz with respect 
to the spatial variable. Then, for any $\mu_0\in\MMM(\Omega)$
there exists a unique, weakly continuous solution 
$t\mapsto \mu(t)$ of \eqref{eq:inhomCE} with $\mu(0)=\mu_0$.

Moreover, for an arbitrary lift $\lambda_0\in\MMM(\CC)$ of $\mu_0$ the curve defined
by 
\begin{equation}\label{eq:ReprMani}
\lambda(t) = [X(t;\cdot), R(t;\cdot)]_{\#}\lambda_0 \in\MMM(\CC),
\end{equation}
where $t\mapsto (X(t;x),R(t;x,r))$ is the solution of the 
ODE system
\[
\dot X(t;x) = \Xi\big(t,X(t;x)\big),\qquad \dot R(t;x,r) = 2 \xi\big(t,X(t;x)\big)R(t;x,r)
\]
with initial conditions $X(0;x) = x$ and $R(0; x, r) = r$ is a lift of
$\mu(t)$.
\end{proposition}

Note that we can solve the equation for $R$ explicitly
and obtain
\[
R(t,x,r) = r\exp\Big(2\int_0^t\xi(s,X(s,x))\dd s\Big).
\]

It is well-known that if $\Xi$ fails to satisfy the 
regularity properties of Proposition \ref{p:represCE}
nothing guarantees uniqueness of the characteristics
$t\mapsto(X(t),R(t))$ and formula \eqref{eq:ReprMani}
does not hold. To overcome this problem probability measures 
concentrated on entire trajectories in the underlying space $\CC$
are introduced, see \cite{Lisi07CACC} and \cite[Sect.\ 8.2]{AmGiSa05GFMS}.

More precisely, we call $\bfpi\in \Prob(\rmC([0,1];\CC))$
a dynamic plan if it is concentrated on absolutely continuous
curves $\wh z\in\calA:=\AC^2([0,1];(\CC,\msd_\mfC))$ and if it satisfies
\[
\int_{\calA}\Big(\int_0^1 \big|\dot {\wh z}\big|_{\mfC}(t)^2\dd t\Big)\dd\bfpi(\wh{z})<\infty
\]
with $|\dot{\wh z}|_{\mfC}$ denoting the metric derivative 
with respect to the cone distance $\msd_\mfC$, see \eqref{eq:defMetricDeriv}. 
Note that any continuous curve $t\mapsto \wh z(t)=[\wh x(t),\wh r(t)]$
with $t\in[0,1]$ satisfies $\wh r\in \rmC([0,1])$ with values in $[0,\infty[$. 
Thus, the set $O_{\wh{r}} = \wh r^{-1}(]0,\infty[)\subset[0,1]$
is open and the restriction of $\wh x$ to $O_{\wh r}$ is also 
continuous. The following lemma gives a characterization of
the absolutely continuous curves in $\CC$ and their metric derivative.
It is proven in \cite{LiMiSa14?Theory}.

\begin{lemma}\label{lem:AbsContInCone}
A curve $t\mapsto\widehat{z}(t)=[\wh{x}(t),\wh r(t)]\in\CC$ 
satisfies $\wh z\in \rmA\rmC^p([0,1];\CC)$ if and only if
\[
\dot r\in \rmL^p(0,1)\quad\text{and}\quad \wh r \big|\dot {\wh x}\big| \in\rmL^p(O_{\wh{r}})
\quad\text{for } O_{\wh r}:= \wh{r}^{-1}(]0,\infty[).
\]
In particular, the metric time derivative is given via
\[
\big|\dot{\wh z}\big|_{\mfC}(t)^2 = \dot{\wh r}(t)^2 
+ \wh{r}(t)^2 \big|\dot{\wh x}(t)|^2\quad\text{for }t\in O_{\wh r}\quad
\text{and $\big|\dot{\wh z}\big|_{\mfC}(t)=0$ otherwise.}
\]
\end{lemma}

For $t\in[0,1]$ we denote by $e_t:\rmC([0,1];\CC)\to\CC$ 
the evaluation map given for $\wh z\in\rmC([0,1];\CC)$ by 
$e_t(\wh{z})=\wh{z}(t)$. With a 
dynamic plan $\bfpi\in\Prob(\rmC([0,1];\CC))$ we associate 
the curve $t\mapsto\lambda(t):=(e_t)_{\#}\bfpi$ which belongs 
to $\AC^2([0,1];(\MMM(\CC),\msW_\mfC))$, see \cite[Thm.~4]{Lisi07CACC}.
Moreover, from the 1-Lipschitz continuity of the projection
$\mfP:\MMM(\CC)\to\MMM(\Omega)$ it follows that the curve
$t\mapsto\mu(t):=\mfP\lambda(t)$ belongs to
$\AC^2([0,1];(\MMM(\Omega),\msHK))$ and the metric derivative of $\mu$ with respect to
$\msHK$ satisfies
\begin{equation}\label{eq:estDynPlan}
|\dot\mu|_{\msHK}(t)^2\leq\int_{\calA}|\dot{\wh z}|_{\mfC}(t)^2\dd\bfpi(\wh{z}).
\end{equation}
The following theorem shows that for every absolutely
continuous curve in $(\MMM(\Omega),{\msHK})$ a dynamic plan 
$\bfpi$ exists such that $\mu$ is induced by $\bfpi$ in the above sense
and equality holds in \eqref{eq:estDynPlan}. The proof is based on 
an extension of \cite[Thm.\,5]{Lisi07CACC} and can be found in 
\cite[Thm.\ 8.4]{LiMiSa14?Theory}.

\begin{theorem}\label{thm:dynPlan}
Let $\mu\in\AC^2([0,1];(\MMM(\Omega),\msHK))$ be given. Then, there exists a
dynamic plan $\bfpi\in\Prob(\rmC([0,1];\CC))$ such that
$\mu(t)=\mfP((e_t)_{\#}\bfpi)$ and
\begin{equation}
  \label{eq:dynPlanIdent}
|\dot\mu|_{\msHK}(t)^2=\int_{\calA}|\dot{\wh{z}}|_{\mfC}(t)^2\,\dd\bfpi(\wh{z})
 \ \text{ for a.a. }t\in [0,1].
\end{equation}
\end{theorem}

Using this result we can also show that all geodesic curves for
the Hellinger-Kantorovich distance are given by projections of
geodesic curves in $\calM_2(\CC)$, i.e.\ all geodesic curves have the
representation \eqref{eq:geodesic}.

\begin{corollary}[Representation of all geodesic curves]\label{co:ConeGeodesic}
Let $[0,1]\ni s\mapsto\mu(s)$ be a geodesic curve
and $\bfpi$ the dynamic plan from Theorem~\ref{thm:dynPlan}.
Then, $s\mapsto \lambda(s) = (e_s)_{\#}\bfpi$ is a geodesic
curve in $\Prob_2(\CC)$ with respect to $\msW_\mfC$. 

In particular,
all geodesic curves in $(\calM(\Omega),\msHK)$ are given by an optimal
plan $\gamma$ for optimal lifts of $\mu(0)$ and $\mu(1)$ in the form
\eqref{eq:geodesic}.
\end{corollary}
\begin{proof}
For $0\leq s<t\leq 1$, we have the elementary
estimates
\begin{align*}
\msW_\mfC(\lambda(s),\lambda(t))^2&\leq\int_{\CC^2}
\msd_\mfC(z_0,z_1)^2\dd(e_s,e_t)_{\#}\bfpi  
= \int_{\calA}\msd_\mfC(\widehat{z}(s),\widehat{z}(t))^2\dd\bfpi\\
&\leq \int_{\calA}\bigg(\int_s^t|\dot{\widehat{z}}|_{\mfC}\dd r\bigg)^2\dd\bfpi
\leq (t{-}s)\int_s^t\int_{\calA}|\dot{\widehat{z}}|_\mfC^2\dd\bfpi\dd r,
\end{align*}
where we have used H\"older's inequality. Since $\mu$ is a geodesic curve,
we have $|\dot\mu|_{\msHK}\equiv\msHK(\mu(0),\mu(1))$ and hence, with
\eqref{eq:dynPlanIdent} we have
\[
\msW_\mfC(\lambda(s),\lambda(t))\leq (t{-}s)\msHK(\mu(0),\mu(1))
\leq(t{-}s)\msW_\mfC(\lambda(0),\lambda(1)).
\]
Arguing as in the proof of Theorem~\ref{thm:geodesics} shows that
$s\mapsto\lambda(s)$ is a geodesic curve. In particular,
all inequalities above are equalities.

>From the dynamic plan $\bfpi$ we immediately find the optimal
transport plan
$\gamma:= \big((e_0),(e_1)\big)_\# \bfpi$ between the optimal lifts
$\lambda(0)$ and $\lambda(1)$, such that $\mu(s)=\mfP \lambda(s)$.
\end{proof}

The following theorem shows that for every curve
$\mu\in\rmA\rmC^2(0,1;(\MMM(\Omega),\msHK))$ we can find a vector and
a scalar field $\Xi$ and $\xi$ such that the continuity equation in
\eqref{eq:inhomCE} is satisfied. Moreover, the $\rmL^2$-norm of
$(\Xi,\xi)$ with respect to $\mu(t)$ provides a lower bound for the
metric time derivative of $\mu$.

\begin{theorem}\label{thm:timeLower}
Let $\mu\in\AC^2([0,1];(\MMM(\Omega),\msHK))$ be given. Then, there exists
a Borel vector field $(\Xi,\xi):[0,1]\ti \Omega \to \R^{d+1}$
such that the continuity equation \eqref{eq:inhomCE} is satisfied
and
\[
\int_\Omega\Big[|\Xi(t,x)|^2
+4|\xi(t,x)|^2\Big] \dd \mu(t) \leq |\dot\mu|_{\msHK}(t)^2
\quad \text{for a.e. }t\in[0,1].
\]
\end{theorem}
\begin{proof}
Let $\bfpi\in\Prob(\rmC([0,1];\CC))$ be a dynamic plan
representing $\mu$ according to Theorem \ref{thm:dynPlan}.
We denote the lift $\lambda(t) = (e_t)_{\#}\bfpi$. Due to
the Disintegration Theorem, see \cite[Thm.\ 5.3.1]{AmGiSa05GFMS}, 
there exists a family of probability measures  $\pi_z(t)\in\Prob(\rmC([0,1];\CC))$ 
for $\lambda$-a.e.\ $z\in\CC$ and each $t\in[0,1]$.
Moreover, $\pi_z(t)$ is concentrated on the subset $\calA_z(t):=\{\wh{z}\in\calA:\,\wh{z}(t) = z\}$
and for every $F\in\rmL^1(\rmC([0,1];\CC);\bfpi)$ we have
\[
\int_{\calA} F(\wh{z})\dd\bfpi(\wh{z}) = 
\int_{\CC}\Big(\int_{\calA_z(t)} F(\wh{z})\dd\pi_{z}(t)
\Big)\dd \lambda(t).
\]
For $t\in[0,1]$ and $z=[x,r]\in\CC\setminus\{\TT\}$ we define the vector fields
\begin{align*}
\wt\Xi(t,z) = \int_{\calA_z(t)} \dot {\wh x}(t)\dd\pi_{z}(t)\quad\text{and}\quad
\wt\xi(t,z) = 
\frac12\int_{\calA_z(t)} \frac{\dot {\wh r}(t)}{\wh r (r)}\dd\pi_{z}(t),
\end{align*}
while for $z=\TT$ we set $\wt\xi(t,z)=\wt\Xi(t,z) = 0$.
Due to Jensen's inequality we have the estimate
\begin{align*}
\int_{\CC} \Big[ \big|\wt \Xi(t,z)\big|^2 + 4\big|\wt \xi(t,z)\big|^2\Big]r^2\dd \lambda(t)
&\leq \int_{\CC} \int_{\calA_z(t)}\Big[ \wh{r}(t)^2\big|\dot{\wh x}(t)\big|^2 
+ \big|\dot{\wh r}(t)\big|^2\Big]\dd \pi_z(t)\dd\lambda(t)\\
& = \int_{\calA} |\dot{\wh z}|_{\mfC}(t)^2 \dd \bfpi = |\dot \mu|_{\msHK}(t)^2.
\end{align*}
To obtain vector fields $\Xi$ and $\xi$ on $\Omega$ we 
employ the Disintegration Theorem for $r^2\lambda$ and $\mu=\Pi_{\#}(r^2\lambda)$
to obtain a family of probability measures $\nu_x$ concentrated 
on $[0,\infty[$ and such that $r^2\lambda =\nu_x\mu$. 
Using again Jensen's inequality we easily check that the fields $\Xi(t,x) 
=\int_{[0,\infty[}\wt\Xi(t,z)\dd \nu_x$ and $\xi(t,x) 
=\int_{[0,\infty[}\wt\xi(t,z)\dd \nu_x$ satisfy
\[
\int_\Omega\Big[|\Xi(t,x)|^2
+4|\xi(t,x)|^2\Big] \dd \mu(t) \leq |\dot\mu|_{\msHK}(t)^2
\quad \text{for a.e. }t\in[0,1].
\]

It remains to show that the continuity equation
\eqref{eq:inhomCE} is satisfied.
For this, we choose a test function of the form
$\varphi(x,t)=\eta(x)\psi(t)$ with $\eta$ and $\psi$
Lipschitz and bounded with compact support in $\Omega$ and $\left]0,1\right[$,
respectively.
We compute
\begin{align*}
\int_0^1\int_\Omega\eta(x)\dot\psi(t)\dd\mu(t)\dd t
&=\int_{\calA}\int_0^1\dot\psi(t)\eta\big(\wh x(t)\big)\,\wh r(t)^2\dd t \dd\bfpi\\
&=-\int_{\calA}\int_0^1
\psi(t)\Big[\nabla\eta\big(x(t)\big)\dot x(t) +4\eta(x(t))\frac{\dot r(t)}{2\wh{r}(t)}\Big]r(t)^2\dd t\dd\bfpi\\
&= -\int_{\CC}\int_0^1\psi(t)\big[\nabla\eta(x)\cdot\wt\Xi(t,z) +\eta(x)\wt\xi(t,z)\big]r^2\dd \lambda(t)\dd t\\
&= -\int_{\Omega}\int_0^1\psi(t)\big[\nabla\eta(x)\cdot\Xi(t,x) +\eta(x)\xi(t,x)\big]\dd \mu(t)\dd t.
\end{align*}
Thus, the continuity equation is satisfied in the distributional sense.
\end{proof}

Next, we show the reverse implication.
\begin{theorem}\label{thm:timeUpper}
Let $t\mapsto \mu(t)$ be a narrowly continuous curve
in $\MMM(\Omega)$ and suppose that there exists a Borel
vector field $\Xi:[0,1]\times\Omega\to\R^d$ and a scalar field
$\xi:[0,1]\times\Omega\to\R$ satisfying $(\Xi(t),\xi(t))\in\rmL^2(\dd\mu(t);\R^{d+1})$
\[
\int_0^1\int_\Omega\Big(|\Xi(t,x)|^2 + 4|\xi(t,x)|^2\Big)\dd \mu(t)\,\dd t<\infty
\]
such that the continuity equation \eqref{eq:inhomCE} is satisfied.
Then, $\mu\in\AC^2([0,1];(\MMM(\Omega),{\msHK}))$ and
\begin{equation}\label{eq:lowMTD}
 |\dot\mu|_{{{\msHK}}}(t)^2 \leq \int_\Omega\Big(|\Xi(t,x)|^2+4|\xi(t,x)|^2\Big)
\dd\mu(t)\quad\text{for a.e. }t\in\left]0,1\right[.
\end{equation}
\end{theorem}
\begin{proof}
Let $\mu$, $\Xi$, and $\xi$ be given as in the statement of the theorem. 
Due to Lemma 3.10 in \cite{Mani07PRUR} for $\eps>0$ we can obtain 
sufficiently smooth approximations $\mu_\eps$, $\Xi_\eps$, and 
$\xi_\eps$ satisfying the continuity equation \eqref{eq:inhomCE}
and converging in a suitable sense to $\mu,\Xi$, and $\xi$, respectively.

Moreover, $\mu_\eps$, $\Xi_\eps$, and $\xi_\eps$ satisfy for any
convex, nondecreasing function $\psi:\left[0,\infty\right[
\to\left[0,\infty\right[$ we have the uniform estimates
\begin{equation}
\begin{split}
\int_{\Omega}\psi(|\Xi_\eps(t,x)|)\dd\mu_\eps(t)
&\leq\int_{\Omega}\psi(|\Xi(t,x)|)\dd\mu(t)\quad\text{and}
\\
\int_{\Omega}\psi(|\xi_\eps(t,x)|)\dd\mu_\eps(t)
&\leq\int_{\Omega}\psi(|\xi(t,x)|)\dd\mu(t)
\end{split}
\end{equation}

Applying the representation result in Proposition \ref{p:represCE}
for $\mu_\eps$, $\Xi_\eps$, and $\xi_\eps$ we obtain the formula
\begin{equation}\label{eq:represForm}
\mu_\eps(t) = \mfP\lambda_\eps(t),\qquad \lambda_\eps(t)
=[X_\eps(t;\cdot),R_\eps(t;\cdot)]_{\#}\lambda_\eps(0),
\end{equation}
where $\lambda_\eps(0)$ is a lift of $\mu_\eps(0)$ and 
$X_\eps$ and $R_\eps$ are the maximal solutions of
\begin{equation}\label{eq:charODEeps}
\dot X_\eps(t;x)   = \Xi_\eps(t,X_\eps(t;x)),\quad
\dot R_\eps(t;x,r) = 2\xi(t,X_\eps(t;x)) R_\eps(t;x,r)
\end{equation}
subject to the initial conditions $X_\eps(0;x) = x$
and $R_\eps(0;x,r) = r$. With this we define the map
$\Phi_\eps:\CC\to \rmC([0,1];\CC)$ via
\[
\Phi_\eps([x,r]) = \big(t\mapsto[X_\eps(t;x),R_\eps(t;x,r)]\big),
\]
and introduce the dynamic plan $\bfpi_\eps\in\MMM(\rmC([0,1];\CC))$
as $\bfpi_\eps=(\Phi_\eps)_{\#}\lambda_\eps(0)$.

We aim to show that the sequence $\bfpi_\eps$ is tight such that 
we can find a subsequence (not relabeled) that is narrowly converging
to a dynamic plan $\bfpi$. Indeed, using Lemma \ref{lem:AbsContInCone}, 
\eqref{eq:charODEeps} and  the representation formula \eqref{eq:represForm} 
we have for $0\leq t_0<t_1\leq 1$ that
\begin{align*}
\int_{\calA}\int_{t_0}^{t_1}\big|\dot{\wh z}\big|_{\CC}(t)^2 \dd t\dd\bfpi_\eps(\wh z)
&= \int_{\CC}\int_{t_0}^{t_1}\Big\{R_\eps(t;x,r)^2|\dot X_\eps(t,x)|^2 + 
\dot R_\eps(t;x,r)^2\Big\}\dd t\,\dd \lambda_\eps(0)\\
&= \int_{\CC}\int_{t_0}^{t_1}\Big\{\big|\Xi_\eps(t,X_\eps)\big|^2 + 
4\big|\xi_\eps(t,X_\eps)\big|^2\Big\} R_\eps^2\dd t\,\dd \lambda_\eps(0)\\
&= \int_{t_0}^{t_1}\int_{\Omega}\Big\{|\Xi_\eps(t,x)|^2 + 
4|\xi_\eps(t,x)|^2\Big\} \,\dd \mu_\eps(t)\dd t\\
&\leq \int_{t_0}^{t_1}\int_{\Omega}\Big\{|\Xi(t,x)|^2 + 
4|\xi(t,x)|^2\Big\} \,\dd \mu(t)\dd t<\infty.
\end{align*}
Since the functional $\wh z\mapsto \int_0^1|\dot{\wh z}|_{\CC}^2\dd t$
has compact sublevels in $\{\wh z\in \rmC([0,T];\CC)\,|\,\wh z(0) = [x,r]\}$
we have shown the tightness of $\bfpi_\eps$ and we can extract a 
subsequence narrowly converging to a limit $\bfpi$ in $\MMM(\rmC([0,1];\CC))$. 
Moreover, due to the lower semicontinuity of $\wh z\mapsto \int_0^1|\dot{\wh z}|_{\CC}^2\dd t$ 
we immediately obtain 
\[
\int_{\calA}\int_{t_0}^{t_1}|\dot{\wh z}|_{\CC}^2 \dd t\dd\bfpi(\wh z)<\infty.
\]
Hence, $\bfpi$ is concentrated on absolutely continuous curves.

It remains to show that $\bfpi$ is a dynamic plan for $\mu$, i.e.,
$\mu(t) = \mfP((e_t)_{\#}\bfpi)$ from which \eqref{eq:lowMTD} follows.
We easily show that due to the construction of $\bfpi_\eps$ we have that
$\mfP((e_t)_{\#}\bfpi_\eps)=\mu_\eps(t)$. The claim now follows from the
continuity of the map $\bfpi\mapsto \mfP((e_t)_{\#}\bfpi$ with respect
to the weak convergence in $\MMM(\rmC([0,1];\CC))$.
\end{proof}

Finally, we can prove the equivalence of the dynamical and 
the transport plan formulation of the Hellinger--Kantorovich distance.
\\[0.3cm]

\emph{Proof of Theorem \ref{th:properties}(v). }
We want to show that for all $\mu_0,\mu_1\in\MMM(\Omega)$ 
we have
\[
{\msHK}(\mu_0, \mu_1) = \msD_{1,4}(\mu_0,\mu_1).
\]
Due to Theorem \ref{thm:timeUpper} we have for any curve $t\mapsto
\mu(t)$ connecting $\mu_0$ and $\mu_1$ and satisfying the continuity
equation \eqref{eq:inhomCE} for a vector and scalar field $\Xi$ and
$\xi$, respectively, the upper estimate
\[
{\msHK}(\mu_0,\mu_1)^2 \leq \int_0^1 |\dot\mu|_{\msHK}^2(t)\dd t 
\leq \int_0^1\int_\Omega\Big\{|\Xi(t,x)|^2+4\xi(t,x)^2\Big\}\dd\mu(t)\dd t.
\]
Thus, by minimizing over all absolutely continuous curves
$t\mapsto \mu(t)$ connecting $\mu_0$ and $\mu_1$ we have
$\msHK(\mu_0,\mu_1)\leq \msD_{1,4}(\mu_0,\mu_1)$.

To show that equality holds, we consider a geodesic curve $s\mapsto
\mu(s)$, which is obviously absolutely continuous with respect to
$\msHK$ and we have $|\dot\mu|_{\msHK}\equiv \msHK(\mu_0,\mu_1)$.  By
Theorem \ref{thm:timeLower} we then have
\[
\msHK(\mu_0,\mu_1)^2 = \int_0^1|\dot\mu|_{\msHK}(t)^2\dd t\geq
\int_0^1\int_\Omega\Big\{|\Xi(t,x)|^2+4\xi(t,x)^2\Big\}\dd\mu(t)\dd t.
\]
Hence, we have proven Theorem \ref{th:properties}(v).
\QED


\section{Geodesic curves for $\msHK$}
\label{se:Geodesics}

In this section, we provide some general results on geodesics as well
as a few illuminating examples. Section \ref{ss:CvxPotenial} deals with 
the geodesic $\Lambda$-convexity of functionals $\calF:\calM(\Omega)
\to \R{\cup}\{\infty\}$. In particular, we show
that the linear functional $\calF_\Phi: \mu \mapsto \int_\Omega
\Phi(x)\dd\mu(x) $ is
geodesically $\Lambda$-convex if and only
if the function $[x,r]\mapsto r^2\Phi(x)$ is $\Lambda$-convex on
$(\CC,\msd_{\mfC})$. In Section \ref{ss:AllDiracM}, we return to the
problem of moving the measure $\mu_0=a_0\delta_{y_0}$ to
$\mu_1=a_1\delta_{y_1}$ and show that in the case $|y_1{-}y_0|=\pi/2$
there is an infinite set of geodesic curves. Moreover, we show that
all these curves are indeed solutions of the formally derived equation 
\begin{equation}
  \label{eq:GC.eqn}
  \frac{\rmd}{\rmd s}\mu+ \div\big(\mu\nabla \xi  )=4\xi\mu, \qquad 
\frac{\rmd}{\rmd s} \xi + \frac12|\nabla \xi|^2 + 2 \xi^2=0,
\end{equation}
where $\xi$ is in fact the same for all geodesic connections. 
In Section \ref{ss:Dilations} we discuss geodesic curves that are induced 
by dilation of measures. 
Section \ref{ss:CharFunct} discusses how the geodesic curve connecting $\mu_0 =
\chi_{[0,1]}$ and $\mu_1=\chi_{[2,3]}$ can be constructed. 

Finally, Section \ref{ss:NotSemiconcave} shows that
$(\calM(\Omega),\msHK)$ is not a positively curved (PC) space in the
sense of Alexandrov (cf.\ \cite[Sect.\,12.3]{AmGiSa05GFMS}) if
$\Omega $ is two-dimensional.

\subsection{Geodesic $\Lambda$-convexity for some functionals}
\label{ss:CvxPotenial}

Here we give some first results of geodesic $\Lambda$-convexity for
functionals $\calF:\calM(\Omega)\to \R{\cup}\{\infty\}$,  
which is defined via 
\begin{equation}
  \label{eq:GLC}
\begin{aligned}   &\forall\, \text{geod.\:curves } \mu:
[0,1]\to \calM(\Omega):\\ &\quad
\calF(\mu(s))\leq (1{-}s)\calF(\mu(0))+ s \calF(\mu(1)) -\Lambda\,
\frac{s(1{-}s)}2 \msHK(\mu(0),\mu(1))^2. 
\end{aligned}
\end{equation}
We first provide an exact characterizations of $\Lambda$ 
for linear functionals, then give some preliminary results and
conjectures for nonlinear functionals. 

It is well-known that in the case of the Wasserstein distance, geodesic 
$\Lambda$-convexity of functionals of the form 
$\calF_\Phi(\mu)=\int_\Omega \Phi(x)\dd\mu(x)$ is satisfied if and only 
if $x\mapsto \Phi(x)$ is $\Lambda$-convex, see \cite[Sect.~9.3]{AmGiSa05GFMS}. 
Hence, it is natural to ask whether the same can be said for the 
Hellinger--Kantorovich distance $\msHK$ and geodesics in the cone space.

We start with a very easy relation for the total mass along the
geodesic curves. It turns out that the mass depends on the parameter
convex and quadratically.

\begin{proposition}\label{pr:TotalMass}
Consider a geodesic curve $[0,1]\ni s \mapsto \mu(s)$ given by
\eqref{eq:geodesic} and set 
\[
m(s):= |\mu|(s) =  \int_\Omega \dd\mu(s).
\]
Then, we have 
\begin{subequations}
  \label{eq:MassEstim}
\begin{align}
  \label{eq:MassEstim.A}
&m(s)= (1{-} s)^2 m(0) + s^2 m(1) + 2s(1{-}s) m_* \\
   \nonumber 
&\hspace*{6em}\text{with }
m_*:=\int_{\CC\ti \CC} r_0r_1 \trcos(|x_1{-}x_0|) \dd
\gamma([x_0,r_0],[x_1,r_1]),\\
  \label{eq:MassEstim.B}
& \tfrac{m(0)m(1)}{m(0){+}m(1)} \leq  (1{-} s)^2
m(0) + s^2 m(1) \leq m(s) \leq \Big((1{-} s)\sqrt{m(0)}+s
\sqrt{m(1)}\Big)^2 
\end{align}
\end{subequations}
for all $s\in[0,1]$. Moreover, $m''(s)=2\msHK(\mu_0,\mu_1)^2\geq 0$
which implies
\begin{equation}
  \label{eq:MassIdent}
m(s) = (1{-}s) m(0) + sm(1) -s(1{-}s)\msHK(\mu_0,\mu_1)^2.
\end{equation}
\end{proposition}
\begin{proof} Using the definition of $\mu(s)$ via the projection 
  and $Z(s;\cdot,\cdot)$ in \eqref{eq:GeodInterpol} we have 
\[
m(s) = \int_{\Omega} \dd \mu(s)= \int_{\CC}\!\! r^2 \dd\lambda(s) 
= \int_{\CC}\!\! r^2 \dd\big(Z(s;\cdot,\cdot)_\# \gamma\big)
= \int_{\CC\ti\CC}\!\!  R(s;z_0,z_1)^2 \dd \gamma(z_0,z_1). 
\]
Now, we can use the explicit quadratic structure of $R^2$ given in
\eqref{eq:GeodInterpol} we find the quadratic formula
\eqref{eq:MassEstim.A}. 

For estimate \eqref{eq:MassEstim.B} we use the fact that $
\trcos(|x_1{-}x_0|)$ takes values only in the interval $[0,1]$ on the
support of $\gamma$, see Proposition \ref{pr:Reservoir}(c). By the
Cauchy-Schwarz estimate we have $0\leq m_* \leq \sqrt{m(0)m(1)}$,
which implies the estimates. 

Obviously, we have $m''(s)=2(m(0)+m(1)-2m_*)$, and comparing to
the characterization \eqref{eq:HKby-wh-r} we find
$m''(s)=2\msHK(\mu_0,\mu_1)^2$ as desired.
\end{proof}

Next, we consider the linear functional $\calF_\Phi(\mu)=\int_\Omega
\Phi(x)\dd \mu(x)$ with $\Phi\in \rmC^0(\Omega)$.

\begin{proposition}
Let $\Phi\in\rmC^0(\Omega)$ be given and define
$\wt\Phi([x,r]) = r^2\Phi(x)$. Then the functional $\calF_\Phi:\MMM(\Omega)\to\R$
is $\Lambda$-convex along  $[0,1]\ni s \mapsto \mu(s)$ given by
\eqref{eq:geodesic} if and only if $\wt\Phi:\CC\to\R$ is $\Lambda$-convex.
\end{proposition}
\begin{proof}
Assume that $\wt\Phi:\CC\to\R$ is $\Lambda$-convex.
We use the definition of $s\mapsto \mu(s)$ in \eqref{eq:geodesic}
to find
\[
\calF_\Phi(\mu(s))=\int_\Omega r^2\Phi(x)\dd\lambda(s) = \int_{\CC}\wt\Phi
\big(Z(s;\cdot,\cdot)\big) \dd\gamma,
\]
where $\gamma\in\MMM_2(\CC{\times}\CC)$ is an optimal plan for
$\mu_0=\mu(0)$ and $\mu_1=\mu(1)$. Thus, with the convexity
of $\wt\Phi$ and the optimality of $\gamma$ we find
\[
\calF_\Phi(\mu(s))\leq(1{-}s)\calF_\Phi(\mu_0) + s\calF_\Phi(\mu_1)
-\frac{\Lambda}{2}s(1{-}s)\msHK(\mu_0,\mu_1)^2.
\]
Conversely, if $\calF_\Phi$ is $\Lambda$-convex on $\MMM(\Omega)$
we can consider geodesic curves for two Dirac measures $\mu_0=a_0\delta_{x_0}$
and $\mu_1=a_1\delta_{x_1}$ to obtain convexity of $\wt\Phi$
along geodesics in $\CC$. However, as transport above the threshold $\pi/2$
is not optimal, this excludes geodesic curves in $\CC$ for distances
$\pi/2<|x_0{-}x_1|<\pi$. In this case, we note that we can
always reduce this case to two overlapping geodesic curves
for distances below $\pi/2$.
\end{proof}

Finally, we provide some negative results for functionals that are
geodesically $\Lambda$-convex for the Wasserstein--Kantorovich distance
but not for the Hellinger--Kantorovich distance. A simple necessary
condition is obtained by realizing that for all $\mu_1\in \MMM(\Omega)$
the Hellinger geodesic 
\[
\mu^\rmH(s):= s^2\,\mu_1
\]
is also the unique geodesic in $(\MMM(\Omega),\msHK)$ connecting
$\mu_0=0$ and $\mu_1$. Indeed, this easily follows from the fact that
the possible lifts of $\mu_0$ are given by $\alpha \delta_\TT \in
\MMM_2(\CC)$ with $\alpha\geq 0$. However, geodesics in
$(\CC,\msd_\mfC)$ connecting $\TT$ and $z_1=[x,r]$ are simply given by
$z(s)=[x,sr]$.   

Applying this to Boltzmann's logarithmic entropy
$\calE:\mu \mapsto \int_\Omega \FBoltz(\rmd \mu/\rmd x) \dd x$, we see
that it is not
geodesically $\Lambda$-convex with respect to $\msHK$. For this, 
consider the geodesic $\mu(s)=s^2 u\dd x$, where $u\in
\rmL^2(\Omega)$, $u\geq 0$, and $|\mu|=\int_\Omega u\dd x >0$ to find 
the relation
\[
\calE(\mu(s)) = s^2 \calE(u\dd x) +(1{-}s^2)\int_\Omega 1\dd x +
2 s^2\log s \:\int_\Omega u\dd x.
\]
Clearly, the last term destroys geodesic $\Lambda$-convexity.
Similarly, for $p\in {]0,1[} \cup {]1,\infty[}$ we may look at
functionals of the form
\begin{equation}
  \label{eq:calEp}
  \calE_p(\mu)= \int_\Omega \frac1{p-1} \Big( \frac{\rmd \mu}{\rmd
  x}\Big)^p \dd x.
\end{equation}
Along the geodesics $\mu(s) = s^2 u \dd x$ we obtain $e_p(s):= \calE_p(\mu(s) =
s^{2p} \calE(\mu(1))$. For $p\in {]1/2,1[}$ we conclude $e''_p(s)\to
-\infty$ for $s\downarrow 0$ due to $e_p(1)<0$. Hence, for these $p$
the functional $\calE_p$ is not geodesically $\Lambda$-convex for any
$\lambda\in \R$ with respect to $\msHK$. 

The following remark supports the conjecture that the functional $\calE_p$ is
geodesically convex on $(\MMM(\Omega),\msHK)$, or more generally on
$(\MMM(\Omega),\msD_{\alpha, \beta})$ for all $p>1$. It is
based on the formal differential calculus developed in
\cite{LieMie13GSGC}, which was in fact the stimulus of this work.  
If this is the case one may consider the geodesically $\Lambda$-convex
gradient system $(\MMM(\Omega),\calE_p{+}\calF_\Phi, \msD_{\alpha,
  \beta})$, which corresponds to the partial differential equation 
\[
\pl_t u = - \bbK_{\alpha,\beta}(u) \rmD\big(\calE_p{+}\calF_\Phi)(u) =
\alpha \Big( \Delta(u^p) + \mathop{\mathrm{div}}\big(u \nabla
\Phi\big)\Big) - \beta \Big( \Phi u + \frac{u^p}{p-1}\Big),
\]
complemented by the no-flux boundary conditions $\nabla
\big(\frac{p}{p-1}u^{p-1}+ \Phi)\cdot\nu=0$. 
Note that this equation always has the solution $u\equiv 0$, which is
different from the unique minimizer $u_\mathrm{min}$ of $\calE_p{+}\calF_\Phi$, if
$\Phi$ attains negative values somewhere. Indeed, we have 
$u_\mathrm{min}:=\Big(\frac{p-1}p \max\{ 0,
-\Phi\}\Big){}^{1/(p-1)}$. We refer to \cite[Eqn.\,(2.1)]{PeQuVa14HSAM} for an
application for modeling of tumor growth.

\begin{remark}[Geodesic convexity via Eulerian calculus]\slshape Following
  ideas in \cite{OttWes05ECCW,DanSav08ECDC} a formal calculus for
  reaction-diffusion systems was developed in \cite{LieMie13GSGC}.  The idea is
  to characterize the geodesic $\Lambda$-convexity of $\calE(u\dd
  x)=\int_\Omega E(u(x)) \dd x$ on
  $(\MMM(\Omega),\msD_{\alpha,\beta})$ by calculating the quadratic
  from $M(u,\cdot)$ generated by the contravariant Hessian of $\calE$:
\[
M(u,\xi)=\langle \xi,\rmD\bfV(u) \bbK_{\alpha,\beta}(u)\xi\rangle-
\frac12 \rmD_u\langle \xi , \bbK_{\alpha,\beta}(u)\xi\rangle 
[\bfV(u)] \text{
  with } \bfV(u)=\bbK_{\alpha,\beta}(u) \rmD\calE(u). 
\]
Then, one needs to show the estimate $M(u,\xi) \geq \Lambda \langle
\xi,\bbK_{\alpha,\beta}(u)\xi\rangle $. 

Following the methods in \cite[Sect.\,4]{LieMie13GSGC}, for
 $u\in \rmC^0_\rmc(\Omega)$ and smooth $\xi$  we obtain 
\begin{align*}
M(u,\xi)=\int_\Omega \bigg(& \alpha^2 \Big( \big( A(u){-}H(u)\big)
(\Delta \xi)^2 + H(u)\big| \rmD^2 \xi\big|^2 \Big) \\    
&+ \alpha\beta \Big(B_1(u) |\nabla\xi|^2 + B_2(u)\xi\Delta \xi \Big)
+\beta^2 B_3(u) \xi^2 \bigg) \dd x,
\end{align*}
\begin{align*}
\text{where }&A(u)= u^2 E''(u), \quad H(u)= uE'(u) - E(u), \quad
B_1(u)=\frac{3u}2 E'(u)- E(u),\\
&B_2(u)=-2u^2 E''(u) + uE'(u)- E(u),\quad B_3(u)= u^2 E''(u)+
\frac{u}2 E'(u).
\end{align*}
For the special case $E(u)=u^p/(p{-}1)$ with $p>1$  we find the relation 
\[
M_p(u,\xi)=\!\!\int_\Omega\!\! \Big[ \alpha^2 \big( (p{-}1)
(\Delta \xi)^2 {+}| \rmD^2 \xi|^2 \big)    
+ \alpha\beta \big(\tfrac{3p{-}2}{2p{-}2} |\nabla\xi|^2 {-} (2p{-}1)\xi\Delta \xi \big)
+\beta^2 \tfrac{2p^2{-}p}{2p{-}2} \xi^2 \Big]u^p \dd x,
\]
which is nonnegative, because the mixed term $\xi\Delta \xi$ can be
estimated via 
\[
{-}\alpha\beta (2p{-}1)\xi\Delta \xi\geq -\alpha^2(p{-}1)
(\Delta \xi)^2- \tfrac{(2p{-}1)^2}{4(p{-}1)} \beta^2 \xi^2.
\]
Thus, the formal Eulerian calculus suggests that $\calE_p$ is
geodesically convex with respect to $\msHK$ for all $p>1$. This
investigation will be continued in subsequent work.    
\end{remark}

\subsection{Geodesic connections for  two Dirac measures}
\label{ss:AllDiracM}
 
While for the characterization of the distance $\msHK$ the choice of
the geodesics is not so relevant, we want to highlight that the
set of geodesic connections between two measures can be very large. 
As shown in Section \eqref{ss:GeodCurves} and Corollary
\ref{co:ConeGeodesic} all geodesic curves can be constructed from optimal
couplings $\gamma \in \calM_2(\CC\ti\CC)$ in
\eqref{eq:DHKbyLift}. However, 
many $\gamma$ lead via the projection $\mfP$ to
the  same geodesic curve. Here we show that the set of geodesics may
still form an infinite-dimensional convex set. 

We treat the case of two Dirac measures $a_j\delta_{y_j}$.
The case $|y_1{-}y_0|\neq \pi/2$ is trivial, since only one geodesic
connection exists. However, for the  critical distance 
$|y_1{-}y_0|=\pi/2$ an uncountable number of linearly independent
connecting geodesics exists, such that the span of the convex set of all
geodesics is infinite dimensional. 

We consider $\mu_0=a_0\delta_{y_0}$ and $\mu_1=a_1\delta_{y_1}$ with
$a_0,a_1>0$. As was shown before there is exactly one connecting
geodesics if $|y_0{-}y_1|\neq \pi/2$. Indeed, for $|y_0{-}y_1|< \pi/2$
we have $\mu(s)=a(s)\delta_{x(s)}$ as discussed in Section
\ref{ss:MassDHW}. For $|y_0{-}y_1|> \pi/2$ we have a pure Hellinger
case with $\mu(s)=(1{-}s)^2a_0\delta_{y_0}+ s^2a_1\delta_{y_1}$. 

For the critical case $|y_0{-}y_1|=\pi/2$ we have a huge set of 
possible geodesics, since we may consider all lifts 
$\wh\lambda_0,\wh\lambda_1 \in \MMM({[0,\infty[})$ satisfying
\[
a_j=\int_{{[0,\infty[}} r^2 \dd \wh\lambda_j(r)  \text{ for }j=1,2\
\text{ and } \ \wh\lambda_0({[0,\infty[}) =
\wh\lambda_1({[0,\infty[}).
\]
Now every coupling $\wh\gamma \in \Gamma(\wh\lambda_0,\wh\lambda_1)
=\set{\wh\gamma \in\MMM ({[0,\infty[}^2)\,}{\,\Pi^i_\#\wh\gamma=\wh\lambda_i}$ 
provides an optimal coupling. To see this,  we set $\lambda_j=\delta_{y_j} \otimes 
\wh\lambda_j\in\MMM(\CC)$ and $\gamma=\delta_{y_0}  \otimes  \delta_{y_1}
 \otimes  \wh\gamma\in\MMM(\CC{\ti}\CC)$. Since $|y_0{-}y_1|=\pi/2$ implies $
\msd_{\mfC}([y_0,r_0],[y_1,r_1])^2= r_0^2 + r_1^2 $, we find
\begin{align*}
\int_{\CC \ti \CC} 
\msd_{\mfC}([x_0,r_0],[x_1,r_1])^2 \dd \gamma&= \int_{{[0,\infty[}^2} 
\msd_{\mfC}([y_0,r_0],[y_1,r_1])^2 \dd \wh\gamma \\
&= 
\int_{{[0,\infty[}^2} \big( r_0^2{+}r_1^2\big)
 \dd \wh\gamma= a_0+a_1 =
\msHK(a_0\delta_{y_0}, a_1\delta_{y_1})^2.
\end{align*}
Now, geodesic curves can be constructed for every $\wh\gamma$ as
defined in \eqref{eq:geodesic}. Obviously the set of all possible
pairs $(\wh\lambda_0,\wh\lambda_1)$ is convex and therefore also 
the set of all $\wh\gamma\in \Gamma(\wh\lambda_0,\wh\lambda_1)$.
Hence, the set of all optimal $\wh\gamma$ is convex. 

However, the mapping from $\wh\gamma$ to $\mu(\cdot)$ is not
surjective, since there is a huge redundancy. Indeed, by the definition
$\mu_s=\mathfrak  P\,Z(s;\cdot,\cdot)_\# \gamma$  we have 
\[
\int_\Omega \psi(x)\dd \mu_s(x) = \int_{{[0,\infty[}^2}
R(s,[y_0,r_0],[y_1,r_1])^2 \psi\big(X(s,[y_0,r_0],[y_1,r_1])\big)  \dd
\wh\gamma(r_0,r_1),
\]
where, using $|y_0{-}y_1|=\pi/2$,  we have
$R(s,[y_0,r_0],[y_1,r_1])^2=(1{-}s)^2r_0^2+s^2r_1^2$ and   
\[
X(s,[y_0,r_0],[y_1,r_1])=(1{-}\rho(s))y_0+ \rho(s)y_1 \text{ with } 
\rho(s)= \frac2\pi \arccos\Big[1 + \Big(\frac{s
  r_1}{(1{-}s)r_0}\Big)^2\Big]^{-1/2} ,
\]
see \eqref{eq:GeodInterpol}.  The observation is that the integrand can be
written in the form $r_1^2 \Phi(s,r_0/r_1)$. In particular, for $r_1>0$
the two geodesics
\[
\begin{split}
\Lambda(s)&=\delta_{R(s;[y_0,r_0],[y_1,r_1])} \otimes 
\delta_{X(s;[y_0,r_0],[y_1,r_1])}
\quad 
\text{and}\\
\wt\Lambda(s)&=\delta_{r_1 R(s;[y_0,r_0/r_1],[y_1,1])} \otimes
\delta_{X(s;[y_0,r_0/r_1],[y_1,1])}
\end{split}
\] 
on $\CC$ give rise, via the projection $\mathfrak P$, to the same geodesic on $\Omega$
given by 
\[
\mu(s)=R(s;[y_0,r_0],[y_1,r_1])^2\delta_{X(s;[y_0,r_0],[y_1,r_1])}.
\] 

Thus, for a coupling $\gamma \in
\Gamma(\delta_{y_0}{\otimes}\wh\lambda_0, \delta_{y_1}\otimes \wh\lambda_1)$
given by $\wh\gamma \in \Gamma(\wh\lambda_0,  \wh\lambda_1)$
we can define the normalization
$N_0\gamma\in\MMM(\CC{\ti}\CC)$ with respect to $s=0$ as follows ($N_1\gamma$ for
$s=1$ can be defined similarly): 
\begin{align*}
\int_{\CC\ti\CC} \Phi(z_0,z_1)\dd N_0\gamma\ &= \ \int_{{]0,\infty[}\ti
  {]0,\infty[} } r_1^2\Phi\big([y_0,r_0/r_1],[y_1,1]\big) \dd 
\wh\gamma(r_0,r_1)\\[0.5em]
&\quad  + \Phi(\TT,\TT)b_{\TT,\TT} +\Phi([y_0,1],\TT)b_0  +\Phi(\TT,[y_1,1])b_1,
\\[0.5em]
\text{where }b_{\TT,\TT}& :=\wh\gamma(\{(0,0)\}), \ \
b_0:=\int_{{]0,\infty[} } \!\!r_0^2 \dd \wh\gamma(r_0,0),\ \ 
b_1:=\int_{{]0,\infty[}} \!\!r_1^2 \dd \wh\gamma(0,r_1). 
\end{align*}
The  terms involving $b_j$ contain the trivial Hellinger terms, where mass is
moved into or generated out of the tip $\TT$. Note that this mass can
be concentrated to the fixed value $r_0=1$ or $r_1=1$, respectively. 
The second term gives the mass that simply stays in $\TT$. The
interesting part is the first one, where still a measure
$\mathfrak n_0\wh\gamma$ survives: 
\[
\int_{{]0,\infty[}\ti
  {]0,\infty[} } r_1^2\Phi([y_0,r_0/r_1],[y_1,1]) \dd 
\wh\gamma(r_0,r_1) =: \int_{{]0,\infty[}} \Phi([y_0,r],[y_1,1]) \dd 
(\mathfrak n_0\wh\gamma)(r).
\]
We will see below that $(\mathfrak n_0\wh\gamma)(\rmd r)$ gives the
mass that leaves $y_0$ with speed $1/r$.

It is easy to see that $\gamma$ and $N_0\gamma$ generate the same
geodesic curve $\mu(\cdot)$. In terms of ${\mathfrak n}_0\wh\gamma$, we can now
write the geodesic curve associated with $\gamma$ in a simpler form,
namely 
\begin{align*}
\int_\Omega \psi(x)\dd \mu_s(x)&=
\int_{{]0,\infty[}} \big((1{-}s)^2r^2{+}s^2\big) \psi\big(
(1{-}\wh\rho(s,r))y_0{+}\wh\rho(s,r)y_1\big) \dd(\mathfrak
n_0\wh\gamma)(r)\\ 
&\quad + 0 \ + \ (1{-}s)^2 b_0 \psi(y_0) \ +\  s^2 b_1 \psi(y_1),
\end{align*}
where
$\wh\rho(s,r)=\frac2\pi\arccos\big[1+(\frac{s}{(1{-}s)r})^2\big]^{-1/2}
\in {]0,1[}$ for $s\in {]0,1[}$ and $r>0$ 
and 
\[
a_0 = b_0+ \int_{{]0,\infty[}} r^2\dd(\mathfrak
n_0\wh\gamma)(r) \ \text{ and } \ a_1= b_1+   \int_{{]0,\infty[}} \dd(\mathfrak
n_0\wh\gamma)(r). 
\] 
  
To simplify the further notation we now assume $\Omega =[-2,2]$,
$y_0=0$, and $y_1=\pi/2$. By the definition of $\wh\rho$ we find
$x=\wh x(r,s):=(1{-}\wh\rho(s,r))y_0{+}\wh\rho(s,r)y_1$ if and only if 
$r=(s\cos x)/((1{-}s)\sin x)$. Now differentiating $
\int_{[0,\pi/2]} \psi(x) \dd\mu_s(x)$ with respect to $s$ we find 
\begin{align*}
&\frac{\rmd}{\rmd s}\int_{[0,\pi/2]} \psi(x) \dd\mu_s(x) = -2(1{-}s)b_0\psi(0)+ 2sb_1 \psi(\pi/2)
\\
&+\int_{]0,\infty[}\Big(  2(s{-}(1{-}s)r^2)\psi(\wh x(r,s)) 
+\big((1{-}s)^2r^2{+}s^2\big) \psi'(\wh x(r,s)) \pl_s \wh x(r,s) \Big) \dd(\mathfrak
n_0\wh\gamma).  
\end{align*}
Defining the function $\xi$ and  $\Xi$ explicitly via 
\begin{equation}
  \label{eq:xi.2Dirac}
  \xi(s,x)=\frac{(\sin x)^2 -s }{2s(1{-}s)} \ \text{ and } \ 
\Xi(s,x)=\pl_x\xi(s,x)
\end{equation}
and eliminating $r$ in the above integral via  $r=(s\cos
x)/((1{-}s)\sin x)$ we obtain the identity
\[
 \frac{\rmd}{\rmd s}\int_{[0,\pi/2]}\!\! \psi(x) \dd\mu_s(x)=
 \int_{[0,\pi/2]}\!\!\Big( 4\xi(s,x)\psi(x) +
 \Xi(s,x) \psi'(x)\Big) \dd\mu_s(x).
\] 
Since $\xi$ satisfies the Hamilton--Jacobi equation $\tfrac{\rmd}{\rmd s} \xi+
\frac12|\nabla \xi|^2 + 2 \xi^2=0$, we conclude that the pair
$(\mu,\xi)$ indeed satisfies the equation \eqref{eq:GC.eqn}.

It is interesting to note that $\xi$ is independent of the measure
$\mathfrak n_0 \wh\gamma$. All the information about the precise form
of the connecting geodesic is solely encoded in the information how the
singularity for $s\searrow 0$ and $s\nearrow 1$ are formed, and this
information is exactly contained in $\mathfrak n_0\wh\gamma$.

\begin{figure}
\noindent%
\includegraphics[width=0.245\textwidth]{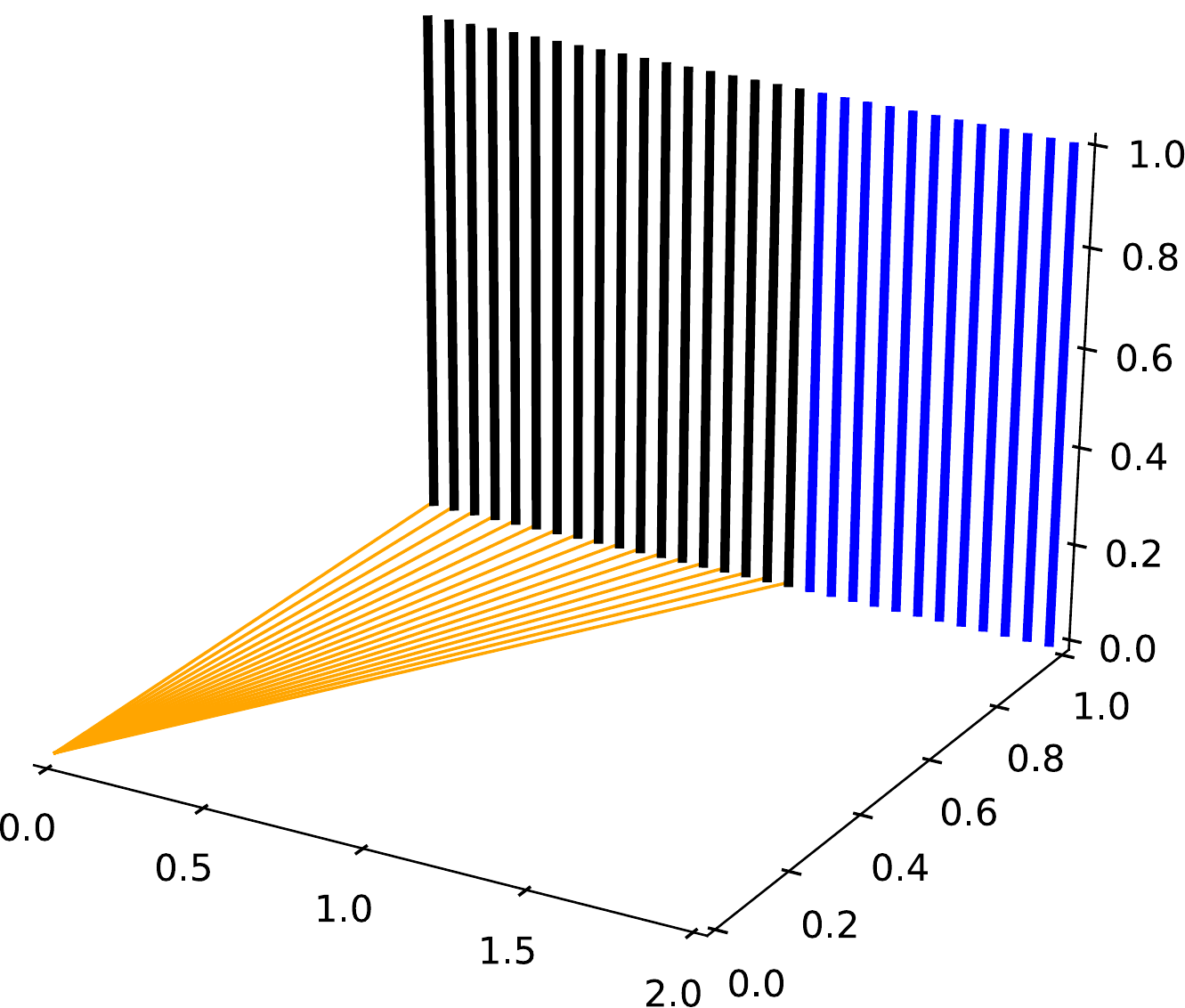}
\includegraphics[width=0.245\textwidth]{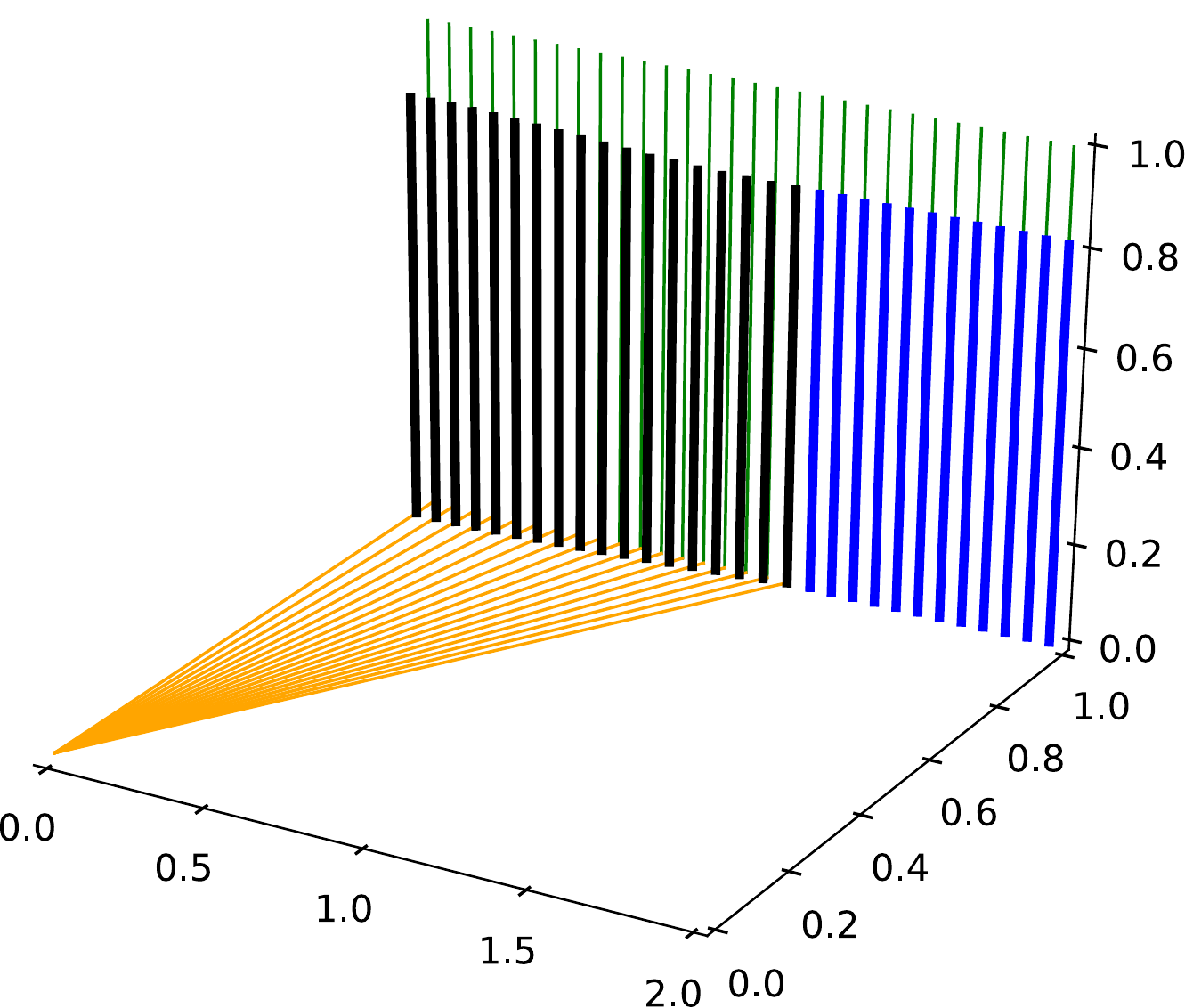}
\includegraphics[width=0.245\textwidth]{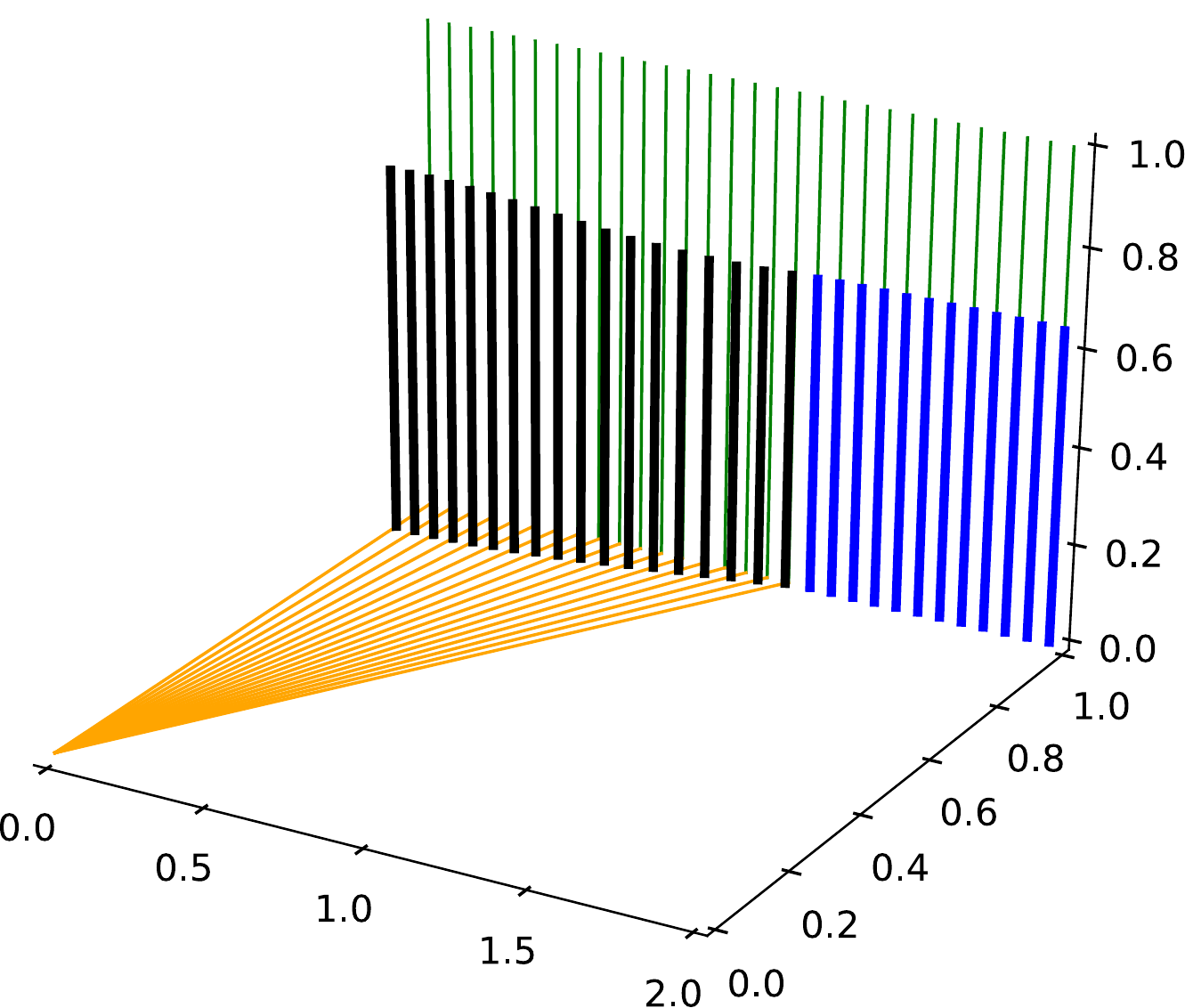}
\includegraphics[width=0.245\textwidth]{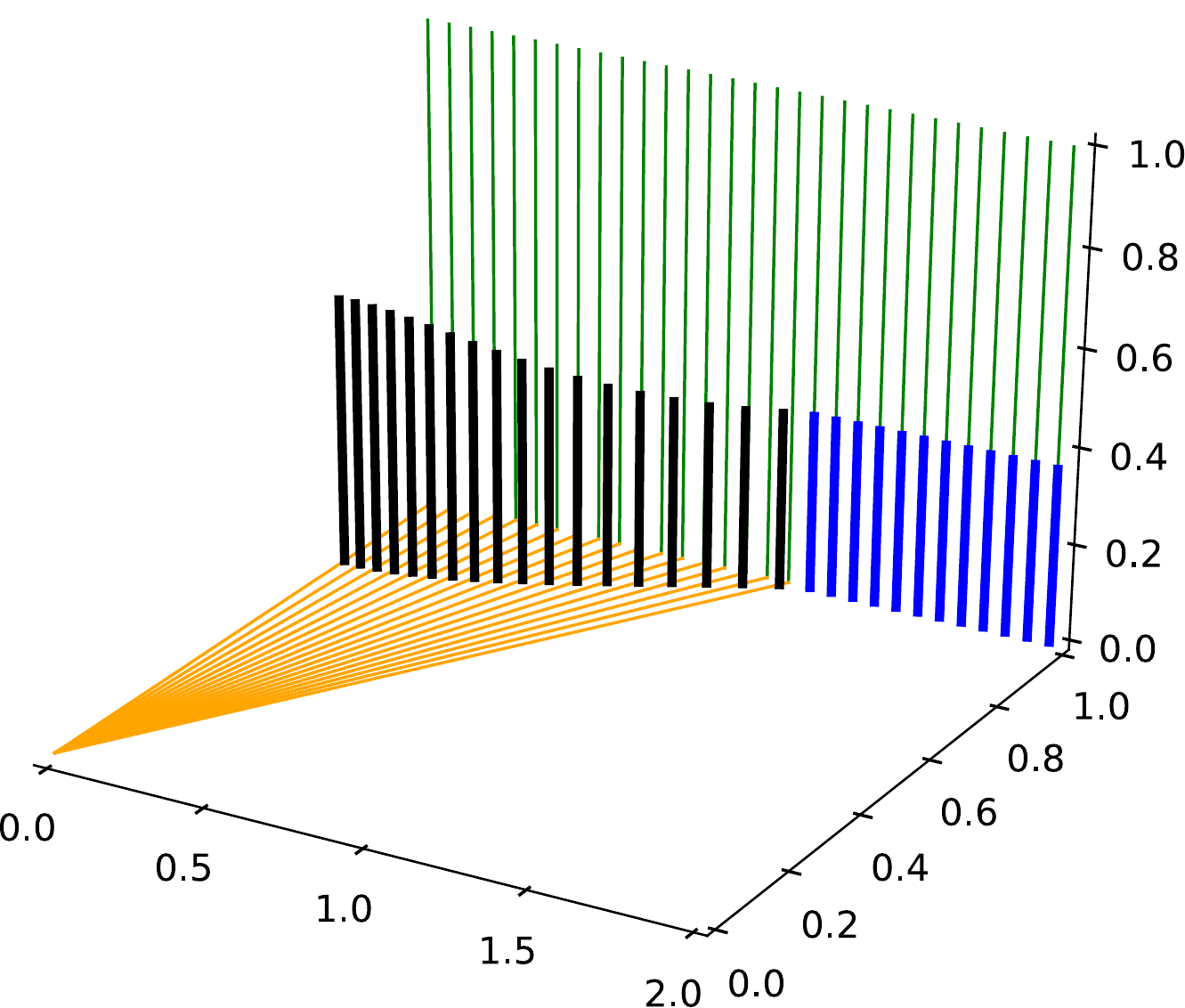}

\noindent%
\includegraphics[width=0.245\textwidth]{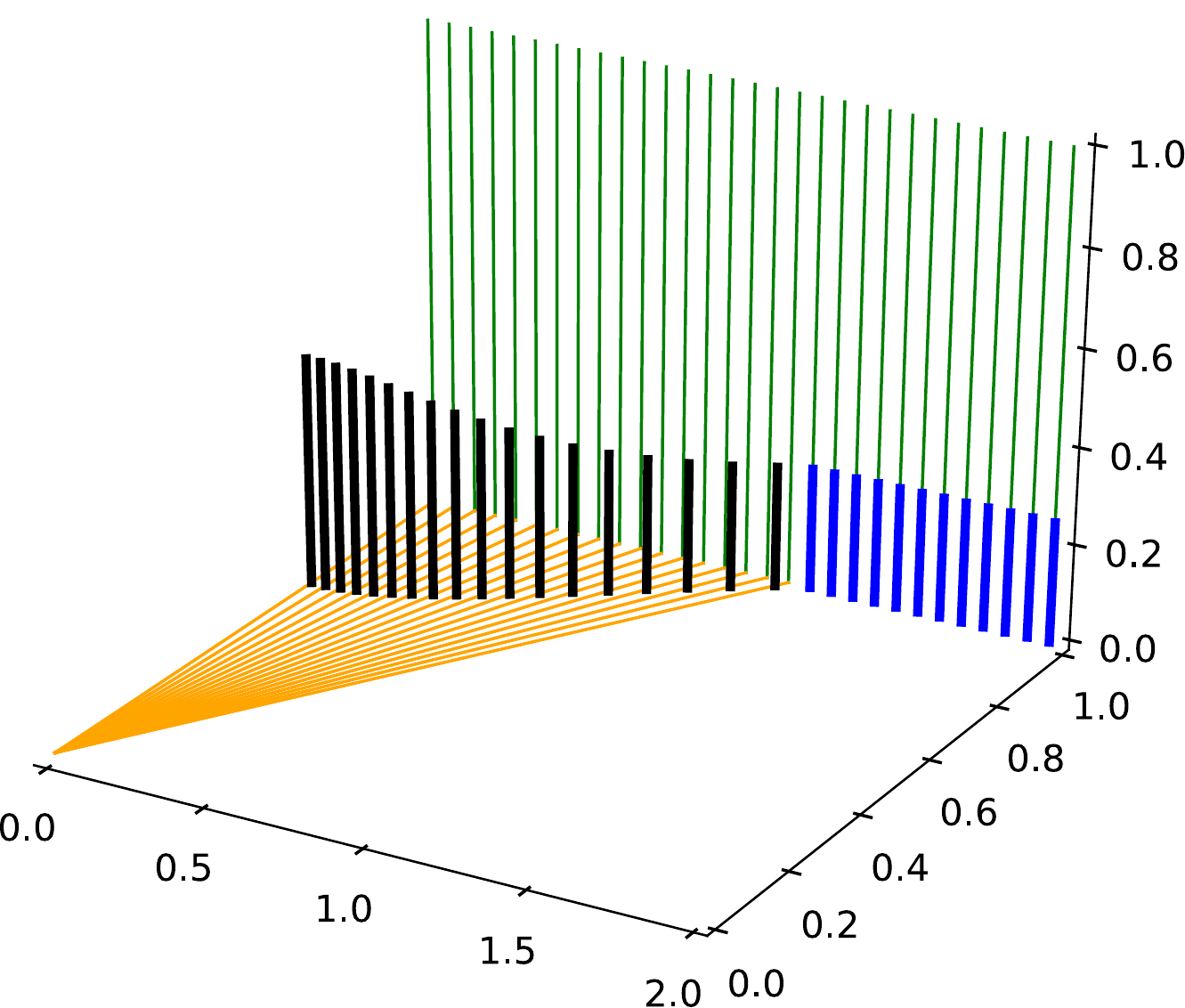}
\includegraphics[width=0.245\textwidth]{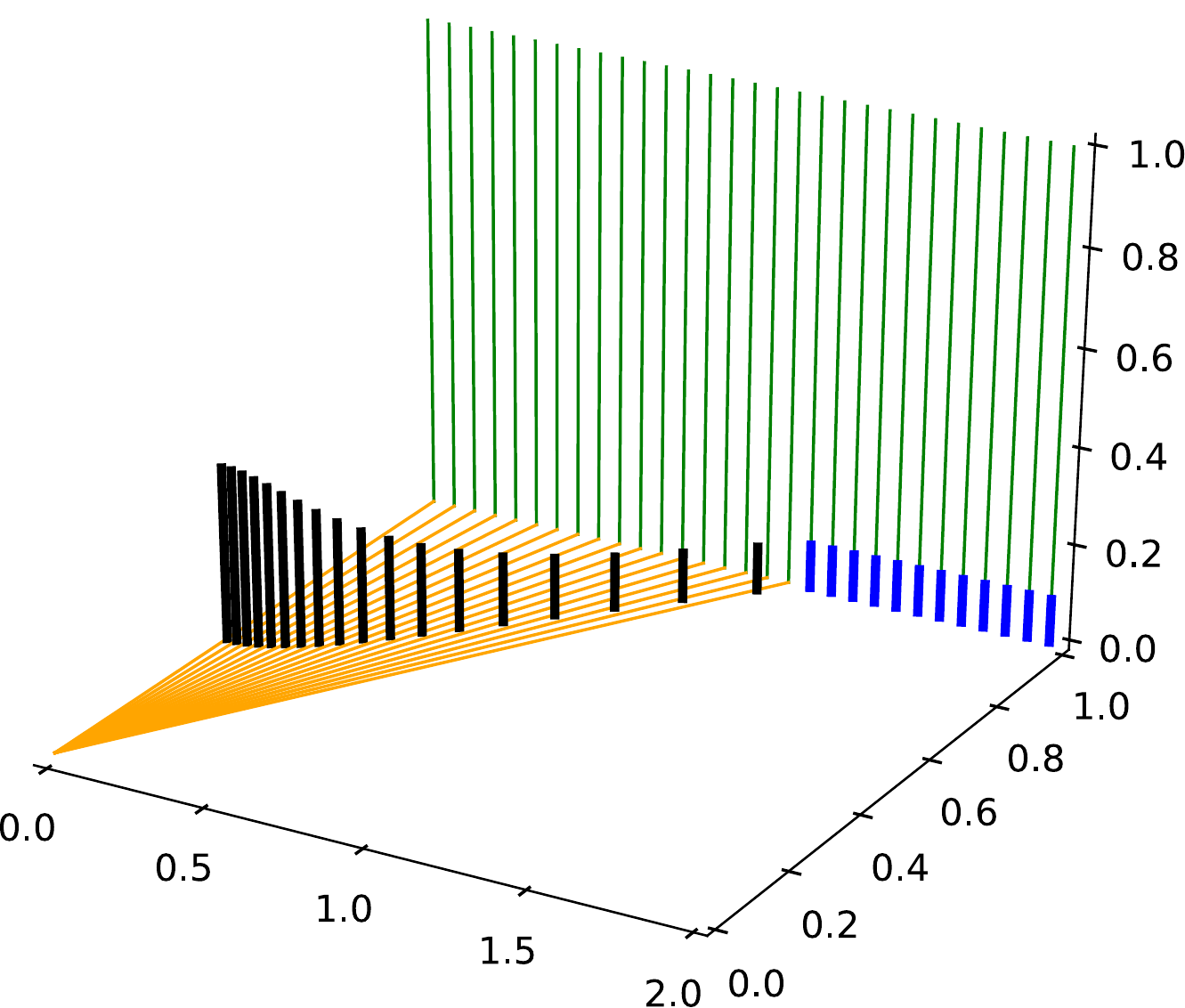}
\includegraphics[width=0.245\textwidth]{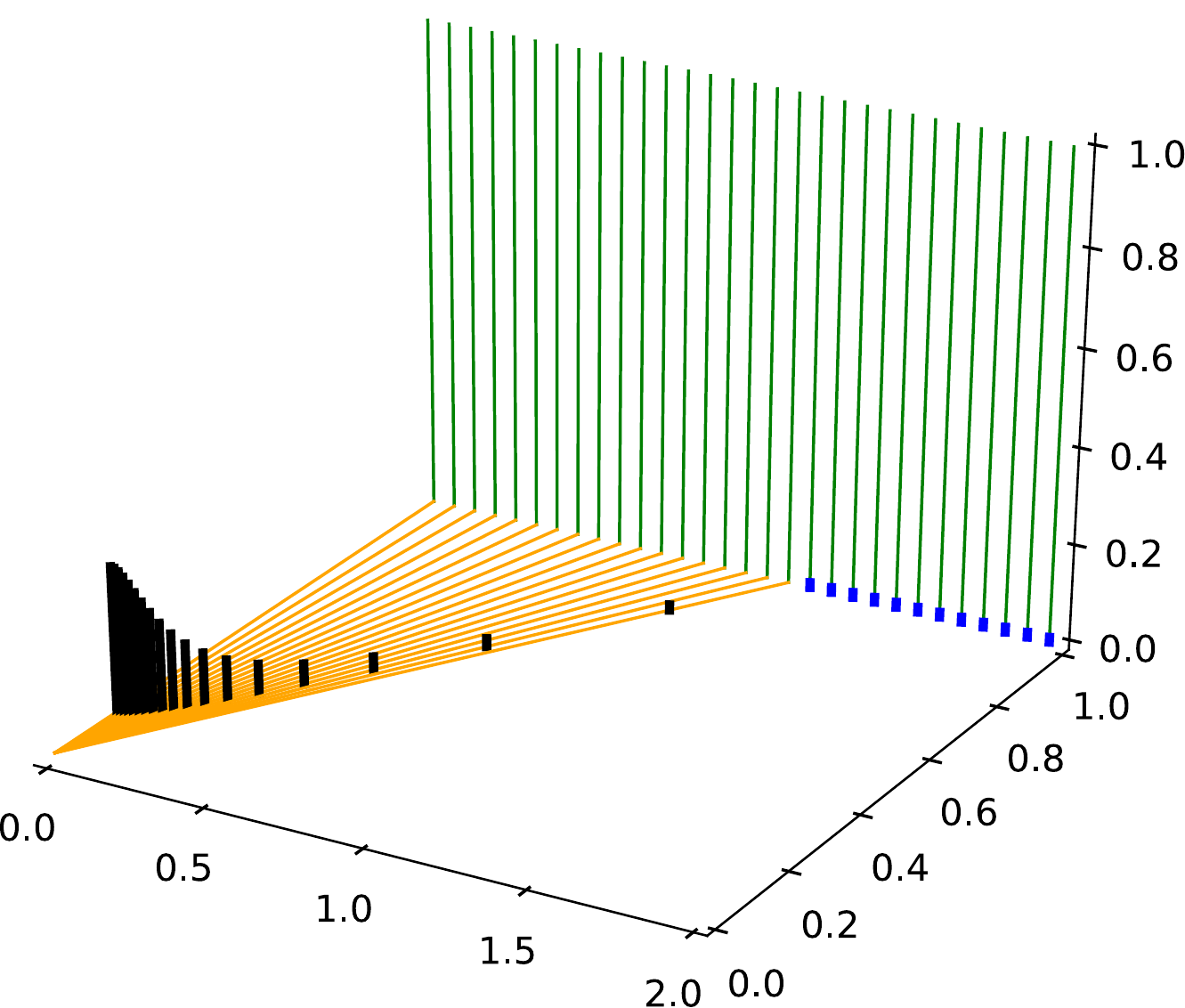}
\includegraphics[width=0.245\textwidth]{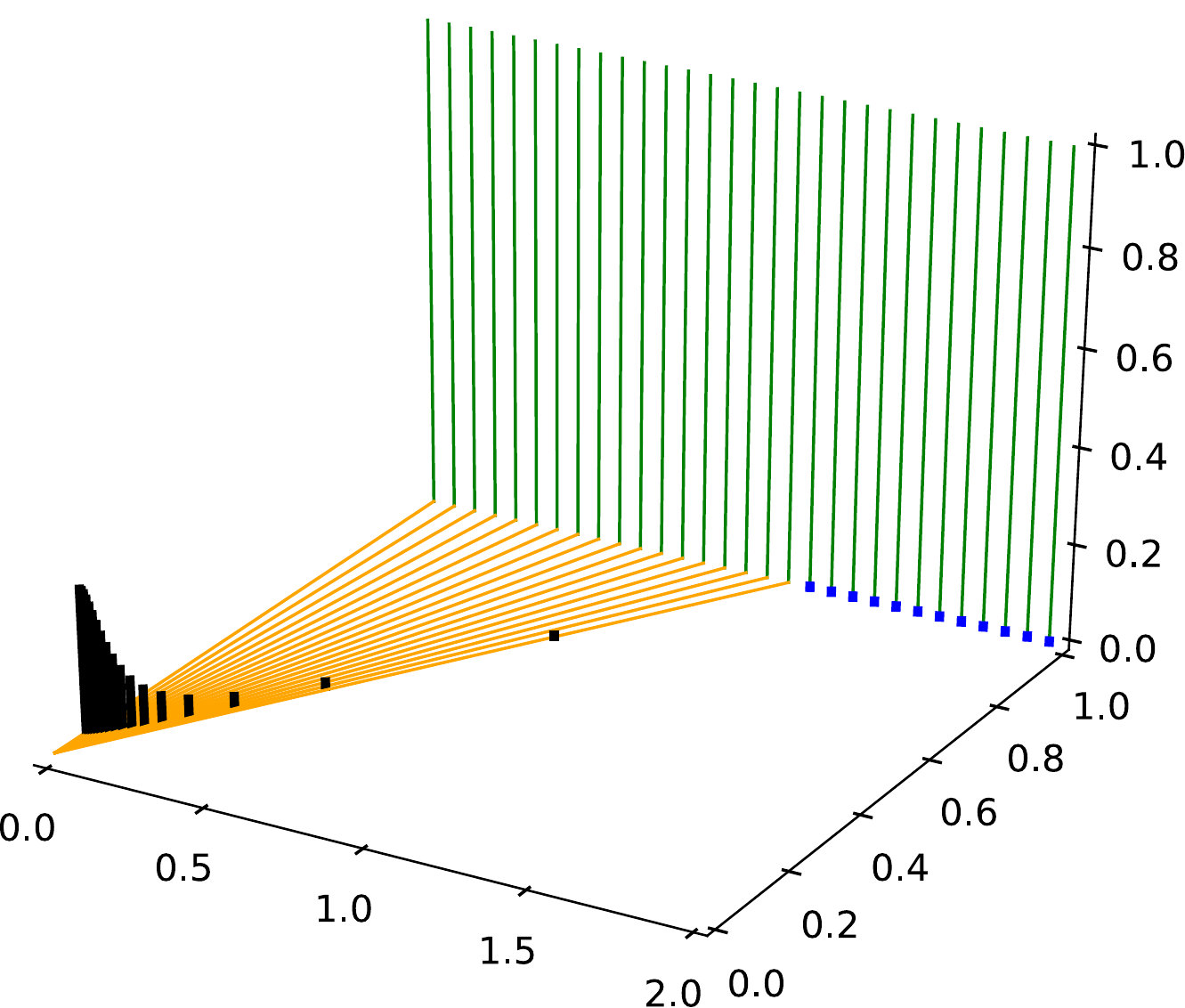}
\caption{
Geodesic curve $s\mapsto \mu(s)$ connecting superposition 
of Dirac masses and single Dirac measure in Example \ref{ex:geodDiracLine}(i) for different values of $s$. 
Here green denotes the measure $\mu_1$ for $N=30$  
while the blue and the black parts correspond to the parts of $\mu(s)$ 
that lie below and above the threshold $\pi/2$. 
The orange curves are the transport lines.}\label{fig:geodesics}
\end{figure}

\begin{figure}  
\includegraphics[width=0.245\textwidth]{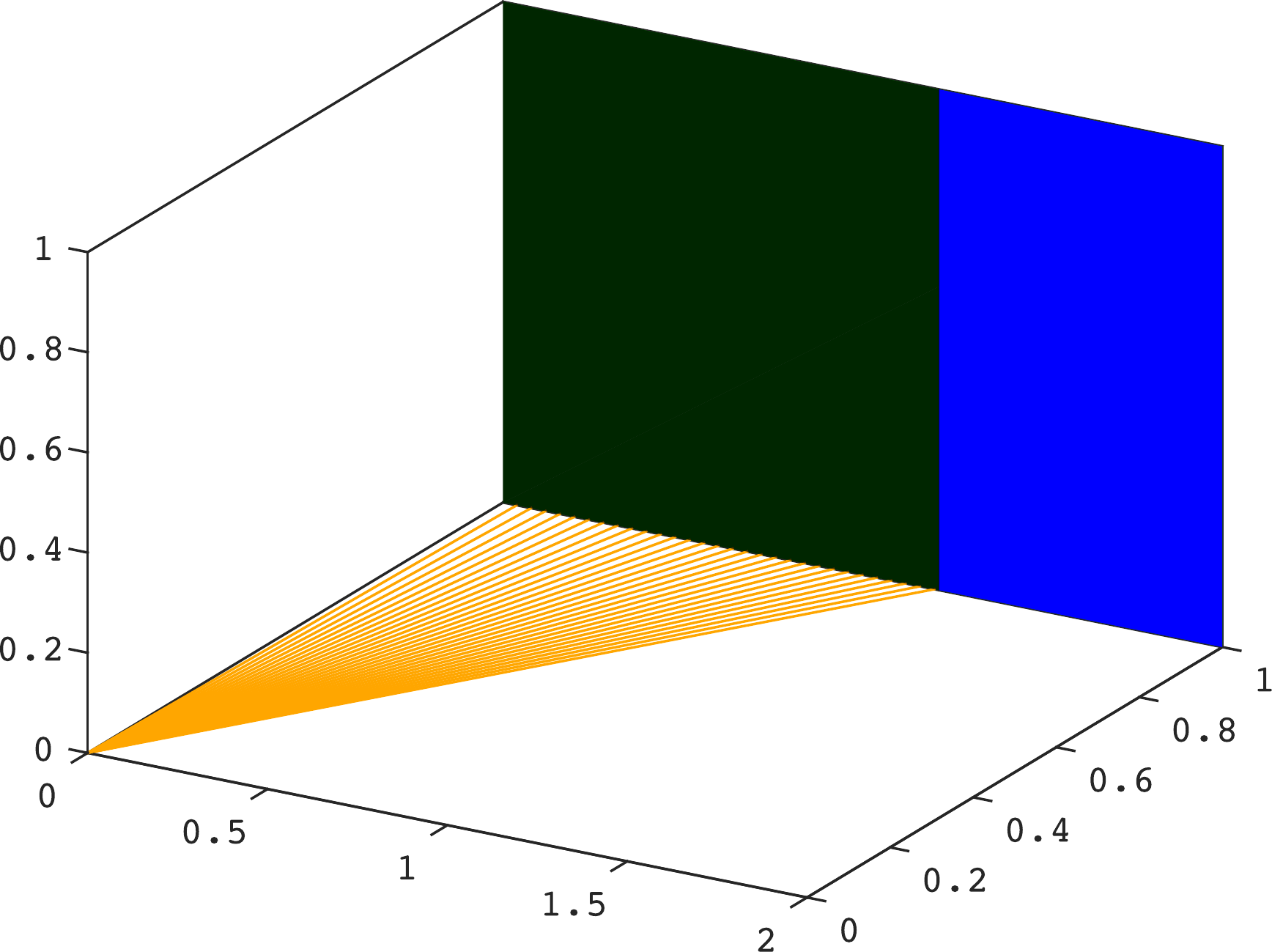}
\includegraphics[width=0.245\textwidth]{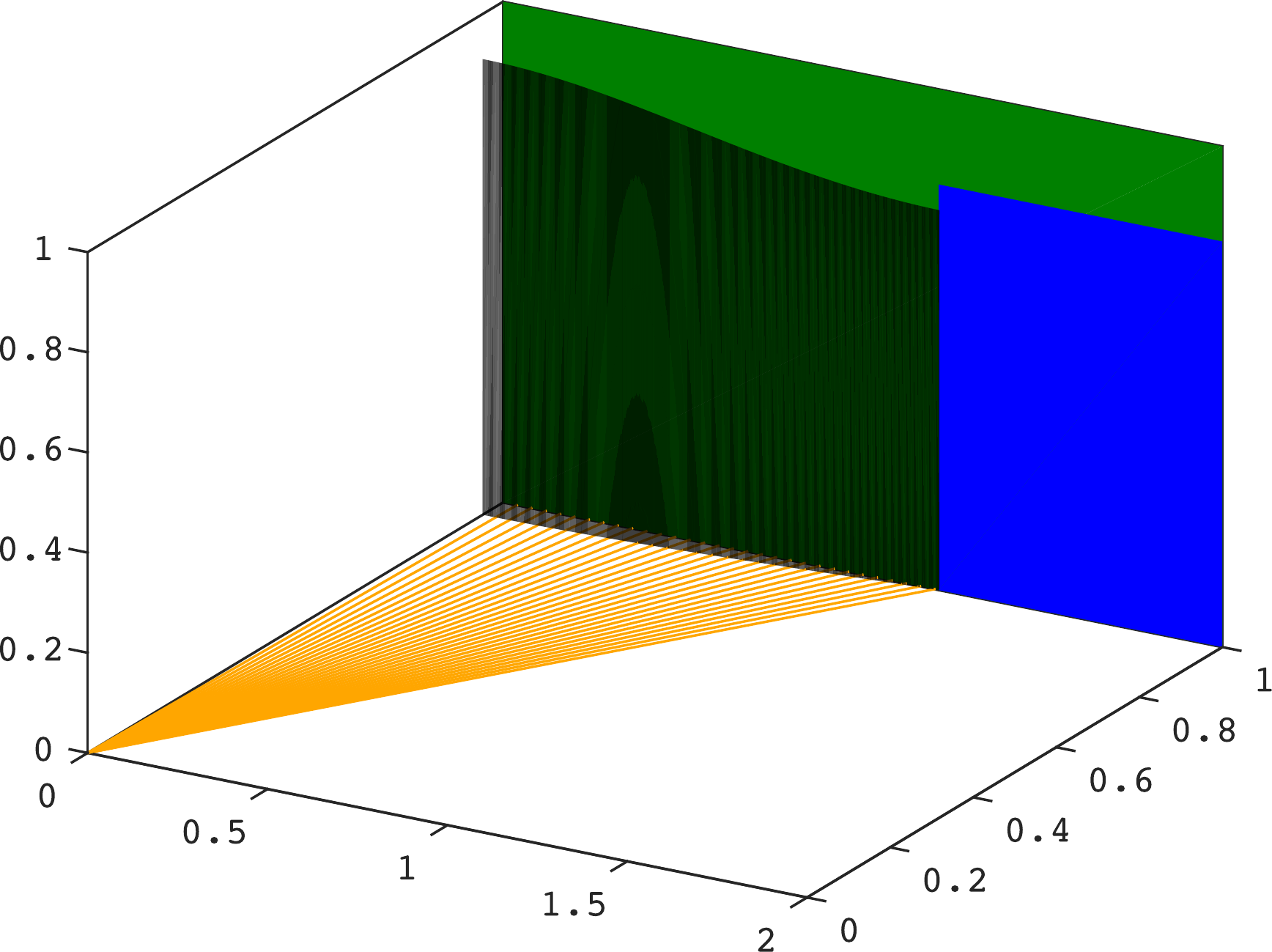}
\includegraphics[width=0.245\textwidth]{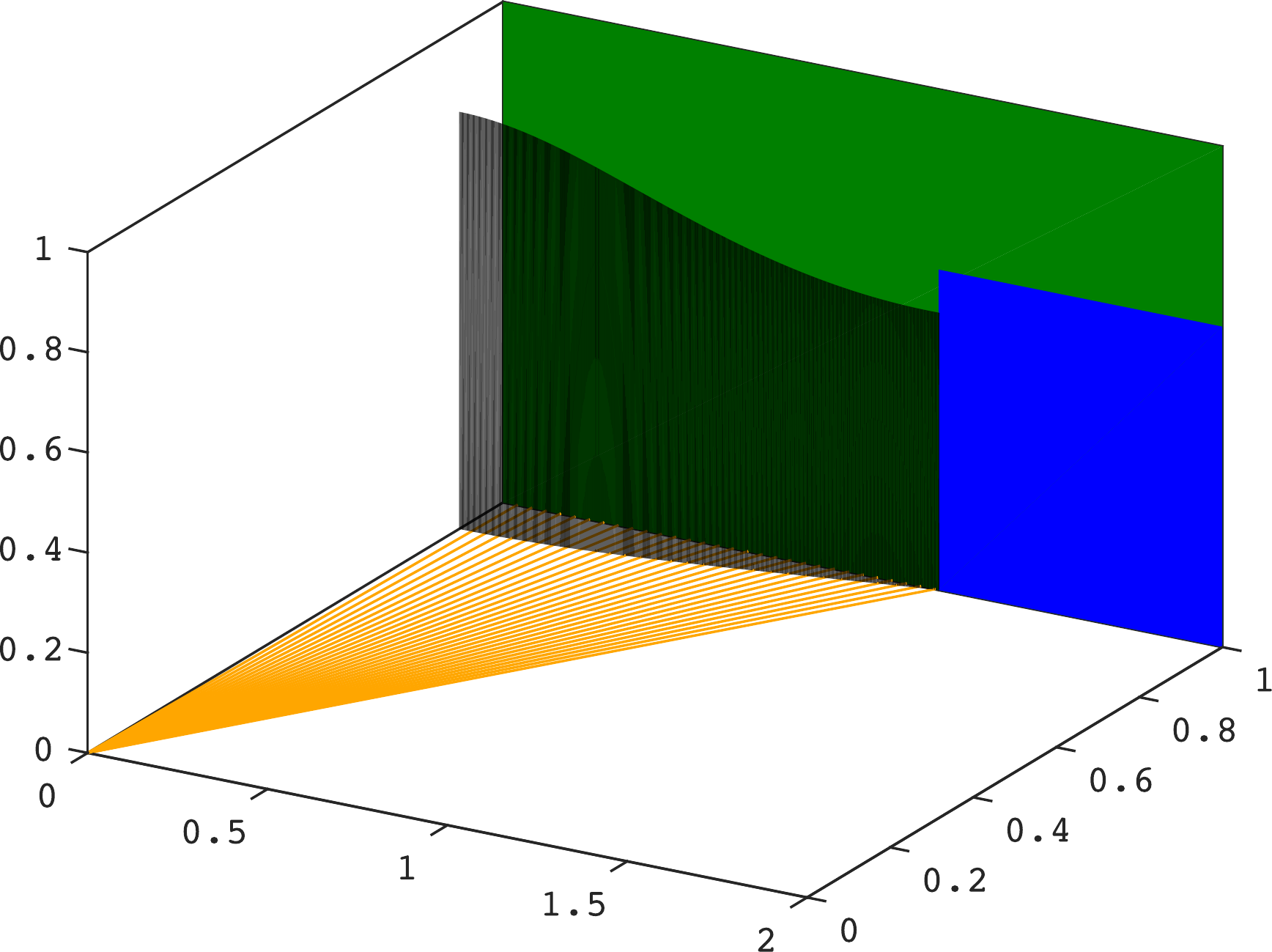}
\includegraphics[width=0.245\textwidth]{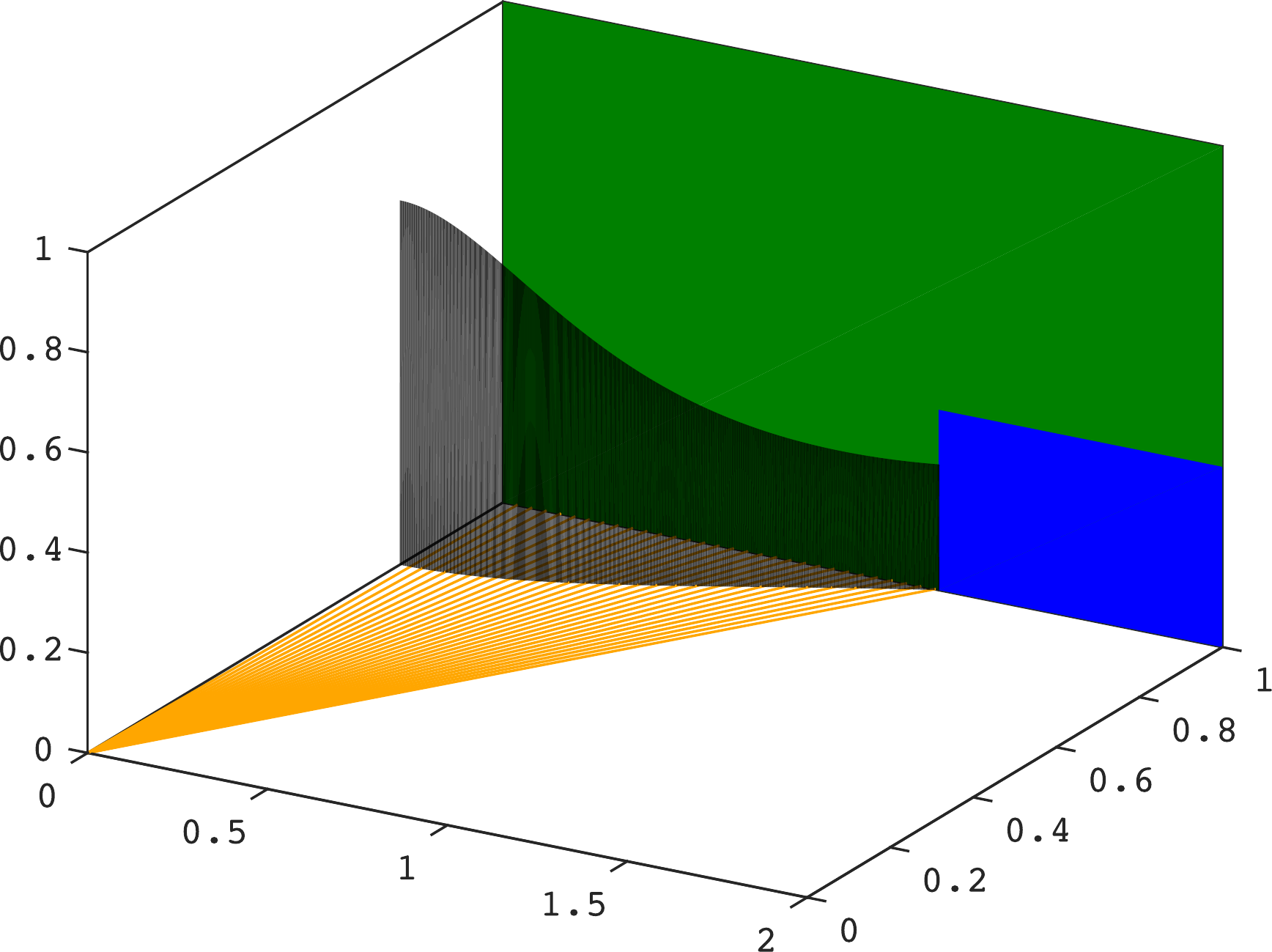}
\medskip

\noindent
\includegraphics[width=0.245\textwidth]{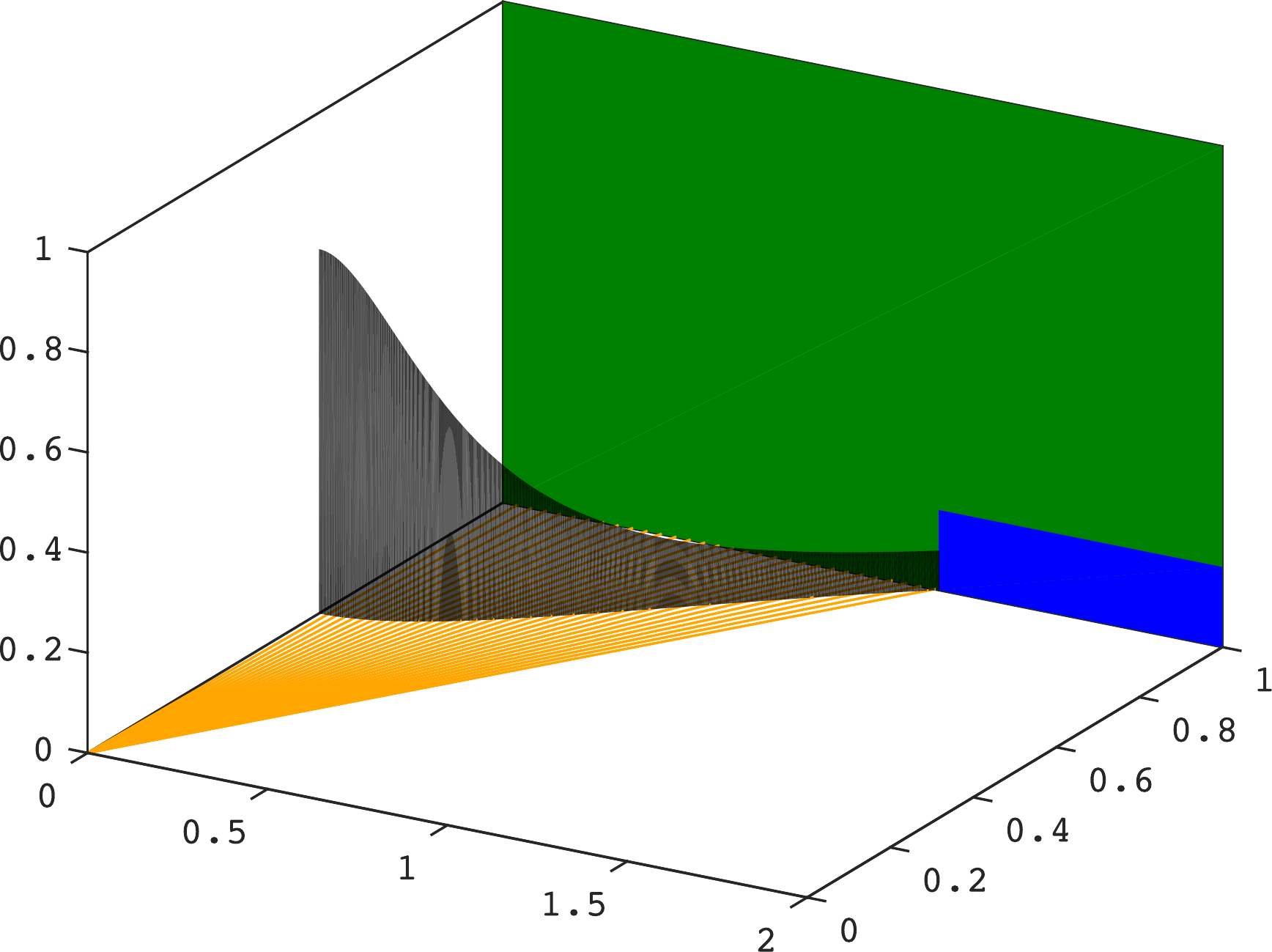}
\includegraphics[width=0.245\textwidth]{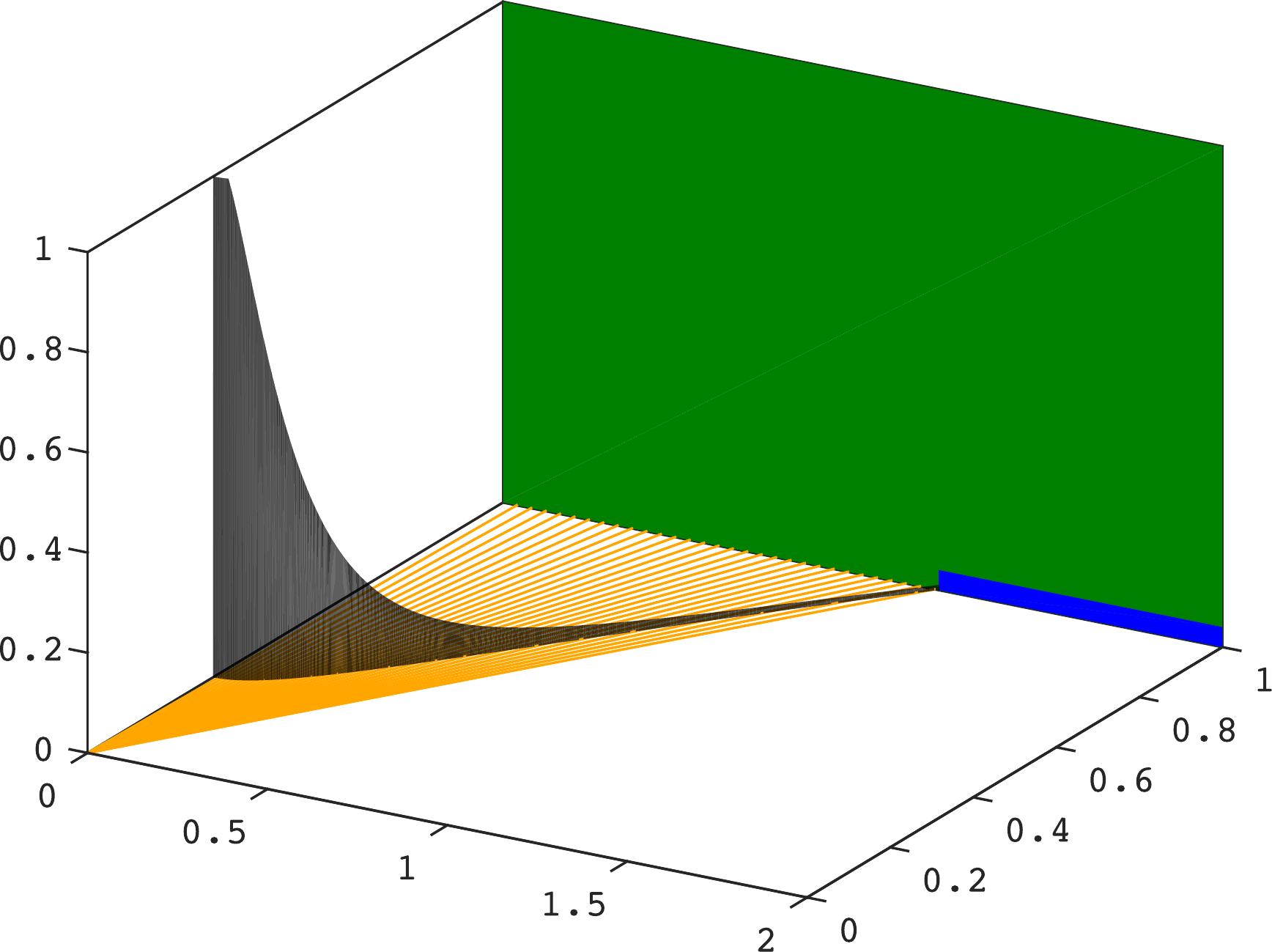}
\includegraphics[width=0.245\textwidth]{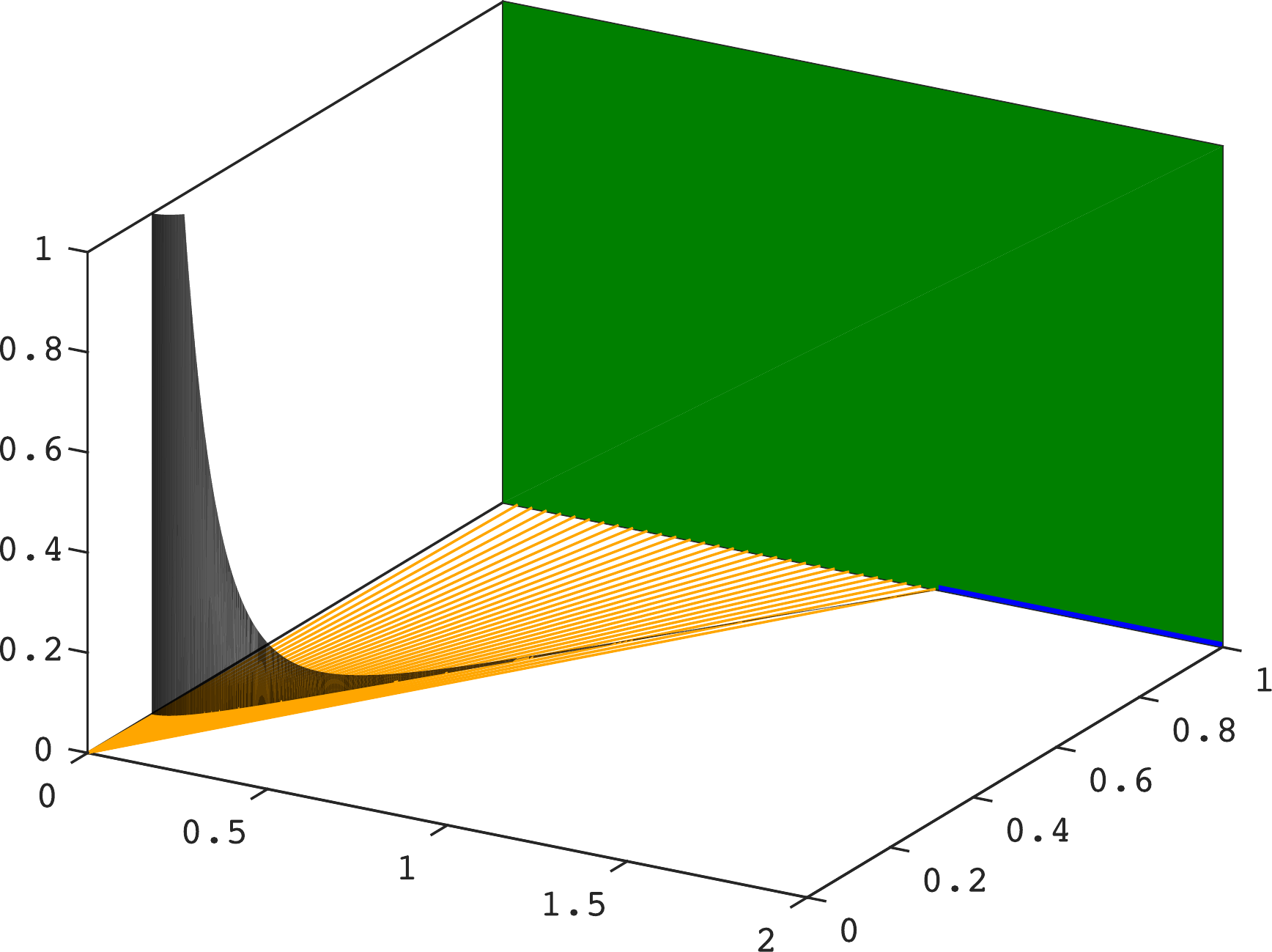}
\includegraphics[width=0.245\textwidth]{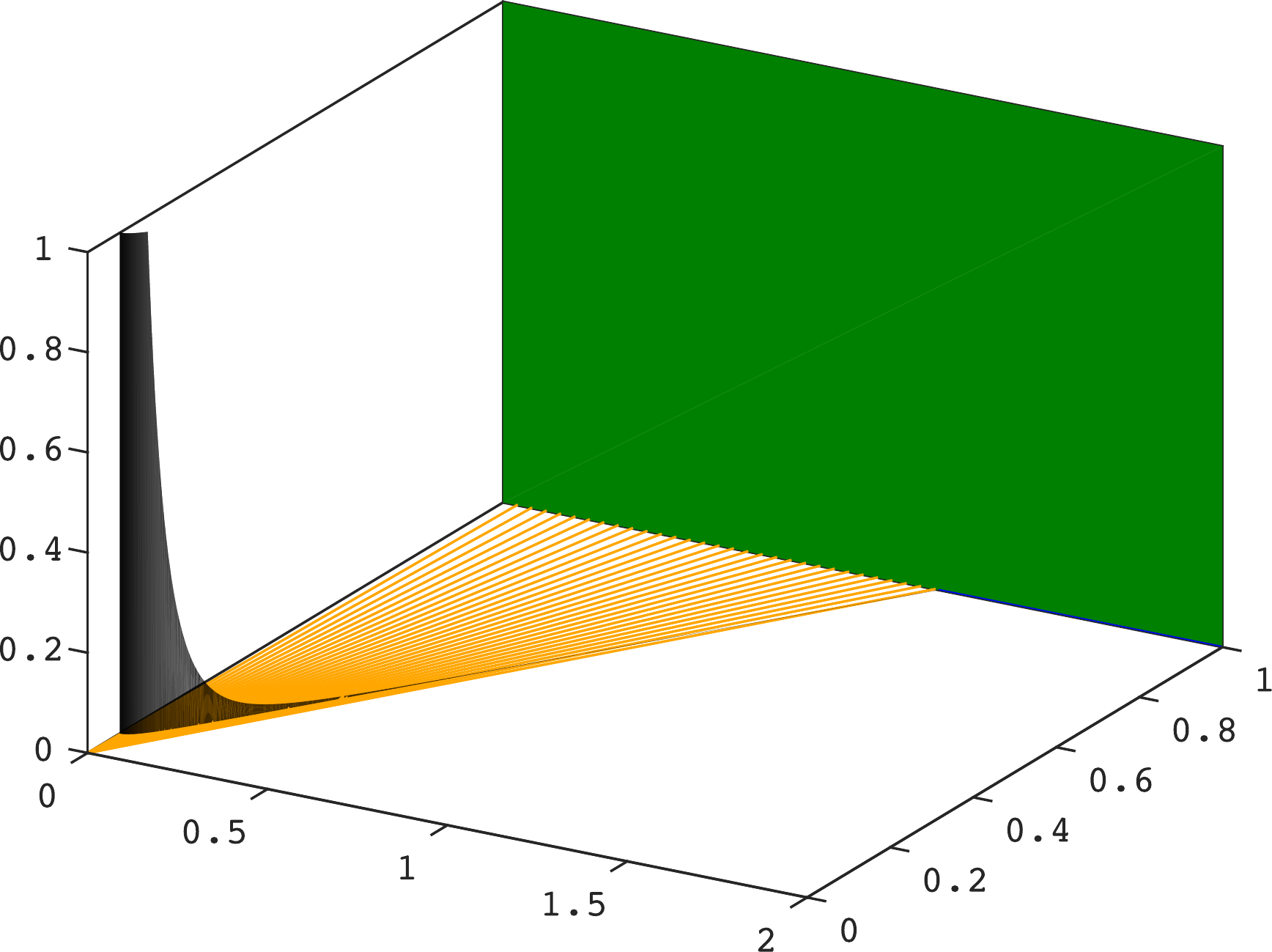}
\caption{
Geodesic curve $s\mapsto \mu(s)$ connecting line measure and single Dirac 
measure in Example \ref{ex:geodDiracLine}(ii) for different values of $s$. 
Here green denotes the measure $\mu_1$ 
while the blue and the black parts correspond to the parts of $\mu(s)$ 
that lie below and above the threshold $\pi/2$. 
The orange curves are the transport lines.
}\label{fig:geodesicsCont}
\end{figure}

\subsection{Dilation of measures}
\label{ss:Dilations}

For the Kantorovich--Wasserstein distance there is a geodesic connection
between any measure $\mu_1$ and the Dirac measure $\mu_0=\mu_1(\Omega)\delta_{y_0}$   
by radially dilating the measure, viz.\ 
\[
\mu^\mafo{KaWa}(s) =X^\mafo{KaWa}(s,\cdot)_\# \mu_1 \ \text{ where }
X^\mafo{KaWa} (s,x)= (1{-}s)y_0+ sx .
\]
This dilation corresponds to the solution $\xi(s,x)=\frac12
|x{-}y_0|^2/s$ of the standard Hamilton-Jacobi equation
$\tfrac{\rmd}{\rmd s}\xi+\frac12|\nabla\xi|^2=0$, and $\tfrac{\rmd}{\rmd s}\mu(s)+ \div\big(
\mu\nabla \xi)=0$.  

A possible generalization of this dilation to the Hellinger--Kantorovich
distance is given by a solution $\xi$ of the modified Hamilton--Jacobi
equation $\tfrac{\rmd}{\rmd s}\xi+\frac12|\nabla\xi|^2+2\xi^2=0$ having the form 
\[
\xi(s,x)= \frac{ \zeta(x)}{2s} \ \text{ with } |\nabla \zeta(x)|^2 + 4 \zeta(x)^2 -
4 \zeta(x)\equiv 0.
\]
The trivial solutions $\zeta\equiv 0$ and $\zeta \equiv 1$ correspond to
constant and pure Hellinger geodesics, respectively. However, there
are many other solutions, e.g.\ 
\[
\wt \zeta(x)= \left\{ \ba{cl}\big(\sin(|x|)\big)^2&\text{for }  
   |x|\leq \pi/2,\\ 1 &\text{for }|x|\geq \pi/2; 
\ea\right. \ \text{ and }\ol \zeta(x)= \min\set { \wt \zeta(x^k{-}y_0^k) }{
k=1,\ldots,d}
\]
if $|y_0^j{-}y_0^k|\geq \pi$ for all $j\neq k$. 

Staying with $\zeta=\wt \zeta$ we see that an arbitrary measure $\mu_1$ can be
connected to $\mu_0=a_0\delta_0$, with $a_0>0$ fixed, by the geodesic connection $\mu(s)=\mu_s$
given via
\begin{align*}
&\int_\Omega \psi(x)\dd \mu_s(x) = \int_{\Omega\cap\{|x|\geq \pi/2\}}s^2\psi(x) \dd\mu_1(x)  \\
&\qquad +\int_{\Omega\cap \{|x|<\pi/2\} } 
\big((1{-}s^2)(\cos|x|)^2+s^2\big) \psi\Big( \arctan\big[ s
\tan|x| \big] \frac x{|x|}  \Big) \dd \mu_1(x),
\end{align*}
where the first term on the right-hand side denotes the pure Hellinger
part, while the second term involves the concentration into
$a_0\delta_0$ for $s\searrow 0$, where the total mass at $s=0$ equals
$a_0:= \int_{\{|x|<\pi/2\}} (\cos|x|)^2\dd \mu_1$. 

Again it is easy to show that the pair $(\mu,\xi)$ with $\xi(s,x)=\wt
\zeta(x)/(2s)$ satisfies the formal equation \eqref{eq:GC.eqn} for
geodesic curves. We also note that the dilation operation is unique,
even if $\mu_1$ has positive mass on the sphere $
\{|x|=\pi/2\}$. This is because of the fixed function $\xi$. Of course
there might be other geodesic curves connecting $\mu_0 = a_0\delta_0$
and $\mu_1$, e.g.\ for $\mu_1=a_1\delta_{y_1}$ with $|y_1|=\pi/2$, where we
have all the solutions constructed in Section \ref{ss:AllDiracM}. 

\begin{example}\label{ex:geodDiracLine}
(i) As a more concrete example, we consider in $\Omega=[0,1]\times[0,2]$ the  measures
$\mu_1= \sum_{k=1}^N\delta_{x_k}$, $N\geq 2$, and $\mu_0=a_0\delta_0$, where
\[
x_k=\binom10 +\frac {k-1}{N-1} \binom02 \quad\text{for } k=1,\dots,N,
\quad\text{and}\quad
a_0=\sum_{k:|x_k| < \pi/2}\cos(|x_k|)^2.
\]
Using the formula above, we obtain the geodesic connection
\begin{align*}
\mu(s) &= \sum_{k:|x_k|\geq \pi/2}s^2 \delta_{x_k}
+\sum_{k:|x_k| < \pi/2}\big((1{-}s^2)(\cos|x_k|)^2+s^2\big)\delta_{\rho_k(s)x_k},
\\[0.5em]&\quad\text{where}\quad
\rho_k(s)=\frac{1}{|x_k|}\arctan\big[s\tan|x_k|\big].
\end{align*}
The geodesic connection $\mu(s)$ is depicted in Figure~\ref{fig:geodesics}
for different values of $s$.

(ii) Similarly, we can compute a geodesic connection $\mu(s)$ for the line
measure $\mu_1 = \delta_1\otimes\calL^1|_{[0,2]}$
which is collapsed into the measure $\mu_0=a_0\delta_0$
with $a_0 = \int_0^2(\cos(y))^2\dd y$. In this case, $\mu(s)$
is concentrated on the set given by the function
\[
X(s;x)=
\begin{cases}
\rho(s;x)x,&\text{for }|x|\leq \pi/2,\\
x&\text{otherwise},
\end{cases} \quad \text{ for } s\in [0,1] \text{ and } x\in\mathrm{supp}\,\mu_1,
\]
where $\rho(s;x) = \arctan(s\tan|x|)/|x|$. On these curves the density with respect
to the one-dimensional Hausdorff measure is for $y\in\Omega$ and 
$\wt X_s(x_2) := X(s;(1,x_2))$ for $x_2\in[0,2]$, given by
\[
\widetilde{a}(s,y)=\frac{a(s;\cdot)}{|\pl_{x_2} X(s;\cdot)|}\circ X_s^{-1}(y)
\]
where the profile $a$ reads
\[
a(s;x)= 
\begin{cases}
(1{-}s^2)(\cos|x|)^2+s^2&\text{for }|x|\leq\pi/2,\\
s^2 &\text{otherwise}.
\end{cases}
\]
The curve $\mu(s)$ is shown in Figure \ref{fig:geodesicsCont}.

\end{example} 

\begin{figure}\centering
\includegraphics[width=0.5\textwidth]{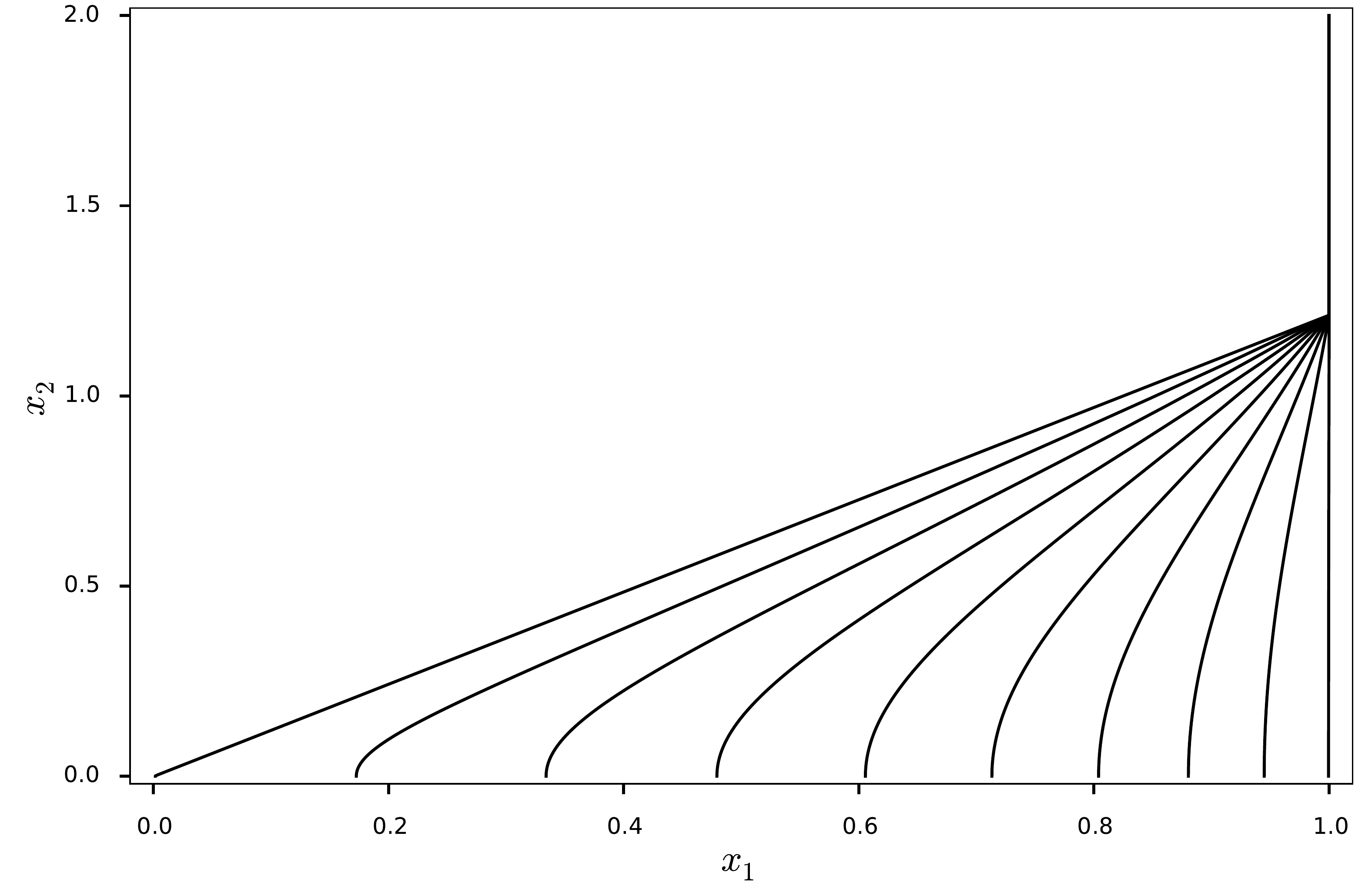}
\caption{Supports of geodesic curve $\mu(s)$ in Example~\ref{ex:geodDiracLine}(ii) 
for different values of $s$.}
\label{fig:transportLines}
\end{figure}
 
\subsection{Transport of characteristic functions}
\label{ss:CharFunct}

Here, we discuss a method to explicitly construct the geodesic
connection between two characteristic functions $\mu_j
=a_j\chi_{[x_j^\mafo{left}, x_j^\mafo{right}]}$ where $\Omega
\subset\R^1$. However, to simplify notations we will restrict to the
specific case 
\[
\mu_0=\chi_{[-\pi/4,  \pi/4]}\rmd x\ \text{ and } \ 
\mu_1=\chi_{[\pi/2,  \pi]}\rmd x.
\]
Obviously, we have the Hellinger parts $\mu_0^\bot = \chi_{[-\pi/4,
  0]}\rmd x$ and $\mu_1=\chi_{[3\pi/4,  \pi]}\rmd x$, which are
absorbed and generated, respectively, without any interaction with the
transport in between. 

To construct a transport geodesic from $\mu_0^\mafo{tr} = \chi_{I_0}
\rmd x$ and $\mu_1^\mafo{tr}=\chi_{I_1}\rmd x$, with $I_0={]0, \pi/2[}
$ and $I_1= {]\pi/2,3\pi/4[}$, we find the functions $r_j:I_j\to \R$
via minimizing the entropy-transport functional
$\calE\calT(\eta;\mu_0,\mu_1)$. We establish the calibration measure $\eta$ via a
map $h:I_0\to I_1$ in the form
\[
\int_{I_0\ti I_1} \Psi(x_0,x_1) \dd \eta(x_0,x_1) = \int_{I_0}
\Psi(x,h(x)) f(x) \dd x.
\]
Checking the marginal conditions $\Pi_\#^j \eta = \varrho_j \rmd
x$ we find $f(x)=\varrho_0(x)= \varrho_1(h(x))h'(x)$. Moreover, the
optimality conditions in Theorem \ref{th:CharDHK} give, for all $x\in I_0$ and $y\in I_1$,  
\begin{equation}
  \label{eq:h-optim}
  \varrho_0(x)\varrho_1(h(x))=\big[\cos(h(x){-}x)\big]^2 
\ \text{ and } \ 
\varrho_0(x) \varrho_1(y) \geq \big[\cos(y{-}x)\big]^2.
\end{equation}
Deriving the first-order optimality conditions  at $y=h(x)$ from the second
relation in \eqref{eq:h-optim} we find 
\[
2\sin(h(x){-}x)\cos(h(x){-}x)h'(x) =\varrho'_0(x)\varrho_1(h(x))h'(x)=
\varrho_0(x)\varrho'_0(x)=\frac12\big(\varrho_0(x)^2\big)'.
\]
Since the first relation in \eqref{eq:h-optim} has the form 
$\varrho_0(x)^2=h'(x)\big[\cos(h(x){-}x)\big]^2 $  we find  
\[
h''(x)=2\big( h'(x)^2{+}h'(x)\big) \tan\big(h(x){-}x),\quad h(0)=\pi/2,\ 
h(\pi/4)=3\pi/4,
\]
which has a unique monotone solution. Indeed, to see this let
$h(x)=\pi/2+x-w(x)$, where now $w(0)=w(\pi/4)=0$ and $w>0$. 
Then the ODE reads $w''=b(w')c(w)$ for suitable functions $b$
and $c$. Rewriting it  in the form $w'w''/b(w')=c(w)w'$ we find 
$A(w(x))=C(w(x))+ \gamma$, where $B'(y)=y/b(y)$ and $C'(y)=c(y)$. An
explicit calculation and exponentiating both sides yields 
\[
\frac{\sqrt{1{-}w'(x)}}{2{-}w'(x)} = c_* \sin w(t) , \quad
  w(0)=w(\pi/4)=0.
\]
Solving for $w'$ we find $w'=g_\pm(c_*\sin w)$ with $g_\pm(a)=
\big(4a^2{-}1\pm \sqrt{1{-}4a^2}\big)/(2a^2)$ for $0< a\leq 1/2$, 
where $\pm g_\pm(a)>0$ for $a\neq 1/2$ and $g_\pm(1/2)=0$. Thus, $w$
will have a unique maximum $w_*=w(w_*)$ with $c_*\sin w_*=1/2$, and $w_*$ can
be determined uniquely from
\begin{align*}
\frac\pi4 &= \int_0^{x_*} \dd x + \int_{x_*}^{\pi/4} \dd x =
\int_0^{w_*} \frac{\rmd w}{g_+(w)} + \int_0^{w_*} \frac{\rmd
  w}{-g_-(w)} \\ 
&= \int_0^{w_*} \frac{\sin w_*}{\sqrt{(\sin w_*)^2 -
    (\sin w)^2} } \dd w =K(w_*)\sin w_*,
\end{align*}
where $K$ is the elliptic $K$ function. Numerically, we find
$w_*=0.4895$ and thus $c_*= 1.0634$. 

Now it is straightforward to show that there is exactly one $c_*>0$
such that a solution with $w(x)>0$ and $w'(x)<1$ exists. 

Based on the function $h$ the densities $\varrho_0$ and $\varrho_1$ are explicitly
known and we may write the geodesic connection $\mu_s=\mu(s)$ for
$\mu_0^\mafo{tr} = \chi_{I_0} 
\rmd x$ and $\mu_1^\mafo{tr}=\chi_{I_1}\rmd x$ as in \eqref{eq:geodesic}
\begin{align*} 
&\int_{[0,3\pi/4]}\psi(y) \dd \mu_s(y)= \int_{I_0} \wh R(s,x)^2
\psi(Y(s,x)) \varrho_0(x) \dd x 
\\[0.5em] &\text{ with }
 \ol R(s,x)^2= (1{-}s)^2r_0^2(x) + s^2 r_1^2 (h(x)) + 2s(1{-}s)\\
& \text{ and } Y(s,y)=x+ \arccos\Big( \frac{(1{-}s)r_0(x)+ s r_1(h(x))
  \cos(h(x){-}x)}{ \ol R(s,x) }\Big), 
\end{align*}
where $r_j(x_j)=\big(\varrho_j(x_j))^{-1/2}$. Thus, the density
$f(s,\cdot)$ of $\mu_s$ satisfies the relation
\begin{align*}
&\int_0^{Y(s,x)}f(s,y)\dd y = \int_0^{3\pi/4} \chi_{\{y\leq Y(s,x)\}}(y) \dd
\mu_s(y) \\
&= \int_{I_o} \ol R(s,x_0)^2 \chi_{\{y\leq Y(s,x)\}}(Y(s,x_0))
\varrho_0(x_0) \dd x_0= \int_0^x \ol R(s,x_0)^2 \varrho(x_0) \dd x_0.
\end{align*}
Differentiation with respect to $x$ gives the explicit formula
\[
f(s,Y(s,x))= \frac{\ol R(s,x)^2 \varrho_0(x)}{\pl_x Y(s,x)} .
\]
In Figure \ref{fig:CharFunct}, we plot the densities together with the
corresponding Hellinger parts.  
\begin{figure}
\includegraphics[width=3.5cm]{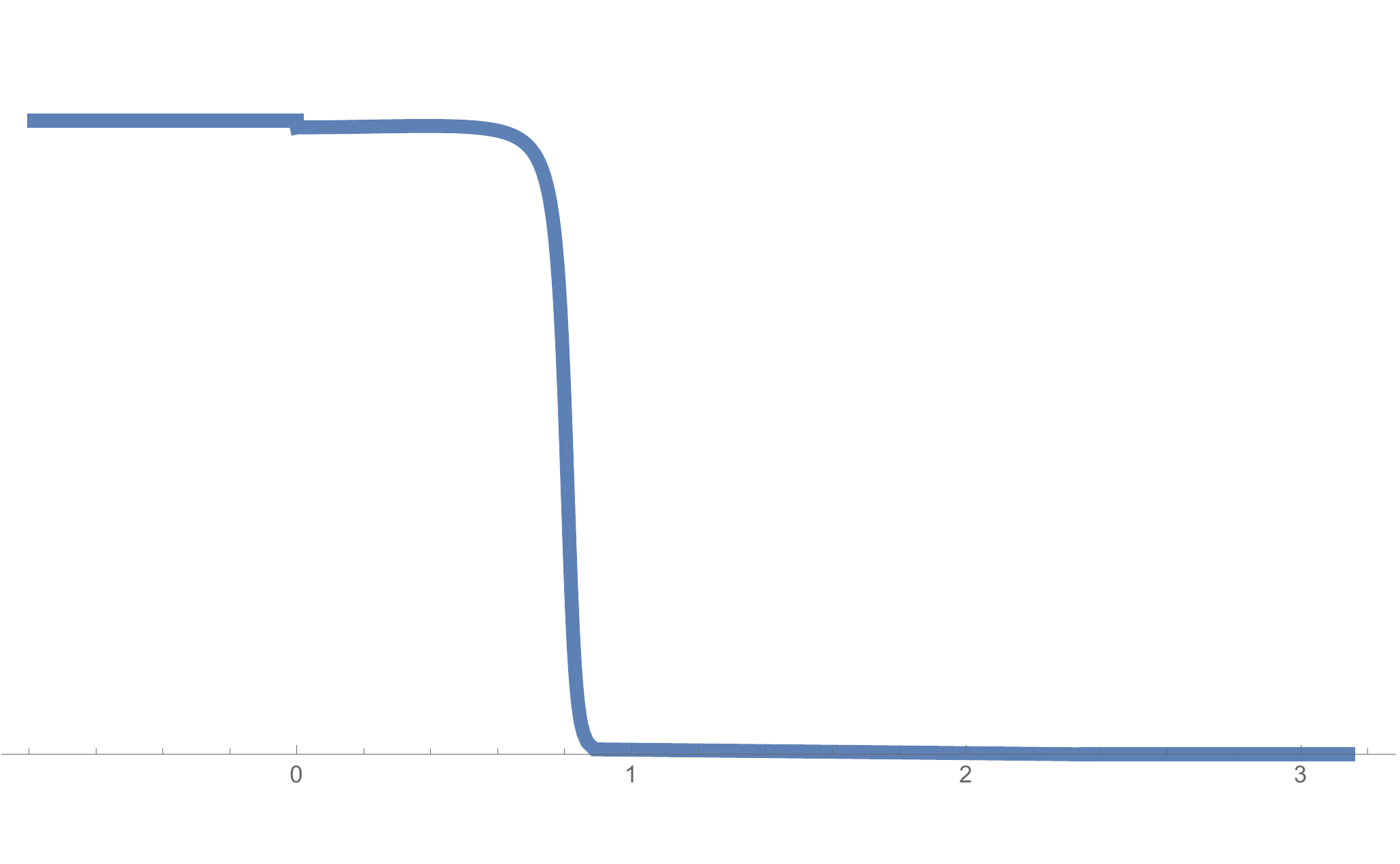}\hfill
\includegraphics[width=3.5cm]{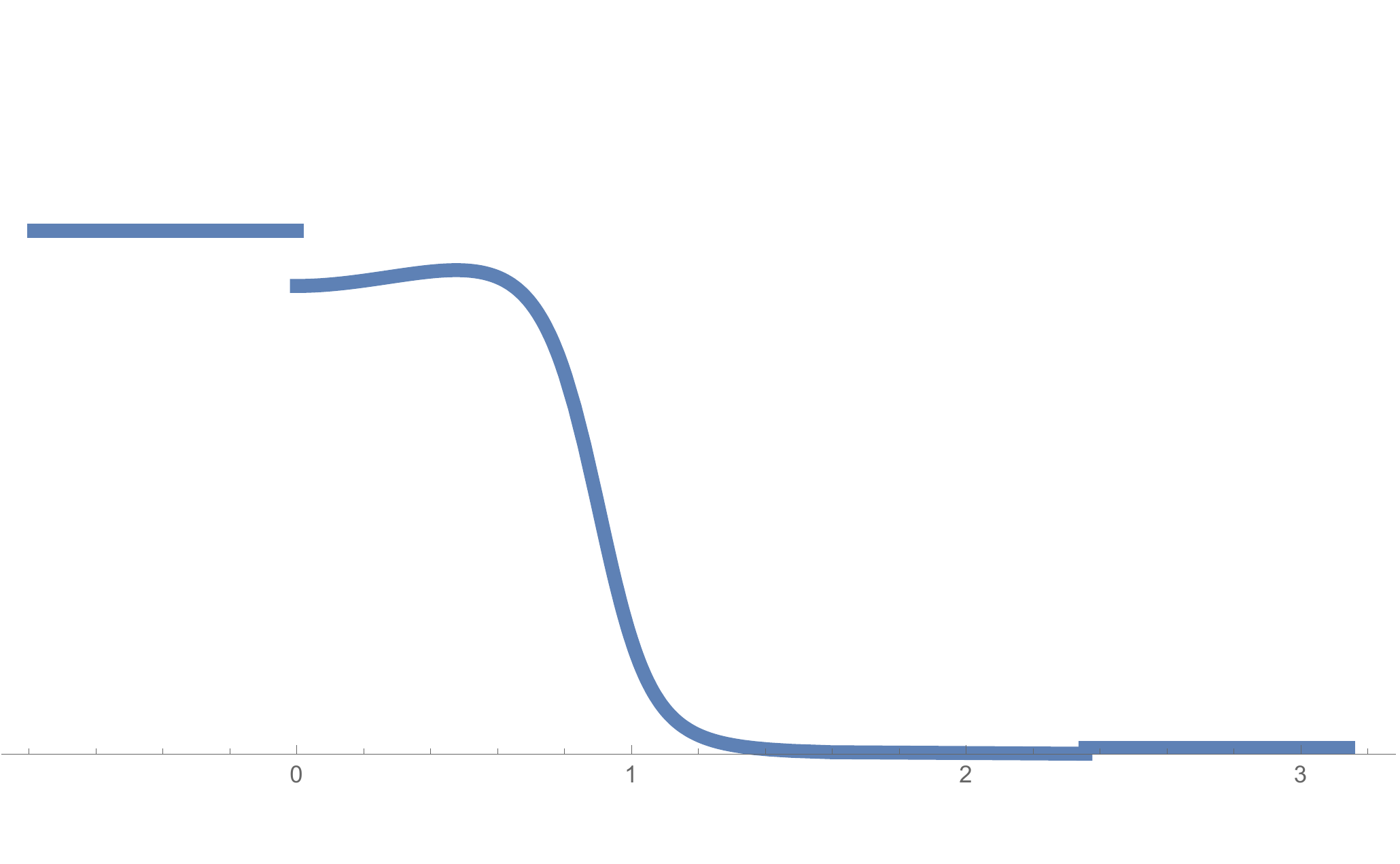}\hfill
\includegraphics[width=3.5cm]{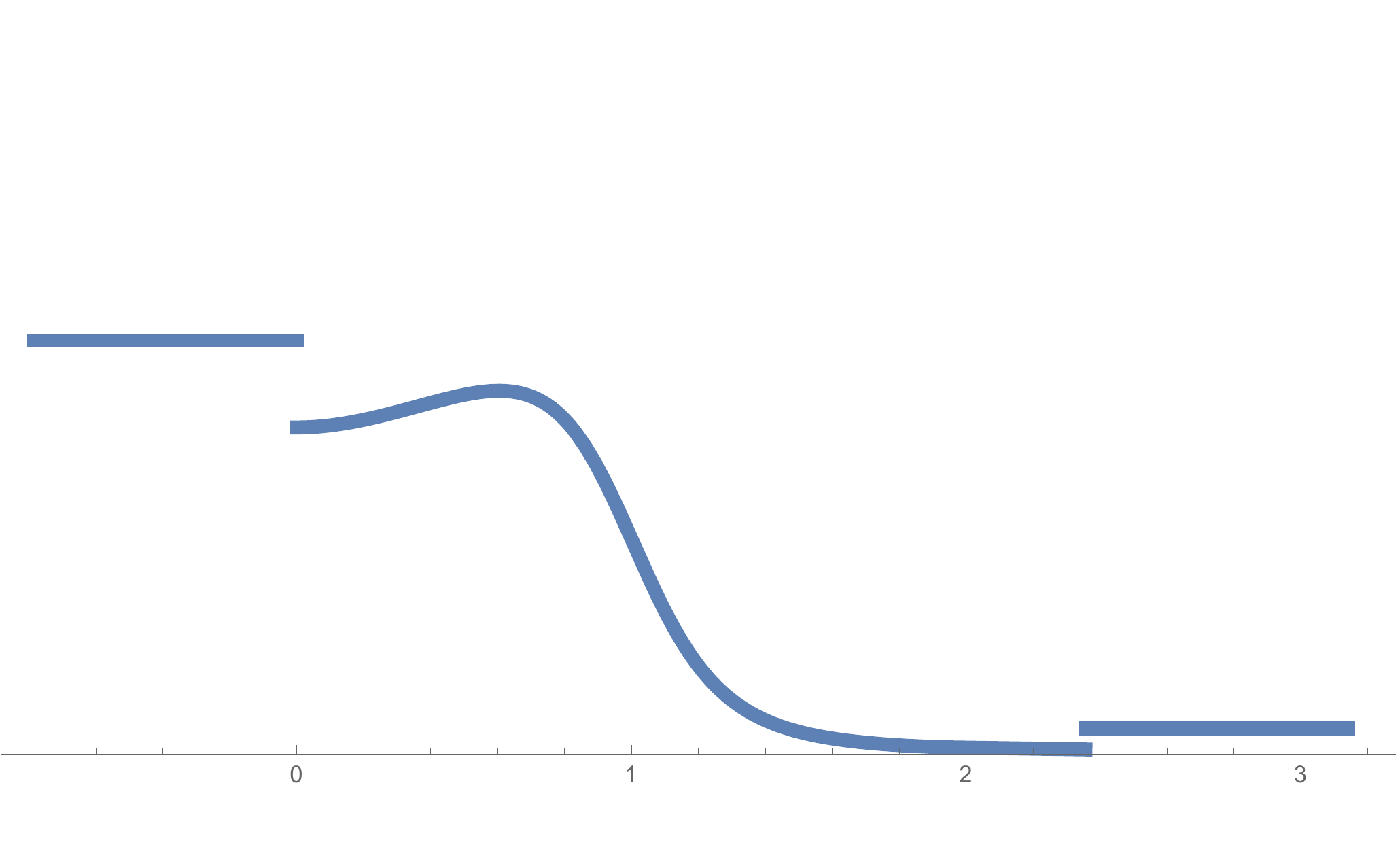}\hfill
\includegraphics[width=3.5cm]{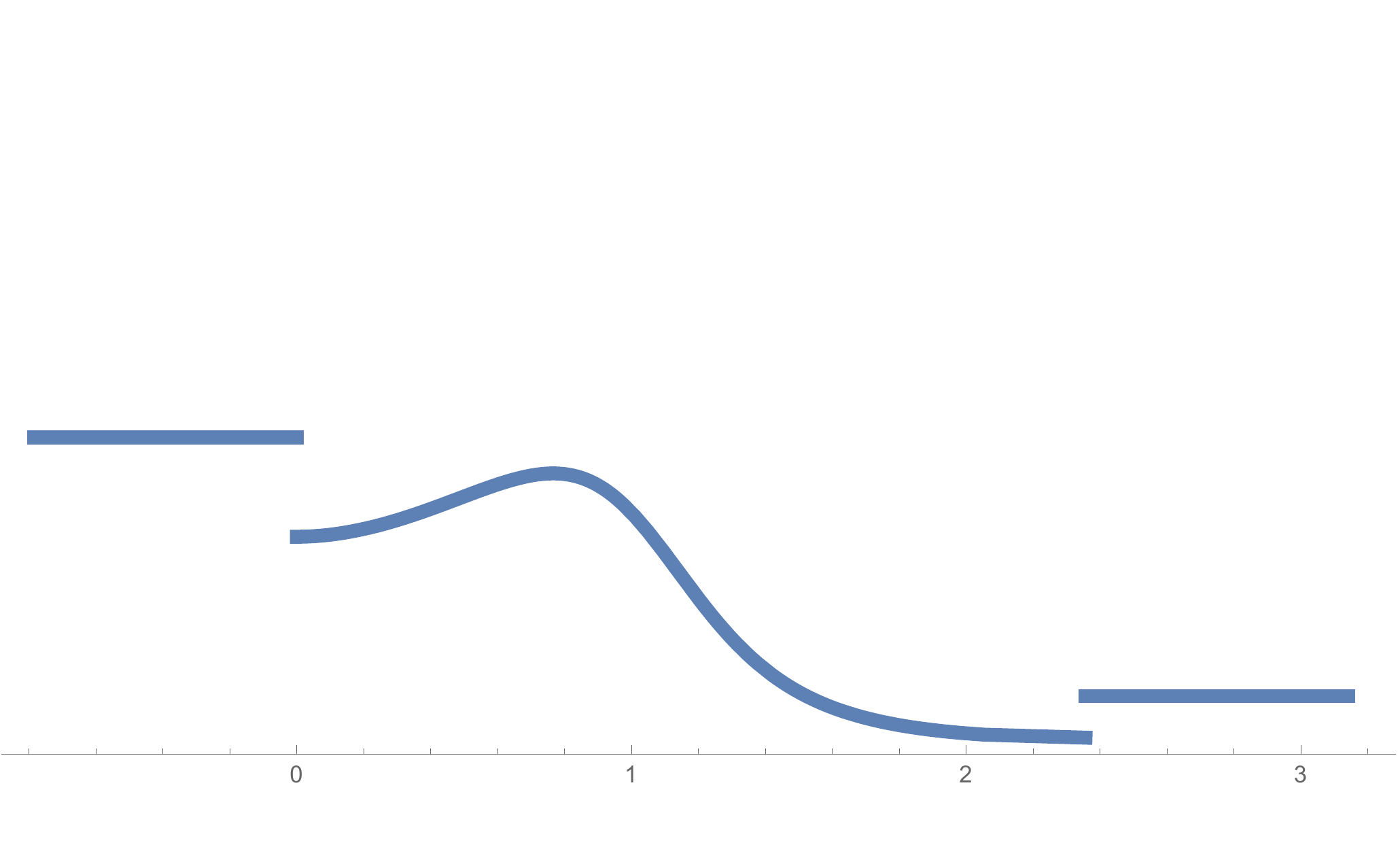}

\includegraphics[width=3.5cm]{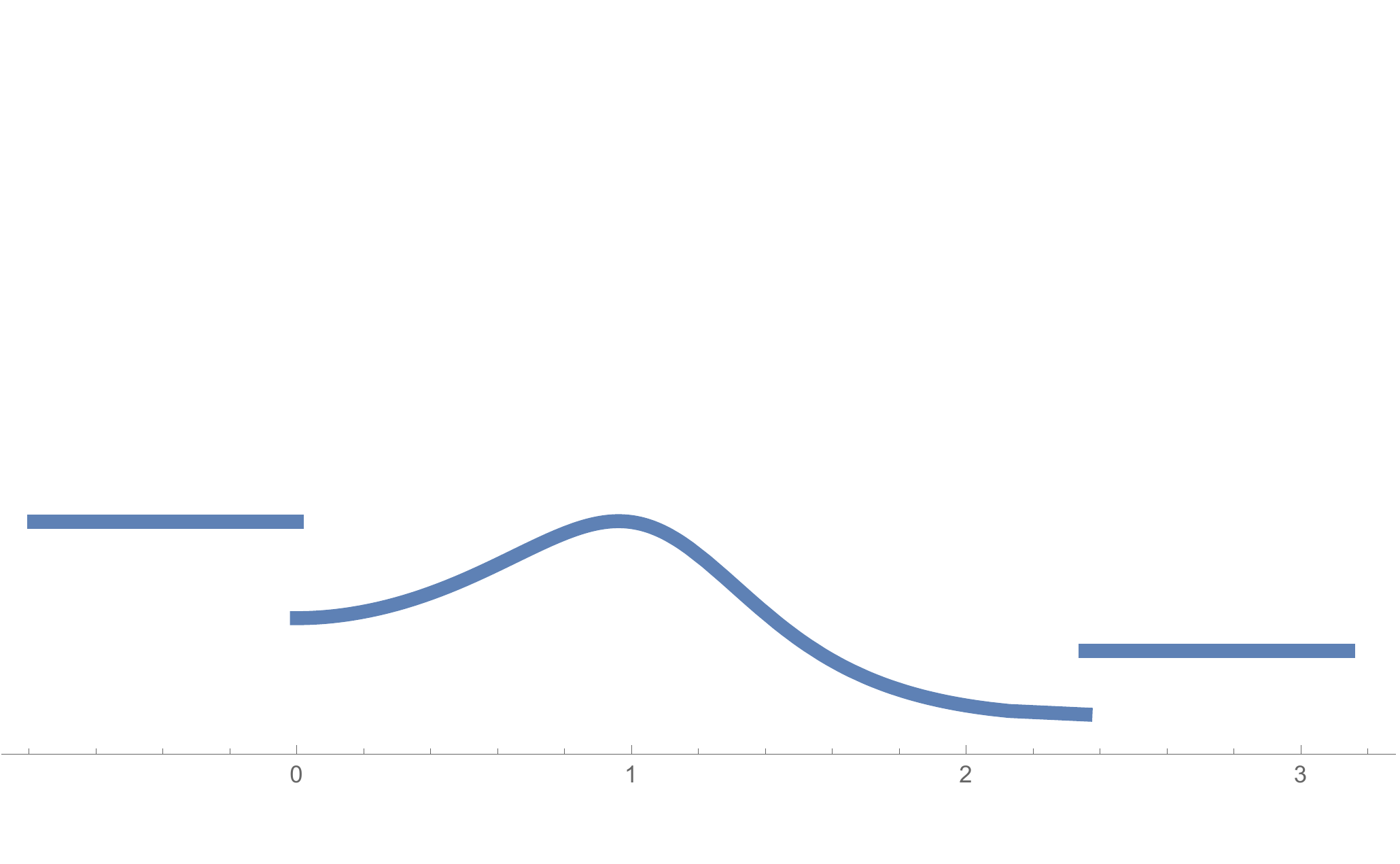}\hfill
\includegraphics[width=3.5cm]{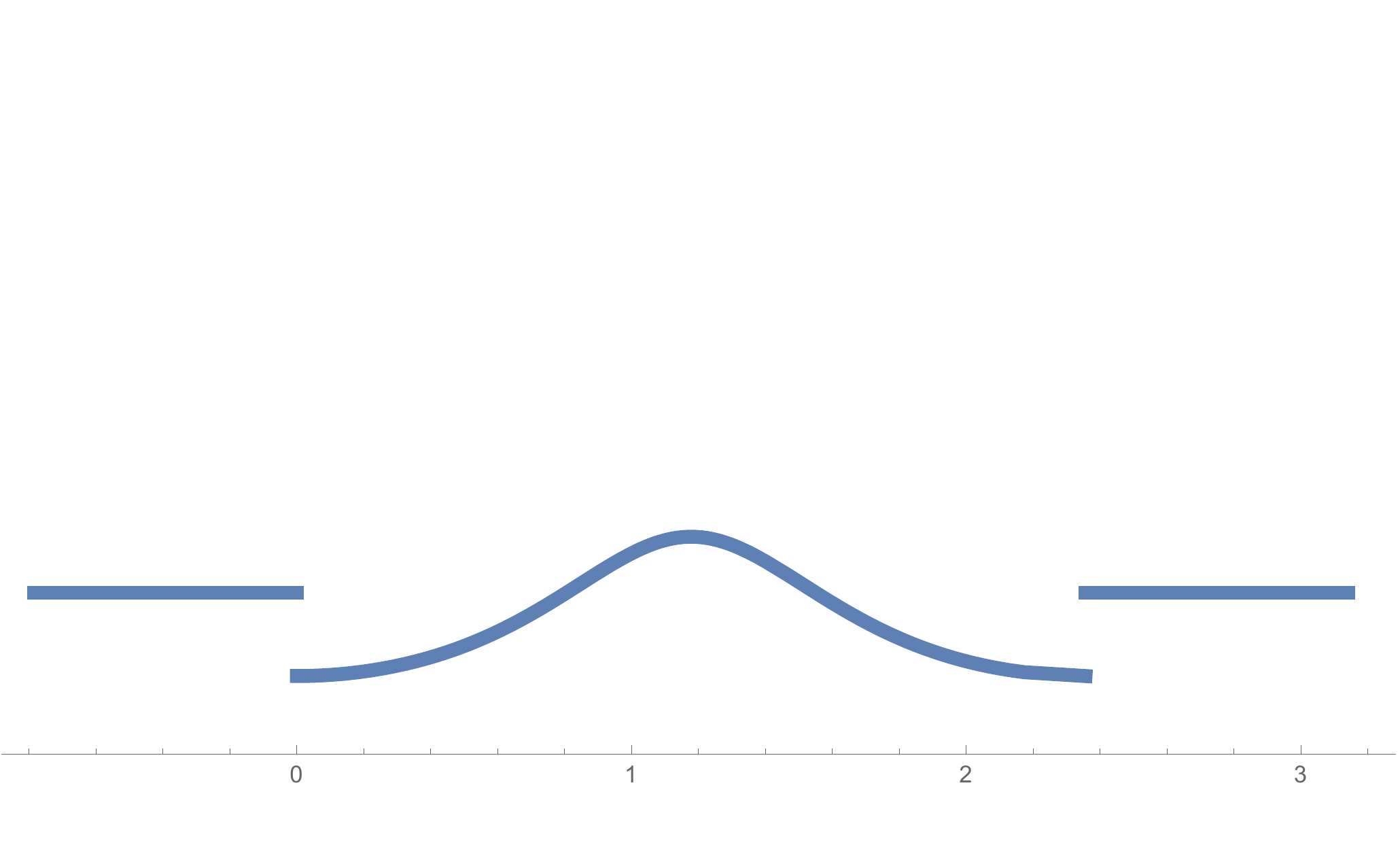}\hfill
\includegraphics[width=3.5cm]{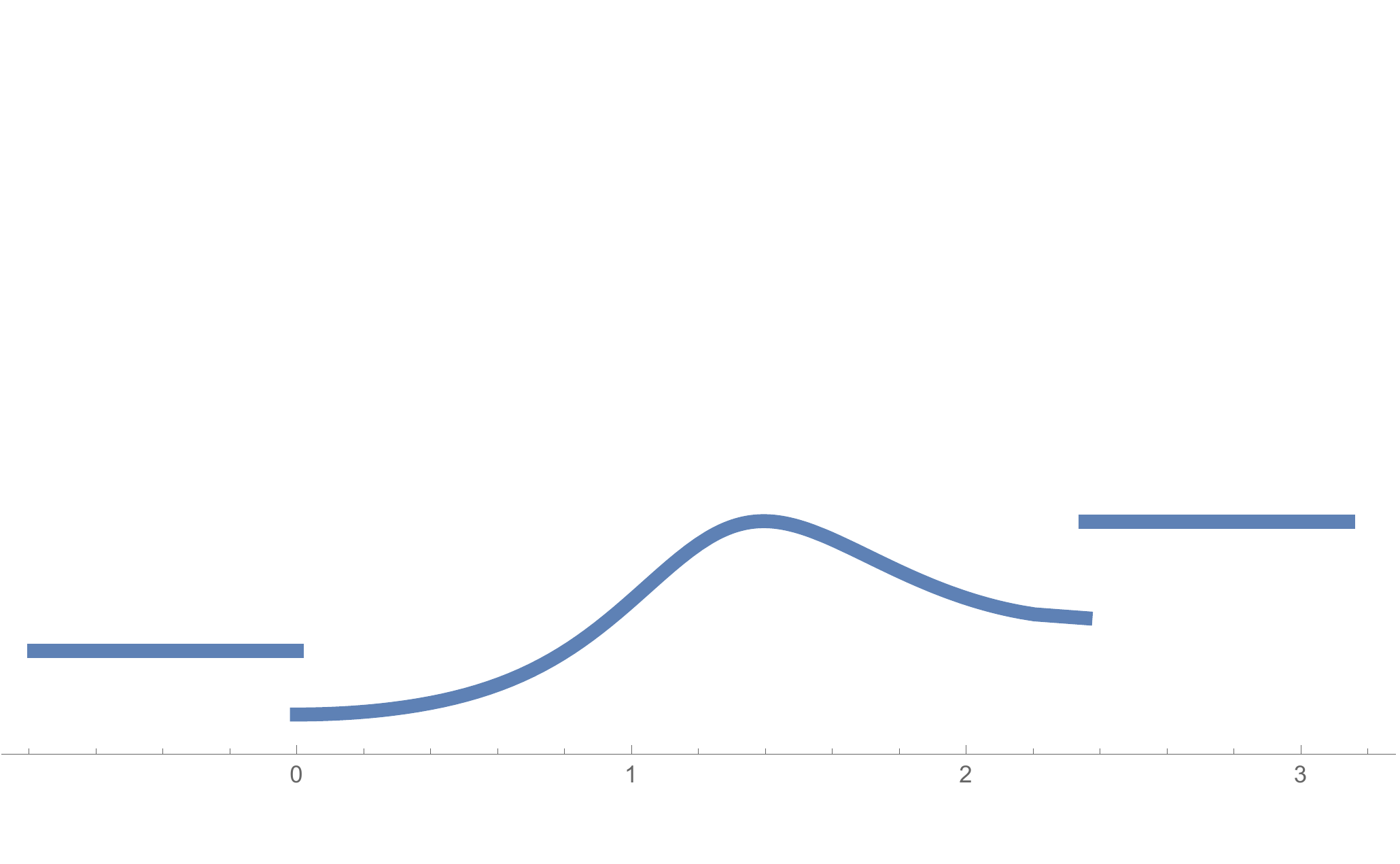}\hfill
\includegraphics[width=3.5cm]{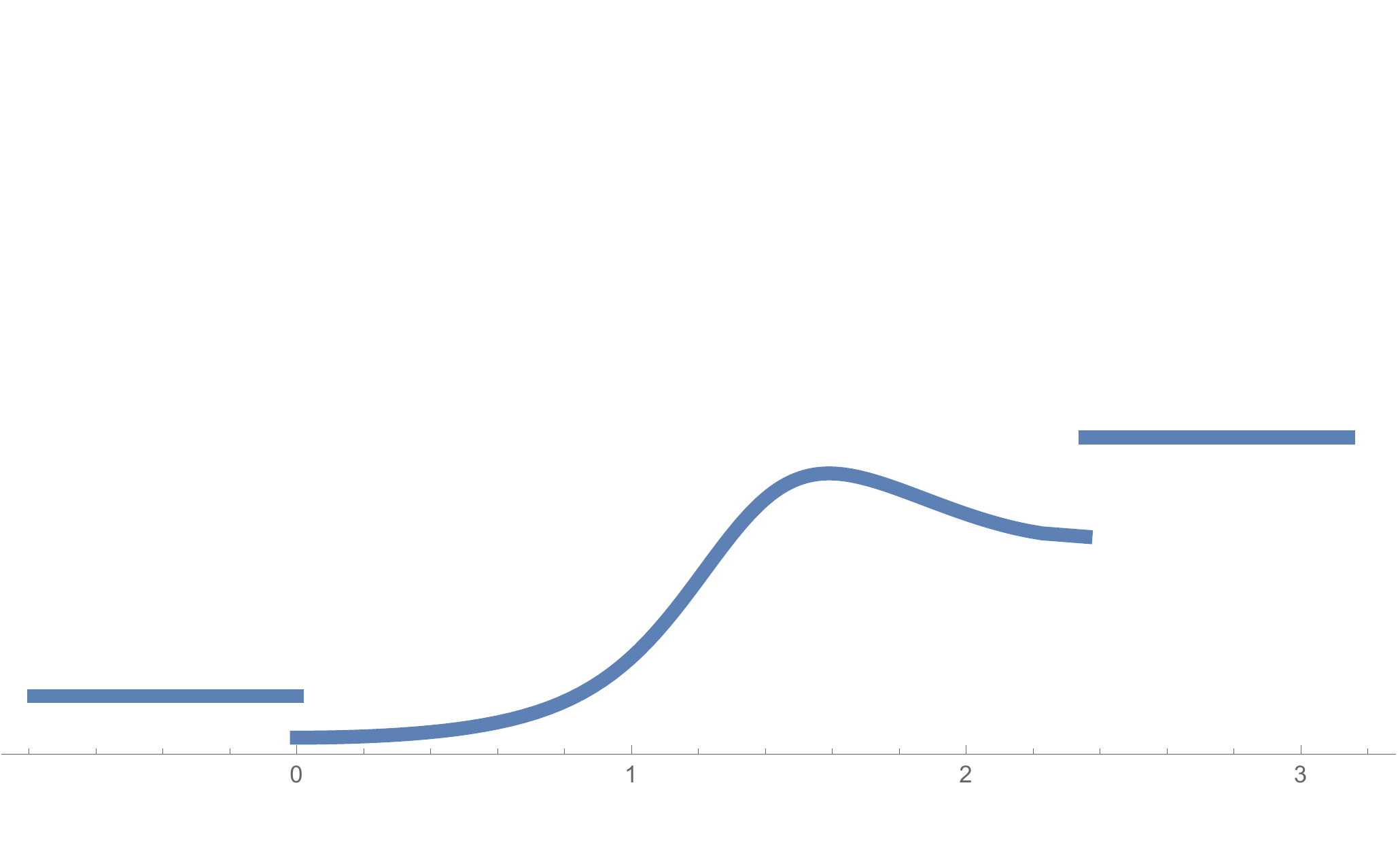}

\caption{The densities $f(s,\cdot)$ of the geodesic curves
  $\mu(s)=f(s,\cdot)\dd x $ connecting
  $\mu_0=\chi_{[-\pi/4,\pi/4]}\rmd x$ and $\mu_1
  =\chi_{[\pi/2,\pi]}\rmd x$  for $s=0.01,\,0.1,\,0.2,\,\ldots,\,0.7$.}
\label{fig:CharFunct}
\end{figure}

\subsection{Towards a characterization of all geodesic connecting two measures}
\label{ss:CharGeodConn}

Here we discuss the question of describing all geodesic curves for two
given measures $\mu_0$ and $\mu_1$. As we have seen in Section
\ref{ss:AllDiracM}, the set of all these curves can be very big, in
fact even infinite dimensional. The final aim would be to define a 
geometric tangent cone in the sense of \cite[Ch.\,12]{AmGiSa05GFMS}. 

The major tool in understanding the structure of all geodesic
connections is Corollary  \ref{co:ConeGeodesic}, which states that all
geodesic curves $s\mapsto \mu(s)$ are given as projections
$\mu(s)=\mfP \lambda(s)$ of geodesic curves $\lambda$ in $\calM_2(\CC)$. 
Thus, writing \eqref{eq:DHKbyLift} more explicitly via 
\begin{equation}
 \label{eq:OptimGamma}
\msHK(\mu_0,\mu_1)=\min\Bigset{\iint_{\CC\ti\CC} \msd_\mfC (z_0,z_1)^2
 \dd\gamma(z_0,z_1) }{ \mfP\Pi^0_\#\gamma = \mu_0,\
 \mfP\Pi^1_\#\gamma = \mu_1 }, 
\end{equation}
we can define the set of optimal plans via 
\[
\mathrm{Opt}_\msHK^\mfC(\mu_0,\mu_1):= \Bigset{ \gamma \in \calM_2(\CC{\ti}\CC)}{
\gamma\text{ is optimal in \eqref{eq:OptimGamma}} },
\]
which is a convex set. Every optimal plan $\gamma$ gives rise to
different geodesic $\lambda(s)=Z(s; \cdot,\cdot)_\#\gamma$ in
$\calM_2(\CC)$. While for different $\gamma$ the geodesics $\lambda$
are different, the same is no longer true for the projections
$\mu(\cdot)=\mfP\lambda(\cdot)$. 

The major redundancies in the set $ \mathrm{Opt}_\msHK^\mfC(\mu_0,\mu_1)$ of
optimal plans is seen through the scaling invariance given in relation
\eqref{eq:Scaling}. If $\gamma$ contains a transport of mass $m$ along
a cone geodesic connecting $[x_0,r_0]$ and $[x_1,r_1]$, then the
contribution to the projection $\mu(s)=\mfP\lambda(s)$ is equal to a
transport of mass $\vartheta^2 m$ along the cone geodesic connecting
$[x_0,r_0/\vartheta]$ and $[x_1,r_1/\vartheta]$ for all $\vartheta>0$.
Thus, we can define a normalization operator $N$ action on plans
$\gamma$, that does not change the projection. 

For this we consider the partition of  $\CC{\ti}\CC$ given by 
the sets $\mathfrak G,\:\mathfrak G_{12}', \:\mathfrak G_1', \: \mathfrak
G_2'$, and $\mathfrak G_0'$:
\begin{gather*}
  \mathfrak G :=\Big\{([x_1,r_1],[x_2,r_2])\in \CC{\ti}\CC:r_1r_2>0,\
  |x_1-x_2|<\pi/2\Big\},\\
  \mathfrak G_{12}':=\Big\{([x_1,r_1],[x_2,r_2])\in \CC{\ti}\CC:r_1r_2>0,\
  |x_1-x_2|\ge\pi/2\Big\},\\
  \mathfrak G_1':=\Big\{([x_1,r_1],\TT):\ r_1>0\Big\},\quad
  \mathfrak G_2':=\Big\{(\TT, [x_2,r_2]):\ r_2>0\Big\},\quad
  \mathfrak G_0':=\{\TT\}{\ti}\{\TT\}.
  \end{gather*}
With these sets we define the scaling function
\begin{displaymath}
  \vartheta(z_1,z_2):=
  \begin{cases}
    \Big(r_1r_2\cos(|x_1-x_2|)\Big)^{1/2}&\text{if }(z_1,z_2)\in
    \mathfrak G,\\
    r_1&\text{if }(z_1,z_2)\in \mathfrak G_{1,2}'\cup \mathfrak G_1',\\
    r_2&\text{if }(z_1,z_2)\in \mathfrak G_2',\\
    1&\text{if }z_1=z_2=\TT
  \end{cases}
\end{displaymath}
and employ the dilation map $h_\vartheta$ from \eqref{eq:1} to
generate the corresponding rescaling $N:\calM_2(\CC{\ti}\CC)\to
\calM_2(\CC{\ti}\CC)$ by $N\gamma:=(h_\vartheta)_\sharp
(\vartheta^2\gamma)\res(\CC{\ti}\CC\setminus \mathfrak G_0')$, 
where $\res$ denotes restriction of a measure. bBy the
definition of $\mfP$ and the scaling property of $\msd_\mfC$ we first
find $\mfP\Pi^j_\#\gamma =\mfP\Pi^j_\#N\gamma $ and $\iint_{\CC\ti\CC}
\msd_\mfC (z_0,z_1)^2 \dd\gamma = \iint_{\CC\ti\CC} \msd_\mfC
(z_0,z_1)^2 \dd(N\gamma)$.  Hence, for each $\gamma \in
\mathrm{Opt}_\msHK^\mfC(\mu_0,\mu_1)$ we again have $N\gamma \in
\mathrm{Opt}_\msHK^\mfC(\mu_0,\mu_1)$. Thus, we define the normalized
optimal plans via
\begin{equation}
 \label{eq:NOP}
\mathrm{NormOpt}_\msHK^\mfC(\mu_0,\mu_1):= \Bigset{ \gamma \in
  \mathrm{Opt}_\msHK^\mfC(\mu_0,\mu_1)  }{ \gamma = N\gamma },
\end{equation}
which is a much smaller but still closed and convex set. 

Using the scaling properties of the geodesics on $\CC$, which are
given via the interpolating functions $Z=(X,R)$ as $\mu(s)=\mfP
Z(s;\cdot,\cdot)_\# \gamma$ (cf.\ \eqref{eq:geodesic}), and the fact
that $\vartheta$ depends on $x_1,x_2$ only through their distance
$|x_1 {-}x_2|$, we also see that $ \gamma$ and $N\gamma$ generate the
same geodesic, viz.\
\[
\wh\mu_\gamma(s):=\mfP\big(Z(s;\cdot,\cdot)_\#)\gamma\big)= \mfP\big( 
Z(s;\cdot,\cdot)_\#(N\gamma)\big) = \wh\mu_{N\gamma}(s). 
\]

 With these preparations we are able to prove that 
there exists a unique geodesic connecting $\mu_0$ to $\mu_1$ if
$\mu_0$ is absolutely continuous with respect to the Lebesgue measure. 
 More precisely, we show that in this case  $N$ eliminates 
all redundancies: $\mathrm{NormOpt}_\msHK^\mfC(\mu_0,\mu_1)$ 
contains only one element and $N\gamma$
characterize the geodesic connection of $\mu_0$ and $\mu_1$ uniquely.

\begin{theorem} 
  \label{thm:unique-gamma}
  For every couple $\mu_0,\mu_1$ 
  in $\MMM(\Omega)$ with $\mu_0$ absolutely continuous
  w.r.t.~the Lebesgue measure, 
  there exists a unique geodesic $\mu$ connecting $\mu_0$ to $\mu_1$
  and a unique $\gamma \in
  \mathrm{NormOpt}_\msHK^\mfC(\mu_0,\mu_1)$.
  In particular, for each geodesic curve $\mu$ 
  connecting $\mu_0$ to $\mu_1$ there exists a unique $\gamma \in
 \mathrm{NormOpt}_\msHK^\mfC(\mu_0,\mu_1) $ such that $\mu=\wh
 \mu_\gamma$. 
\end{theorem} 
\begin{proof}
  By the above discussion we have just to check the uniqueness of
  $\gamma$.
  We first show that 
  any $\gamma\in \mathrm{NormOpt}_\msHK^\mfC(\mu_0,\mu_1) $ does not
  charge $\mathfrak G_{12}'$. 
  
  Since $\gamma$ is optimal,  the rescaling given by $N$ and
  \cite[Thm.~7.21]{LiMiSa14?Theory}
  yield that  $\gamma(\mathfrak G_{12}')=
  \gamma(\wt{\mathfrak G}_{12}')$
  where 
  \[
  \wt{\mathfrak
  G}_{12}'=\Big\{([x_1,1],[x_2,r_2])\in \CC{\ti}\CC:r_2>0,\
  |x_1-x_2|=\pi/2\Big\}\subset \mathfrak G_{12}'.
  \]
  Setting $\wt\mu_0:=\Pi^1_\sharp (\gamma\res \wt{\mathfrak
    G}_{12}') =\mfP\Pi^1_\sharp (\gamma\res \wt{\mathfrak
    G}_{12}')$, we have $\wt\mu_0=\eta\mu_0\le \mu_0$; in
  particular $\wt\mu_0$ is absolutely continuous with respect to
  the Lebesgue measure.  If we set $f(x):=\min_{y\in
    \mathrm{supp}(\mu_1)}|x-y|$, applying
  \cite[Thm.\,6.3(b)]{LiMiSa14?Theory} we deduce that
  \begin{displaymath}
    f(x)=\pi/2\quad \text{for }\wt\mu_0\text{-a.e.~}x\in \Omega.
  \end{displaymath}
  Applying the co-area formula to $f$, see
  \cite[Lem.\,3.2.34]{Federer69}, we have
  \begin{displaymath}
    \mathscr{L}^d\Big(\bigset{x\in \R^d }{ \pi/2 {-} \eps\le f(x)\le
      \pi/2 {+}\eps }\Big)=
    \int_{\pi/2-\eps}^{\pi/2+\eps} \!\!
    \mathscr H^{d-1}(f^{-1}(p))\dd p\quad\text{for every }\eps>0,
  \end{displaymath}
  so that passing to the limit as $\eps\downarrow 0$ we get $
  \mathscr{L}^d\big(\set{x\in \R^d}{ f(x)= \pi/2 }\big)=0$.  It
  follows that $\wt\mu_0$ is the null measure and 
  $\gamma({\mathfrak G}_{12}') =0$.

   Let us now suppose that $\gamma',\gamma''\in
   \mathrm{NormOpt}_\msHK^\mfC(\mu_0,\mu_1)$.
  %

  Combining Theorems 7.2.1 (iv) and 6.7 of \cite{LiMiSa14?Theory}
  (where we use the absolute continuity of $\mu_0$ again) we see
  that the restrictions of $\gamma'$ and $\gamma''$ to $\mathfrak
  G$ coincide.  By subtracting this common part from both of them and
  the corresponding homogeneous marginals from $\mu_0$ and $\mu_1$, it
  is not restrictive to assume that $\gamma'$ and $\gamma''$ are
  concentrated in the complement of $\mathfrak G$. By the previous
  claim, we obtain that $\gamma'$ and $\gamma''$ are concentrated on
  $\mathfrak G_1'\cup \mathfrak G_2'$.
  
  It is also easy to see that the restrictions of $\gamma'$ and
  $\gamma''$ to $\mathfrak G_i$, $i=1,2$, coincide as well:
  considering e.g.~$\mathfrak G_1'$, by construction we have that
  $\Pi^1_\sharp\gamma'=\Pi^1_\sharp\gamma''=\mu_1\otimes \delta_1$,
  whereas
  $\Pi^2_\sharp\gamma'=\Pi^2_\sharp\gamma''=\mu_1(\R^d)\delta_\TT.$ It
  follows that $\gamma'\res \mathfrak G_1'= \gamma''\res \mathfrak
  G_1'= (\mu_1\otimes \delta_1)\otimes \delta_\TT$.  A similar
  argument holds for $\mathfrak G_2'$. This proves the result.
\end{proof}

 The major point in the above proof is to show that $\gamma\big(
\wt{\mathfrak G}'_{12}\big)=0$, i.e.\ there is no transport
over the distance $\pi/2$. From Section \ref{ss:AllDiracM} we know
that in the opposite case $\mathrm{NormOpt}_\msHK^\mfC(\mu_0,\mu_1)$
can be infinite dimensional. 

We expect that, by  refining the arguments above and using the
dual characterization of $\msHK$,  it is  possible to prove
that for each geodesic curve connecting $\mu_0$ and $\mu_1$ (both
not necessarily  absolutely continuous w.r.t.~$\mathscr L^d$)
there exists a unique $\gamma \in \mathrm{NormOpt}_\msHK^\mfC
(\mu_0,\mu_1) $ such that $\mu=\wh \mu_\gamma$.  As in the
Kantorovich-Wasserstein case, the optimal plan $\gamma$ should be
uniquely determined by fixing an intermediate point $\wh
\mu_\gamma(s)$  with $s\in {]0,1[}$,  along a geodesic.  In
particular, geodesics for the Hellinger-Kantorovich distance $\msHK$
in $\mathcal M(\Omega)$  will then be  nonbranching. 
Indeed, for the rich set of geodesics connecting two Dirac masses at
distance $\pi/2$ discussed in Section \ref{ss:AllDiracM} this can be
shown by direct inspection.  We will address these questions in a
forthcoming paper.

\subsection{$\msHK$ is not semiconcave}
\label{ss:NotSemiconcave} 

The metric on a geodesic space $(Y,\msd)$ is called $K$-semiconcave,
if for all points $y_0,y_1, y_* \in Y$ and all minimal geodesic curves
$\wt y:[0,1]\to Y$ with $\wt y(i)=y_i$ for $i\in \{0,1\}$,  we have 
\begin{equation}
  \label{eq:SemiConcave}
\msd(\wt y(s),y_*)^2 \geq (1{-}s)\msd(y_0,y_*)^2 + s \msd(y_1,y_*)^2 -
Ks(1{-}s) \msd(y_0,y_1)^2. 
\end{equation}
It is well-known that the Wasserstein distance on a domain
$\Omega\subset \R^d$ is 1-semiconcave (such that
$(\calP_2(\Omega),\msW)$ is a positively curved (PC) space, see
\cite[Ch.\,12.3]{AmGiSa05GFMS}), and it is easy to check that the
Hellinger-Kakutani distance $\msH$ is 1-semiconcave. Indeed, with the
notation of Section \ref{ss:ScaRDE} we have the identity
\[
\msH(\mu^\msH(s),\mu_*)^2 = (1{-}s)\msH(\mu_0,\mu_*)^2+ s
\msH(\mu_1,\mu_*)^2 -s(1{-}s) \msH(\mu_0,\mu_1)^2
\]
for any $\mu_0,\mu_1,\mu_*\in \calM(\Omega)$, where $s\mapsto
\mu^\msH(s)\in \calM(\Omega)$ is the Hellinger geodesic from 
\eqref{eq:HellGeod}. 

In contrast, the Hellinger-Kantorovich distance is not $K$-semiconcave
for any $K$ if $\Omega$ is not one-dimensional. For the
one-dimensional case  $\Omega \subset \R$  it is shown in
\cite[Thm.\,8.9]{LiMiSa14?Theory} that $(\calM([a,b]),\msHK)$ is a PC
space, which means that 1-semiconcavity holds. The following result
shows that $(\calM(\Omega),\msHK)$ is not a PC space if $\Omega$ has
dimension $d\geq 2$.

For this case we consider a simple example, namely 
\[
\mu_0=  \delta_{x_0},\ \mu_1=\delta_{x_1}, \ \mu_*=b\delta_z,
\quad\text{with } x_0=0, \ x_1=\tfrac{\pi}2 e_1,  
\ z= \tfrac\pi4 e_1+ ye_2,
\]
where $e_1$ and $e_2$ are the first two unit vectors and $y>0$. 
As a geodesic curve we choose 
\[
\mu(s)= a(s)\delta_{\rho(s)e_1} \quad\text{with }a(s)=(1{-}s)^2+ s^2
\text{ and } \rho(s)=\arctan(s/(1{-}s)).
\]
We have $\msHK(\mu_0,\mu_1)^2=2$ and $\mu(1/2)=\frac\pi4
e_1$, and all the quantities in the semiconcavity condition
\eqref{eq:SemiConcave} can be evaluated explicitly. This yields a
lower bound for $K$, namely
\begin{align*}
K&\geq
\frac{\frac12\msHK(\mu_0,\mu_*)^2+ \frac12\msHK(\mu_1,\mu_*)^2 -
  \msHK(\mu(\frac12),\mu_*)^2 }{\frac14\msHK(\mu_0,\mu_1)^2 }  \\
&= 1 + \sqrt b \:\phi(y) \ \text{ with } \phi(y)= 1 + \sqrt{8} \cos_{\pi/2}(y) -
4\cos_{\pi/2} \sqrt{y^2+ \pi^2/16}, 
\end{align*}
where $\cos_{\pi/2}a=\cos\big(\min\{|a|,\pi/2\}\big)$. Since
$\phi(y)>0$ for $y \in {]0,\sqrt{3}\pi/4[}$ and since $b$ can be
chosen arbitrarily large, we see that there cannot
exist a finite $K$ such that $\msHK$ is $K$-semiconcave.  


\subsection*{Acknowledgments}
M.L.\ was partially supported by the
Einstein Stiftung Berlin via the ECMath/\textsc{Matheon} project SE2.
A.M.\ was partially supported by 
DFG via project C5 within CRC 1114 (Scaling cascades in complex
systems) and by the ERC AdG.\,267802 \emph{AnaMultiScale}.
G.S.\ was partially supported by PRIN10/11 grant from
MIUR for the project \emph{Calculus of Variations}. 

{\footnotesize
\newcommand{\etalchar}[1]{$^{#1}$}
\def\cprime{$'$} \def\ocirc#1{\ifmmode\setbox0=\hbox{$#1$}\dimen0=\ht0
  \advance\dimen0 by1pt\rlap{\hbox to\wd0{\hss\raise\dimen0
  \hbox{\hskip.2em$\scriptscriptstyle\circ$}\hss}}#1\else {\accent"17 #1}\fi}

}

\end{document}